\documentclass[11 pt]{amsart}
\usepackage{amssymb,amsfonts,amsthm}
\usepackage{graphicx, color}
\usepackage{bm}
\usepackage[nocompress]{cite}
\newcommand{\Z}{\mathbb{Z}}

\newtheorem{thm}{Theorem}[section]
\newtheorem{cor}[thm]{Corollary}
\newtheorem{lem}[thm]{Lemma}

\theoremstyle{definition}

\theoremstyle{remark}

\numberwithin{equation}{section}
\begin{document}
\title[NESS]{Non-equilibrium steady state for a three-mode energy cascade model}
\author[Hani, Li, Nahmod, Staffilani]{Zaher Hani, Yao Li, Andrea Nahmod, and Gigliola Staffilani}

\maketitle

\begin{abstract} Motivated by the central phenomenon of energy cascades in wave turbulence theory, we construct non-equilibrium statistical steady states (NESS), or invariant measures, for a simplified model derived from the nonlinear Schr\"odinger (NLS) equation with external forcing and dissipation. This new perspective to studying energy cascades, distinct from traditional analyses based on kinetic equations and their cascade spectra, focuses on the underlying statistical steady state that is expected to hold when the cascade spectra of wave turbulence manifest. 

In the full generality of the (infinite dimensional) nonlinear Schr\"odinger equation, constructing such invariant measures is more involved than the rigorous justification of the Kolmogorov-Zakharov (KZ) spectra, which itself remains an outstanding open question despite the recent progress on mathematical wave turbulence. Since such complexity remains far beyond the current knowledge (even for much simpler chain models), we confine our analysis to a three-mode reduced system  that captures the resonant dynamics of the NLS equation, offering a tractable framework for constructing the NESS. For this, we introduce a novel approach based on solving an elliptic Feynman-Kac equation to construct the needed Lyapunov function.
\end{abstract}

\section{Introduction}
The mechanisms of energy transfers between the modes of a nonlinear Hamiltonian system is a fundamental question for many fields of science, from fluids to oceanography, all the way to climate science. It is often studied in the context of out-of-equilibrium statistical physics, where one allows for external forcing and dissipation to act on some modes of the system. These non-Hamiltonian components are kept absent or negligible on the part of the system of interest, called the \emph{inertial range}, where the energy flux between the modes is essentially intrinsic. 

For instance, for systems that model wave-type phenomena, this theory of nonequilibrium statistical physics goes by the name of \emph{wave (or weak) turbulence theory}, in which the energy transfer happens between the different frequency (or Fourier) modes of the nonlinear wave system. This corresponds to an energy transfer between the spatial scales of the wave system. Wave turbulence theory studies this energy dynamics through a kinetic equation, called the \emph{wave kinetic equation} (WKE), which models the square of the amplitude of the Fourier modes. Through this kinetic equation, the statistics of energy cascades were (heuristically) derived by physicists starting with the work of Zakharov in the 1960s, see for example \cite{Z12, Nazarenko} for a summary. They corresponded to stationary solutions of the wave kinetic equation that are of power-law type, called the Kolmogorov-Zakharov (Z-K) cascade spectra (since they mirror Kolmogorov's spectra in hydrodynamic turbulence). For example, in the case of the nonlinear Schr\"odinger (NLS) equation
\begin{equation}\label{CNLS}
-i\partial_t v(t, x)+\Delta v (t,x)=\pm |v(t,x)|^2 v(t,x)
\end{equation}
the wave kinetic equation takes the form (under appropriate homogeneity assumptions)
\begin{equation}\label{WKE}
\partial_t n(t, k)=\mathcal C(n, n, n),
\end{equation}
where $\mathcal C$ is a cubic collision kernel. The quantity $n(k)$ describes the dynamics of the mass density $\mathbb E|\widehat v(t, k)|^2$ of \eqref{CNLS} with respect to an assigned  probability measure to the initial data. In other words,
$$
n(t, k) \approx \mathbb E|\widehat v(t, k)|^2,
$$
in an appropriate limit of large domain and small nonlinearities. The nonequilibrium cascade spectra are the following stationary solutions of \eqref{WKE} (up to constant multiplicative factors)
\begin{equation}\label{KZ}
n(k)=|k|^{-d}\qquad n(k)=|k|^{-d+2/3}.
\end{equation}
The first (respectively second) solution, called the \emph{direct (resp. inverse) cascade spectrum}, stands for a nonequilibrium steady state in which the energy of the system migrates at a constant flux from low to high frequencies (resp. high to low). From a practical standpoint, such spectra are observed in systems with a source and sink of energy present at the two ends of an inertial range of frequencies over which the cascade spectra is exhibited. 

While the rigorous derivation of the wave kinetic equation has featured substantial progress over the past few years, especially for the NLS  \cite{H1,HD1, HD2, HD3, HD4}, our rigorous understanding of the cascade spectra remains very poor. In fact, even the exact sense that those solutions satisfy the kinetic equation is not fully understood. 

In this paper, we attempt to understand the energy cascade statistics with a different perspective from the one described above. Rather than go through the kinetic equation and the cascade spectra, we would like to construct \emph{nonequilibrium statistical steady states (NESS)} or invariant measures for a (very!) reduced model coming from the nonlinear Schr\"odinger equation in the presence of forcing and dissipation. Note that the cascade spectra mentioned above can be understood to correspond to such NESS, where the observable given by the mass density $|a_k|^2$ has a power-law expectation as in \eqref{KZ}. Unfortunately, the problem of constructing such nonequilibrium invariant measures for the forced-dissipated NLS system is not any easier than that of rigorously justifying the KZ-spectra. This explains why we perform all our analysis on a three-mode reduced system for the resonant dynamics of the NLS equation. 

\subsection{Derivation of the reduced system}
The reduced model on which we will conduct our mathematical experiments appeared before (without forcing or dissipation) in the work of Colliander, Keel, Staffilani, Takaoka, and Tao \cite{CKSTT}, who attempted to understand the energy cascade problematic from yet another perspective, namely the construction of deterministic solutions of \eqref{CNLS} defined on a unit square torus, and exhibiting growth of Sobolev norms (see \cite{GG, HPTV, HPSW, CS} for more on this perspective). In Fourier space, the NLS equation takes the form:

$$
- i \partial_t a_K - 4\pi^2K^2 a_K =\sum_{(K_1,K_2,K_3) \in \mathcal{S}_{K}} a_{K_1}(t) \overline{a_{K_2}(t)} a_{K_3}(t)
$$
where $\mathcal{S}_{K} = \{ ((K_1,K_2,K_3) \in (\Z^d)^3 \; : \; K_1-
K_2 + K_3 =K \}$. A first reduction is obtained by restricting the latter set of nonlinear interactions to the subfamily of \emph{resonant interactions} defined by the following subset of $\mathcal S(K)$

$$
  \mathcal{R}(K) = \mathcal{S}(K) \cap \{ |K_{1}|^{2} - |K_{2}|^{2} +
  |K_{3}|^{2} = |K|^{2} \} \,.
  $$
 Restricting the NLS on the resonant family (i.e. restricting the sum over $\mathcal S_K$ to $\mathcal R_K$ in the above equation), we obtain the resonant
 NLS system, which approximates the dynamics of the full NLS in an appropriate sense for sufficiently small solutions. Next, a finite dimensional reduction is  performed in \cite{CKSTT} which restricts the resonant NLS system to a finite set of modes $\Lambda
\subset \mathbb{Z}^{2}$ of the form
$$
\Lambda=\Lambda_1 \cup \Lambda_2 \ldots \cup \Lambda_n\qquad \Lambda_j
\subset \Z^2 \,,
$$
such that the value of $a_K$ is uniform on each $\Lambda_j$. Roughly speaking, one can understand $\Lambda_1$ to be a set dominated (in a specific quantitative sense) by low frequencies, whereas $\Lambda_n$ is one dominated by high frequencies. Denoting by $c_j(t)$ the common value of $a_K(t)$ for $K \in \Lambda_j$, the resonant NLS system on $\Lambda$ reduces to the following ODE model (referred to as the toy model in \cite{CKSTT}):
\begin{equation}\label{ODE model}
{-i\dot{c}_j= 2(\sum_{k}|c_k|^2)c_j -|c_j|^2 c_j +2(c_{j-1}^2+c_{j+1}^2) \overline{c_j}\quad j=1, \ldots, n}.
\end{equation}
Equation \eqref{ODE model} is a Hamiltonian system with Hamiltonian
\begin{equation}\label{Hamiltonian prime}
H (c)=
\frac12(\sum_{j}|c_j|^2)^2-\frac{1}{4}\sum_{j}|c_j|^4+\frac12\sum_{j=1}^n
(c_{j-1}^2\overline{c_j}^2+\overline{c_{j-1}}^2{c_j}^2) \,.
\end{equation}
In action angle coordinates, $(I_j, \varphi_j):=(|c_j|^2, \arg c_j)$, the Hamiltonian has the form:
\begin{equation}\label{HamIphi prime}
H(I, \varphi)=\frac12(\sum_{j=1}^n I_j)^2-\frac14\sum_{j=1}^n  I_j^2 + \sum_{j=1}^n \, I_j I_{j-1}\cos 2(\varphi_{j}-\varphi_{j-1}),
\end{equation}
%
and the equations of motion take the form:
\begin{equation}\label{eqprime}
\begin{split}
\dot I_j=&-\frac{\partial H}{\partial \varphi_j}=2I_j I_{j-1}\sin2(\varphi_j-\varphi_{j-1})+2I_j I_{j+1}\sin 2(\varphi_j-\varphi_{j+1})\\
\dot \varphi_j=&\frac{\partial H}{\partial I_j}=\sum_{\ell}I_\ell-\frac{1}{2}I_j+I_{j-1}\cos 2(\varphi_j-\varphi_{j-1})+I_{j+1}\cos 2(\varphi_{j}-\varphi_{j+1}).
\end{split}
\end{equation}

The seminal paper \cite{CKSTT} shows
that there exists an ``energy cascade'' solution for (NLS) whose energy is
concentrated at time $t=0$ on the low frequency region $\Lambda_1$, and cascades in time to become concentrated at time $t = T$ on the high frequency region $\Lambda_N$. This energy cascade for NLS is exhibited by proving energy diffusion for the ODE model \eqref{ODE model}, namely by constructing a solution thereof such that at time $t=0$, the energy is mostly concentrated on the first mode (with the others being very small), and at time $t=T$ it is concentrated on the last mode (with the remaining modes being very small). 

\medskip

In this work, we take a statistical approach to understanding the energy cascade for the ODE model defined in \eqref{ODE model}. Specifically, we aim to construct nonequilibrium invariant measures for this model, analogous to those developed for heat conduction systems, which have been extensively studied in the literature \cite{lepri2003thermal, eckmann1999non, eckmann2006nonequilibrium, rey2002exponential}. As noted above, these nonequilibrium measures are motivated by and are expected to be closely connected to the Kolmogorov-Zakharov cascade spectra, as both describe steady states characterized by a flux of energy across different modes. It is a fruitful direction for future research, which we do not pursue in this work, to investigate such connections more deeply (see \cite{Bedrossian23} for some discussions in this direction).


It is well known that to exhibit such nonequilibrium invariant measures, one must inject energy into the lowest mode and extract it from the highest mode. More precisely, we couple the lowest and highest modes to stochastic forcing terms and introduce damping terms to absorb energy from the system. This setup closely resembles the stochastic heat baths used in many microscopic models of heat conduction \cite{eckmann1999non, eckmann1999entropy}. Our objective is to establish the existence, uniqueness, and ergodicity of the nonequilibrium steady state (NESS) of the modified system, which is commonly viewed as the first step in studying the properties of the NESS. It is important to note that the existence of a unique NESS is not guaranteed a priori; significant analysis is required to understand mechanisms that prevent the system from ``overheating" or ``freezing," both of which are crucial to establishing the desired properties of NESS. In this paper, we only consider the case when $n = 3$, where there's only one mode (the middle one)
that is not connected to a heat bath. We believe that longer chains can be studied, but given how challenging it is to deal with the three-mode case, we do not attempt this generalization here. We also remark that the NESS in energy cascade models remains poorly understood; to date, even for a two-mode system (with a different forcing setup), only a formal analysis has been carried out \cite{deville2007nonequilibrium}.

\subsection{Statement of the result}

Setting $\theta_{1} =
2(\varphi_{1} - \varphi_{2})$ and $\theta_{3} = 2 (\varphi_{3} -
\varphi_{3})$, equation \eqref{eqprime} becomes a closed equation in $(I_1, I_2, I_3, \theta_1, \theta_2)$ when $n = 3$.
We will connect the first and third modes to heat baths and
damping forces, which converts the equation into the following 
stochastic differential equation:
\begin{align}
  \label{bigsystem}
  \mathrm{d}I_{1}& = [2 I_{1} I_{2} (\sin \theta_{1} -
  \gamma) + \gamma(T_{1} - I_{1}^{3}) ] \mathrm{d}t + \sqrt{\gamma
  T_{1} I_{1}/2} \mathrm{d}B^{(1)}_{t} \\\nonumber
  \mathrm{d}I_{2} &= - 2 I_{2} (I_{1} \sin \theta_{1} + I_{3} \sin
  \theta_{3}) \mathrm{d}t\\\nonumber
  \mathrm{d}I_{3} &= [2 I_{2}I_{3}(\sin \theta_{3} - \gamma) +
  \gamma(T_{3} - I_{3}^{3})] \mathrm{d}t + \sqrt{\gamma T_{3}I_{3}/2}
  \mathrm{d}B^{(2)}t\\\nonumber
  \mathrm{d} \theta_{1}& =  I_{2}(1 + 2 \cos \theta_{1}) - I_{1} (1 +
  2 \cos \theta_{1}) - 2 I_{3} \cos \theta_{3} \mathrm{d}t +
  \sqrt{\gamma} g(I_{2}, \theta_{1}) \mathrm{d}B^{(3)}_{t}\\\nonumber
  \mathrm{d} \theta
  _{3} &= I_{2} (1 + 2 \cos \theta_{3}) - I_{3} (1 + 2 \cos
  \theta_{3}) - 2 I_{1}\cos \theta_{1} \mathrm{d}t + \sqrt{\gamma } g(I_{2}, \theta_{3})\mathrm{d}B^{(4)}_{t}\nonumber
\end{align}
where $B^{(i)}_t, i = 1, 2, 3, 4$ are independent Wiener processes, $0 < \gamma \ll 1$ is the damping strength, $T_{1}$ and $T_{3}$ are
the boundary temperatures, and the weight function $g$ will be
described in assumption {\bf (H)} later on (see Section \ref{highenergy}). Note that when $\gamma = 0$, the stochastic system is equal to the deterministic system \eqref{eqprime}. 
The exact choices of the forcing and dissipation and the choice of $g$ will be justified later in detail as we discuss the challenges involved in proving a result of this type (see Section 3). 
We allowed ourselves the liberty of designing this particular type of external forcing and dissipation at the modes $I_1, I_3$ (and their associated angles $\theta_1, \theta_3$), since this simplification does not alter the main challenge in proving the existence of a NESS for a system like \eqref{bigsystem}, which is to show that the internal modes that are not connected to the heat baths ($I_2$ in our case) do not undergo ``overheating" (energy being stuck in $I_2$) or ``freezing" (energy being blocked at $I_2$).

Roughly speaking, our main result states that when the temperature of two
heat baths are significantly different, equation \eqref{bigsystem}
admits a unique nonequilibrium steady state. In addition, we also prove
that the speed of convergence of the density towards the steady state is at least power-type in time.  
More precisely, denote by 
$\mathbf{x} = (I_{1}, I_{2}, I_{3}, \theta_{1}, \theta_{3}) \in
\mathbb{R}^{3}_{+} \times \mathbb{T}^{2} := \Omega$ be a state of the
system. Denote the Borel $\sigma$-algebra on $\Omega$ by $\mathcal{B}(\Omega)$ and define the transition kernels $P^{t}(\mathbf{x},
\cdot): \Omega \times \mathcal{B}(\Omega) \mapsto [0,1]$ by the relation
$$
  P^{t}(\mathbf{x}, A) = \mathbb{P}[\mathbf{x}_{t} \in A \,|\, \mathbf{x}_{0} = \mathbf{x}] \,,
$$
where $A\in \mathcal B(\Omega)$. Then the following holds:

\medskip

{\bf Theorem 1. } 
Let $\beta_{0} > 1$, $\gamma > 0$, and $0 < T_{1} < \infty$ be constants. Assume $T_3$ is sufficiently large, then there exist a function $g(I,\theta)$ and a Lyapunov function $\mathcal V$ satisfying $\mathcal V(\mathbf{x}) \to \infty$ as $|\mathbf{x}|\to \infty$, such that 
\begin{itemize}
  \item Equation \eqref{bigsystem} admits an invariant probability measure $\pi$
    satisfying
    
$$
  \int_{\Omega} \mathcal{V}( \mathbf{x}) \pi ( \mathrm{d} \mathbf{x})
  < \infty
  $$
\item There exists  constants $C_1, C_2, C_3$ such that for every $\mathbf{x}, \mathbf{y} \in \Omega$, 
$$
  \| P^{t}(\mathbf{x}, \cdot) - P^{t}( \mathbf{y}, \cdot) \|_{TV} \leq
  C_1( \mathcal{V}( \mathbf{x}) + \mathcal{V}( \mathbf{y}))(t+1)^{-\beta_{0}}
  $$
  
  and 
    $$
    \| P^{t}(\mathbf{x}, \cdot) - \pi\|_{TV} \leq
  C \mathcal{V}( \mathbf{x})(t+1)^{-\beta_{0}} + C (t+1)^{1 - \beta_{0} }
    $$
    where $\|\cdot\|_{TV}$ is the total variation norm.
\end{itemize}

\bigskip

\subsection{Remarks on the result} 

\begin{itemize}

\item  It is well known that the Gibbs measure is an equilibrium state when the system is not connected to a heat bath. Under our choice of forcing and dissipation, however, the Gibbs measure is no longer an invariant probability measure, even when the two heat baths are at the same temperature— in contrast to classical heat conduction models \cite{eckmann1999entropy, eckmann1999non, eckmann2006nonequilibrium}. This is a technical choice made to enable a rigorous proof. On the other hand, unlike microscopic heat conduction, it is less clear what type of heat bath best captures the physics of an energy cascade involving infinitely many modes. Moreover, recent results show that in the presence of thermal reservoirs, the Gibbs measure generally fails to be invariant for a broad class of systems \cite{carlen2025stationary}. Even when the invariance of the Gibbs measure can be established by explicit calculation, uniqueness and ergodicity typically require additional justification. In contrast, a nonequilibrium steady state (NESS) generally does not admit an explicit form, making the analysis of its existence, uniqueness, and ergodicity a serious mathematical challenge.

\item While there is an extensive body of work on the ergodicity of NESS for nonlinear chain models and nonlinear networks modeling microscopic heat conduction, existing techniques do not apply to our energy cascade model \eqref{bigsystem}. In classical heat conduction, energy naturally flows from the high-temperature bath to the low-temperature one. In our energy cascade model, however, the mechanism of energy transfer between modes is much less explicit. It is not a priori clear what prevents energy from accumulating or blowing up at the interior mode $I_2$, or from depleting entirely and thereby disrupting energy transport. A careful examination of system \eqref{bigsystem} reveals that energy transfer to or from the middle of the chain, when it has too little or too much energy, respectively, depends crucially on the phase variables being “correct.” This sensitivity complicates the construction of a Lyapunov function, as natural candidates may behave improperly at the “wrong” phase, causing the function to move to an undesirable direction.
 
\item To overcome the aforementioned problem that phases have to be properly aligned to facilitate the energy transfer between modes, we develop a novel Feynman-Kac-Lyapunov method. This approach is partially motivated by the idea of ``stabilization by noise" proposed in \cite{herzog2015noise}. It introduces a pre-factor to the Lyapunov function, determined by solving a Feynman-Kac equation. After incorporating this pre-factor, the Lyapunov function assigns larger values to ``wrong phases", so that during the evolution, even if the energy temporarily moves in an undesirable direction, the overall value of the Lyapunov function decreases as the phase variables move toward the correct configuration. Prior to the development of our method, the only available approach was to study the change of a Lyapunov function over long time intervals, as described in \cite{hening2018coexistence, bedrossian2024stationary}, relying on the idea that, eventually, energy would drift in the correct direction once the phases align. However, when the system requires a Lyapunov function that behaves differently in different regions to control the full dynamics, significant technical difficulties arise. Our Feynman-Kac-Lyapunov method provides a more elegant and flexible tool to control such dynamics and to establish stochastic stability for high-dimensional stochastic differential equations. 

\item High- and infinite-dimensional stochastic differential equations (SDEs) arising from the study of nonlinear partial differential equations have attracted significant attention in the past decade. Notable examples include the stochastic Navier–Stokes equations \cite{hairer2006ergodicity,bedrossian2022almost}, degenerate energy transfer SDEs \cite{bedrossian2024stationary, bedrossian2024existence}, and shell models \cite{mattingly2007simple}. A common difficulty in these problems is managing the nonlinear interactions among many variables. The Feynman-Kac-Lyapunov function developed in this paper provides a valuable tool for understanding how a small subsystem can stabilize the entire system. Thus, it holds considerable potential for application to a broader class of problems.

\end{itemize}

\subsection{Future Perspectives}

The results and techniques developed in this paper open several directions for future study, of which we highlight the following:

\begin{itemize}
\item In order to complete the proof, we simplified the heat bath and its coupling with the 3-mode chain, resulting, as mentioned above,  in a system where the Gibbs measure is not invariant even when the two baths have the same temperature. Since the main difficulty lies in understanding stabilization in the out-of-equilibrium setting, this simplification does not change the essential challenges addressed in this work. Nevertheless, it would be very interesting to study what heat baths best describe the physics of the energy cascade, as well as the existence of a NESS under more ``natural" heat baths and couplings. We do not expect the outcome to change qualitatively, and we intend to pursue this direction in future work. This investigation may require numerically verifiable assumptions, such as the existence of an invariant probability measure for a reduced ``broken" system when $I_2 = 0$. Furthermore, although we only prove here that the convergence rate to the NESS is at least polynomial, we expect it to be exponential.

\item The original motivation of this study is to understand the energy cascade phenomenon in nonlinear wave systems and to offer a new perspective on it via NESS. It would be highly interesting to draw more explicit connections between this viewpoint and the kinetic perspective provided by wave kinetic theory. As mentioned above,  the wave kinetic equation admits formal stationary solutions corresponding to nonequilibrium energy cascades, known as the Kolmogorov-Zakharov spectra. The rigorous understanding of these solutions, and their true implications for energy cascades, remains very incomplete (see \cite{CDG24, GLM24} for some recent progress). From this perspective, our work can be seen as one of the first rigorous results shedding light on the statistical description of cascades and fluxes in nonlinear wave systems. Further studies of the existence, statistical properties, and energy fluxes of NESS for longer chain models could enable detailed comparisons with the scaling laws predicted by Kolmogorov-Zakharov theory. Even a numerical investigation of these questions would be of considerable interest.

\end{itemize}

\subsection{Organization of the paper} In Section 2, we review some important notions and results stochastic differential equations. In Section 3, we explain our novel strategy to constructing the Lyapunov functions and the invariant measure for our stochastic differential equation. In Section 4, we construct the component of the Lyapunov function on high energy regions, and in Section 5, we deal with the low energy ones. 

\subsection{Acknowledgements} Over the past decade, many colleagues have contributed through discussions and feedback on various aspects of this project. In particular, we would like to thank Luc Rey-Bellet, Jonathan Mattingly, and Zhongwei Shen for their valuable suggestions and insights.

The first, third, and forth authors were partially supported by a Simons Simons Foundation Collaboration Grant on Wave Turbulence. Also, the first author was partially supported by NSF grant DMS-235024. The second author was partially supported by NSF DMS-2108628. The third author was partially supported by NSF DMS-2052740, NSF DMS-2400036. The fourth author was partially supported by NSF DMS-2052651, NSF DMS-2306378.

\section{Probability preliminary}
\bigskip
In this section we collect a series of known results about stochastic differential equations that will be used in this paper. Consider a stochastic differential equation 
\begin{equation}
  \label{SDE}
  \mathrm{d} X_{t} = b(X_{t}) \mathrm{d}t + \sigma (X_{t}) \mathrm{d} B_{t} \,,
\end{equation}
where $b: \mathbb{R}^d \rightarrow \mathbb{R}^d$ is a vector field in $\mathbb{R}^d$, $\sigma : \mathbb{R}^d \rightarrow \mathbb{R}^{d\times n}$ is a matrix-valued function, and $B_t$ is the Wiener process in $\mathbb{R}^n$. The infinitesimal generator of equation \eqref{SDE} is
$$
  \mathcal{L} = \sum_{i,j = 1}^{d} a_{ij}(x) \frac{\partial ^{2}}{\partial
    x_{i} \partial x_{j}} + \sum_{i = 1}^{d} b_{i}
  \frac{\partial}{\partial x_{i}}  \,,
$$

We will need the following results regarding equation \eqref{SDE}. 

\medskip
{\bf A. Stochastic stability.}
Denote the Borel $\sigma$-algebra on $\mathbb{R}^{d}$ by $\mathcal{B}(\mathbb{R}^{d})$. Let $P^{t}(x,
\cdot): \mathbb{R}^{d}\times \mathcal{B}(\mathbb{R}^{d}) \mapsto \mathbb{R}$ be the transition kernel of equation \eqref{SDE} such that
$$
  P^{t}(x, A) = \mathbb{P}[X_{t} \in A \,|\, X_{0} = x] \,,
$$
where $t > 0$, $x \in \mathbb{R}^d$, and $A \in \mathcal{B}(\mathbb{R}^d)$. Assume the following assumptions hold:

{\bf (A1)} There exists a {\it small set} $C \in \mathcal{B}(\mathbb{R}^{d})$, constants $T > 0$, $\eta > 0$, and a probability
measure $\nu$ supported by $C$, such that
$$
  P^{T}(x, A) \geq \eta \nu (A) \quad \forall A \in \mathcal{B}(
  \mathbb{R}^{d}), \quad x \in C \,.
$$

{\bf (A2G)} There exists a {\it Lyapunov function} $V: \mathbb{R}^{d}
\rightarrow [1, \infty)$ that takes the value
$\infty$ as $x \rightarrow \infty$ and constants $a > 0$, $d <
\infty$, such that
$$
  \mathcal{L} V \leq -a V + d \,.
  $$

We refer to \cite{meyn2012markov} for additional discussion of the small set and the drift condition. The following known result in \cite{mattingly2002ergodicity} is well known.

\begin{thm}
  \label{HMSthm25}
  Let $X_{t}$ be a stochastic differential equation that satisfies
  assumptions {\bf (A2G)} and {\bf (A1)} with set $C$ given by
$$
  C = \left\{ x \,|\, V(x) \leq \frac{2d}{a( \gamma - e^{-a T})}\right\} 
  $$
  for some $\gamma \in (e^{-a T/2}, 1)$. Let $\hat{X}_{n} = X_{nT}$ be an 
  embedded chain of $X_{t}$. Then there exists a unique
  invariant probability measure $\pi$ and constants $r(\gamma) \in (0,
  1)$, $\kappa(\gamma) \in (0, \infty)$, such that for all measurable
  function $f$ with $|f| \leq V$, we have
  
$$
  |\mathbb{E}_{X_{0}} f(\hat{X}_{n}) - \pi (f) | \leq \kappa r^{n} V(\hat{X}_{0}) \,.
$$
\end{thm}

Note that according to Lemma 2.3 in \cite{mattingly2002ergodicity}, Assumption {\bf (A1)} is implied by  the following more
verifiable assumption {\bf (A1')}.

\medskip
{\bf (A1')}: There exists a compact set $C \in
\mathcal{B}(\mathbb{R}^{d})$ such that for some $y$ in the interior of
$C$, there is, for any $\delta > 0$, a $t_{1} = t_{1}(\delta)$ such
that
$$
  P^{t_{1}}(x, \mathcal{B}_{\delta}(y) ) > 0 \qquad \forall x \in C \,,
$$
where $\mathcal{B}_{\delta}(y)$ is the open ball of radius $\delta$
centered at $y$. In addition, the transition kernel possesses a
jointly continuous density function on $C$ such that
$$
  P^{t}(x, A) = \int_{A} p^{t}(x,y) \mathrm{d}y \quad \forall x \in C,
  A \in \mathcal{B}(\mathbb{R}^{d}) \cap \mathcal{B}(C) \,.
$$

Assumption {\bf (A2G)} gives the criterion of geometric ergodicity. A
power-law ergodicity requires a weaker condition of the Lyapunov
function, described in the following Assumption {\bf (A2P)}: 

\medskip
{\bf (A2P)} There exists a Lyapunov function $V: \mathbb{R}^d \rightarrow [1, \infty)$ with precompact sublevel sets such that  
$$
\mathcal{L} V \leq K - V^\alpha
$$
for a constant $0<\alpha < 1$.

We have the following known result modified from Theorem 4.1 of \cite{hairer2010convergence} by replacing $\phi(V)$ in the original theorem by $V^\alpha$.

\begin{thm}
\label{hairernote41}
    Let $X_t$ be a stochastic differential equation that satisfies assumption {\bf (A2P)}. Let $\| \cdot \|_{TV}$ denote the total variation norm. Assume for any constant $C > 0$, the sublevel set $\{ V \leq C\}$ satisfies assumption {\bf (A1)}. Then 
    \begin{itemize}
        \item There exists an invariant probability measure $\pi$ satisfying 
        $$
        \int V(x)^\alpha \pi (\mathrm{d}x) < \infty
        $$
        \item There exists a constant $C$ such that for every $x, y \in \mathbb{R}^d$,
        $$
        \|P^{t}(x, \cdot) - P^{t}(y, \cdot)\|_{TV} \leq C (V(x) + V(y)
        ) (t+1)^{-1/(1-\alpha)} \,.
        $$
      \item There exists a constant $C$ such that for every $x \in \mathbb{R}^{d}$,
        $$
                  \| P^{t}(x, \cdot) - \pi \|_{TV} \leq C V(x)
                  (t+1)^{-1/(1 - \alpha)} + C (t+1)^{-\alpha/(1-
                    \alpha)}  \,.
                  $$
    \end{itemize}
  \end{thm}
\bigskip

  {\bf B. Multiplicative ergodic theorem.}

The ordinary stochastic stability of a Markov chain is about the
``additive ergodicity''. In contrast, the multiplicative ergodicity theorem for Markov processes estimates
the integral
$$
  \exp\left ( \int_{0}^{t} F( X_{s}) \mathrm{d}s\right )
$$
of an observable $F$. Obviously multiplicative ergodicity is a
stronger stochastic stability result that requires stronger
conditions. We have the following result from \cite{kontoyiannis2003spectral} for continuous-time
Markov processes.

\medskip

\begin{thm}
 \label{MET}
Let $X_{t}$ be a continuous Markov process on $\Omega$. Assume
\begin{itemize}
\item[(1)] The Markov process $X_{t}$ is geometrically ergodic with an
  invariant probability measure $\pi$ and Lyapunov function $V: X\rightarrow [1, \infty )$ such that $\pi(
  V^{2}) < \infty$, and 
  \item[(2)] There exists a measurable function $F: \Omega \rightarrow [-1, 1]$ has
    zero mean $\pi(F) = 0$ and nontrivial asymptotic variance
    
$$
  \sigma^{2} = \lim_{t \rightarrow \infty} \mbox{Var}_{x}\left[
  \frac{\int_{0}^{t} F( X_{s}) \mathrm{d}s}{\sqrt{t}}\right] > 0 \,.
  $$
Let $\mathcal{L}$ be the infinitesimal generator of $X_{t}$. Then for
any complex number $\alpha = a + i \omega$ with sufficiently small
$|a|$ and $|\omega|$, the operator $\mathcal{L} + \alpha F$ admits a
maximal and isolated eigenvalue $\Lambda(\alpha)$ and an eigenfunction
$\hat{f}_{\alpha}(x)$, such that
$$
  \left| \mathbb{E}_{x}\left[ \exp \left ( \alpha \int_{0}^{t} F( X_{s})
    \mathrm{d}s  - t \Lambda (\alpha) \right )\right] - \hat{f}_{\alpha}(x)
  \right| \leq B_{0} |\alpha| e^{-b_{0}t} V(x)
  $$
  for constants $b_{0} > 0$ and $B_{0} < \infty$. 
\end{itemize}
\end{thm}

For discrete Markov processes, under stronger drift conditions, the
eigenvalue $\Lambda (\alpha)$ is analytical with respect to
$\alpha$. In other words, we have $\Lambda (\alpha) = O( \alpha^{2})$
when $|\alpha| \ll 1$. See \cite{kontoyiannis2005large} for more details. However, to the best of our knowledge, the
corresponding result for continuous-time Markov processes has not been made 
available in the literature. Instead, we can only use the eigenvalue perturbation theory in \cite{kato2013perturbation} to show that $\Lambda(\alpha)$ is $o(\alpha)$ small, hence $\lim_{\alpha \rightarrow 0} \Lambda(\alpha) = 0$.

\bigskip

{\bf C. Stochastic representation of PDE solutions.}
Consider an elliptic operator in a bounded domain $D \subset
\mathbb{R}^{d}$ of the form
$$
  L_{D} := \sum_{i,j = 1}^{d} a_{ij}(x) \frac{\partial ^{2}}{\partial
    x_{i} \partial x_{j}} + \sum_{i = 1}^{d} b_{i}
  \frac{\partial}{\partial x_{i}} + c(x) \,,
  $$
  where the matrix $A = \{ a_{ij}\}$ is non-degenerate and there exist $\Lambda_{1}, \Lambda_{2}>0$ such that 
  
$$
  \Lambda_{1} |\eta|^{2} \leq \sum_{i,j = 1}^{d} a_{ij}(x)
  \eta_{i}\eta_{j} \leq \Lambda_{2} |\eta|^{2}, \quad \forall x \in D,
  \eta \in \mathbb{R}^{d} \,,
  $$
  all coefficients of $L_D$ are Lipschitz continuous and bounded in $D$,
  and $\partial D \in C^{2, \beta}$. Note that $\partial D \in C^{2, \beta}$ already implies that the exterior cone condition is satisfied by every $y \in \partial D$. Hence, $D$ is a regular bounded domain in the sense of \cite{freidlin1985functional}. 

  In addition, assume that the matrix $A = \{ a_{ij}\}$ can be represented in
  the form $A = \sigma \sigma ^{T}$ for a matrix-valued function
  $\sigma$ with Lipschitz continuous elements.
  Then if $c = 0$, $L_{D}$ is equal to the
  infinitesimal generator of the stochastic differential equation \eqref{SDE}.


  Consider the following Cauchy-Dirichlet problem. 
  
\begin{align}
  \label{cauchy}
  & L_{D} u(x)= f(x), \quad  x \in D\\\nonumber
  & u(x_{0}) = \psi(x_{0}), \quad  x_{0} \in \partial D \nonumber\,.
\end{align}
It is relatively well-known that if equation \eqref{cauchy} admits a
solution, then it must have a stochastic representation up to the
first exit time
$$
  \tau_{x}^{D} = \inf\{ t > 0 \,|\, X_{0} \ x, X_{t} \in \partial D \}
$$
The following Theorem comes from \cite{freidlin2012random}.

\begin{thm}
  \label{freidlinPDE}
  Suppose
$$
  \mathbb{E}\left[\int_{0}^{\tau_{x}^{D}} \exp\left ( \int_{0}^{t}c(X_{s})
    \mathrm{d}s\right) \mathrm{d}t \right] < \infty 
  $$
  for all $X_{0} = x \in D$. If $u(x)$ is a bounded $C^{2}$ solution of
  equation \eqref{cauchy}, then
$$
  u(x) = - \mathbb{E}\left[\int_{0}^{\tau_{x}^{D}} f(X_{t})\exp\left ( \int_{0}^{t}c(X_{s})
    \mathrm{d}s\right) \mathrm{d}t\right] + \mathbb{E}\left[
  \psi(X_{\tau_{x}^{D}})\exp\left ( \int_{0}^{\tau_{x}^{D}}c(X_{s})
    \mathrm{d}s\right)\right] \,.
$$
\end{thm}

However, Theorem \ref{freidlinPDE} does not say whether the equation
\eqref{cauchy} admits a solution. The existence of solution comes from 
heavier machinery. Denote the principal eigenvalue of $L_D$ on $D$ by
$\lambda_{1}(L_{D}, D)$. By \cite{berestycki1994principal}, we have
$$
  \lambda_{1}(L_{D}, D) = \sup\{ \lambda \,|\, \exists \phi > 0 \mbox{ in }
  D \mbox{ satisfying } (L_{D} + \lambda) \phi \leq 0 \} \,.
$$ 
We have the well-known result from \cite{berestycki1994principal} (Theorem 1.2 of \cite{berestycki1994principal}).

\begin{thm}
  \label{nirenberg}
  Assume $\lambda_{1}(L_{D}, D) >0$. Then given $f \in L^{d}(D)$, there is
  a unique solution $u \in W^{2, d}_{loc}$ for \eqref{cauchy} with zero boundary
  condition $\psi = 0$. In addition, $|u|$ is bounded.
\end{thm}

It is known that the principal eigenvalue $\lambda_{1}(L_D, D)$ has a stochastic
representation. The following theorem is given by Theorem 6.2 of
\cite{liu2002strongly}.

\begin{thm}
  \label{liu}
  Assume $L_{D}$ and $D$ satisfy the regularity condition described
  above. We have
  
$$
  \lambda_{1}( L_{D}, D)= - \lim_{t \rightarrow \infty} \frac{1}{t} \ln \left (
    \sup_{x \in D} \mathbb{E}_{x}\left [  \mathbf{1}_{\{t <
        \tau_{x}^{D}\}} \exp \left ( \int_{0}^{t} c(X_{s}) \mathrm{d}s \right ) \right ]
  \right ) \,.
$$
\end{thm}

Therefore, summarizing Theorem \ref{MET}, Theorem \ref{freidlinPDE}, Theorem
\ref{nirenberg}, Theorem \ref{liu}, and eigenvalue perturbation theory
in \cite{kato2013perturbation}, the following theorem gives the criterion that the
Feynman-Kac equation \eqref{cauchy} admits a strictly positive
solution. 

\begin{thm}
  \label{thm28}
Assume that the stochastic differential equation \eqref{SDE}, denoted by
$X_{t}$, has $C^3$ coefficients and satisfies condition (1) of Theorem \ref{MET} with respect to
an invariant probability measure $\pi$ and a Lyapunov function
$V$. Let $D\subset \mathbb{R}^{n}$ be a bounded domain with regular
$C^{4}$ boundary. Further assume that $c_{0}(x)$ is a $C^3$ function such that 
\begin{itemize}
    \item[(a)] $\pi( \mathbf{1}_D c_0) < 0$, and
    \item[(b)] Function $F(x) := c_0(x)\mathbf{1}_D(x) - \pi( \mathbf{1}_D c_0)$ satisfies condition (2) of Theorem \ref{MET}.
\end{itemize}
Then
there exists a constant $\alpha_{0} > 0$, such that for any $0 <
\alpha < \alpha_{0}$ and any bounded $C^{2}$ functions $f(x) < 0$ and
$C^{2}$ function $\psi(x) > 0$, the Cauchy-Dirichlet problem 
\begin{align}
  \label{cauchySDE}
  & \mathcal{L} u + \alpha c_{0}  u = f(x), \quad  x \in D\\\nonumber
  & u(x_{0}) = \psi(x_{0}), \quad  x_{0} \in \partial D \nonumber\,.
\end{align}
admits a strictly positive solution $u$.
\end{thm}

The proof of Theorem \eqref{thm28} is given in the appendix. 

\bigskip

{\bf D. Hypoellipticity.}

Assumption {\bf (A1')} requires the transition kernel to possess a
jointly continuous density function on $C$. Later we would need to claim that in our context this assumption is satisfied. To do so we will use   H\"ormander's Theorem \ref{hormander} below.  
To state this  theorem we still consider a stochastic differential equation as in 
\eqref{SDE} with the assumption that $b$ is a smooth vector field and $\sigma$ is a $n\times d$ matrix
smoothly dependent on $x$ for $n \leq d$. Define vector fields $V_{0}, V_{1}, \cdots, V_{n}$ such that
\begin{equation}
  \label{Vks}
  V_{k} = \left \{
\begin{array}{lcl}
 \sum_{i = 1}^{N} \sigma_{ik}  \partial_{x_{i}}&if  & 1 \leq k
                                                              \leq n\\
  \sum_{i = 1}^{N} \left ( b_{i} -\frac{1}{2} \sum_{k = 1}^{n} (V_{k} \sigma_{i,k})\right )\partial_{x_{i}}
\end{array}
  \right . \,.
\end{equation}

For two vector fields $V_{1}$ and $V_{2}$, their Lie bracket is defined by
$$
  [V_{1}, V_{2}](x) = DV_{1}(x) V_{2}(x) - DV_{2}(x) V_{1}(x) \,,
$$
where $DV_{1}$ and $DV_{2}$ are Jacobian matrices. Further, for vector
fields $\{V_{1}, V_{2},\cdots, V_{n}\}$, we define the Lie algebra $\mathrm{Lie}_{x}(V_{1},
\cdots, V_{n})$ generated by $\{V_{1}, V_{2},\cdots, V_{n}\}$ as the
vector space spanned by
\begin{align*}
  & V_{i_{0}}(x) \\
  &[ V_{i_{0}}(x), V_{i_{1}}(x)] \\
  &\left [ [ V_{i_{0}}, V_{i_{1}} ] , V_{i_{2}}\right ]\\
  &\quad \vdots 
\end{align*}
for $0 \leq i_{0}, i_{1}, \cdots \leq n$.

The following H\"ormander's theorem comes from Theorem 7.4.20 of
\cite{stroock2008partial}.

\begin{thm}
  \label{hormander}
  Define vector fields $V_{0}, V_{1}, \cdots, V_{n}$ from $\sigma$ and
  $b$ as described in equation \eqref{Vks}. If $D \subset \mathbb{R}^{d}$ is an
  open subset such that
  
$$
  \mathrm{Lie}_{x}\left ( [V_{0}, V_{1}], \cdots, [V_{0}, V_{m}],
    V_{1}, \cdots, V_{n}\right ) = \mathbb{R}^{d}
  $$
  for all $x \in D$, then equation \eqref{SDE} admits a smooth
  transition density function $p(t,x,y)$ such that
  
$$
  P^{t}(x, A) = \int_{A} p(t,x,y) \mathrm{d}y
  $$
  for all measurable $A \subset D$.
\end{thm}

\section{Proof strategy}
\label{mainresult}

\subsection{Main idea: when Feynman-Kac meets Lyapunov}

A fundamental tool for establishing stochastic stability—meaning the existence, uniqueness, and convergence rate of an invariant probability measure for a stochastic process—is the Lyapunov function method. Drift conditions such as {\bf (A2G)} and {\bf (A2P)} in the previous section are commonly used to prove stochastic stability of Markov processes. However, constructing a Lyapunov function in higher-dimensional settings remains a significant challenge. For example, in this
paper, the intricate interaction between energy components  $I_{1}, I_{2}, I_{3}$ and phases
$\theta_{1}, \theta_{3}$ makes an explicit Lyapunov function construction nearly impossible.

Our approach uses the Feynman-Kac representation to construct a pre-factor for the Lyapunov function. To illustrate our strategy, consider the following toy model with one-sided coupled $X$ and $Y$ terms
\begin{align}
  \label{toymodel}
  \mathrm{d} X_{t}& = X_{t} c(Y_{t}) \mathrm{d}t \\\nonumber
  \mathrm{d} Y_{t}& = g(Y_{t}) \mathrm{d} t + \sigma (Y_{t})
                    \mathrm{d} B_{t} \,,
\end{align}
where $X_{t} \in \mathbb{R}$ is a scalar variable, $Y_{t} \in \mathbb{R}^{d}$, $c$ is a
continuous function on $\mathbb{R}^{d}$, $g: \mathbb{R}^{d} \rightarrow \mathbb{R}^d$ is a vector field, $\sigma: \mathbb{R}^d \rightarrow \mathbb{R}^{d\times n}$ is a matrix-valued function, and $B_{t}$ is the Wiener process in $\mathbb{R}^{n}$. If $c < 0$, obviously $X_{t}$ is stable. However, if
$c > 0$ for some $y \in \mathbb{R}^{d}$, things become less clear.

To address stability, we then consider a Lyapunov function of the form
$$
  V = |x|^{\alpha} f(y) \,,
$$
where $f(y)$ is a pre-factor function to be determined. For simplicity, we first analyze the case where $x > 0$ and $y \in \Omega$ for some compact
set $\Omega \subset \mathbb{R}^{d}$. (In the actual proof,
we need to show that $\Omega$ attracts $Y_{t}$ in a strong enough
manner.) Let $\mathcal{L}$ be the infinitesimal generator of equation
\eqref{toymodel}. To ensure $V$ serves as a Lyapunov function, we require
$$
  \mathcal{L} V = x^{\alpha}( \mathcal{L}^{y} f + \alpha c f ) < 0 \,,
$$
where $\mathcal{L}^{y}$ is the infinitesimal generator of $Y_{t}$. It
is easy to see that $V$ is a Lyapunov function for all $(x, y)$ with
$x > 0, y \in \Omega$ if we have
\begin{equation}
  \label{feynman}
  \mathcal{L}^{y} f + \alpha c f = -1 
\end{equation}
in $\Omega$. The operator $\mathcal{L}^{y} + \alpha c$ is called the
Feynman-Kac operator. Thus, the existence of $V$ is equivalent to finding a strictly positive solution to \eqref{feynman}. By Theorem \ref{freidlinPDE}, once a solution exists, its positivity follows automatically. The existence of a
solution, which is given by Theorem \ref{thm28}, can be summarized as follows. Because of
Theorem \ref{nirenberg} and Theorem \ref{liu}, the existence of a
solution to equation \eqref{feynman} boils down to the estimation of
\begin{equation}
  \label{toyMET}
  \mathbb{E}_{y}\left [  \mathbf{1}_{\{t <\tau_{y}^{\Omega}\}} \exp
    \left ( \int_{0}^{t} c(Y_{s}) \mathrm{d}s \right ) \right ] \,,
\end{equation}
where $\tau_{y}^{\Omega}$ is the first exit time from $\Omega$ when
starting at $y \in \Omega$. If the expectation in \eqref{toyMET} is negative for all initial value $y$, then equation \eqref{feynman} admits a positive solution. This follows from the multiplicative ergodic theory for Markov processes \cite{kontoyiannis2003spectral, kontoyiannis2005large}. As shown in
Theorem \ref{MET}, the multiplicative ergodic theorem follows if
$Y_{t}$ itself is geometrically ergodic and admits a suitable
Lyapunov function. Let $\pi^{Y}$ be the
invariant probability measure of $Y_{t}$. If $0 <\alpha \ll 1$, the eigenvalue perturbation
theory in \cite{kato2013perturbation} shows that $\Lambda(\alpha)$
is an $o(\alpha)$ perturbation of $\alpha \pi^Y(c)$. Therefore, the function $V = x^\alpha f(y)$ serves as a Lyapunov function of system \eqref{toymodel} if $\pi^{Y}(c) < 0$ and $\alpha$ is sufficiently small. In other words, if $Y_t$ has strong enough mixing properties and the coefficient $c(Y_{t})$ is negative on average, then $X_t$ remains stable.

We note that while similar "stabilization by noise" results—often based on the idea of "averaging" a subset of variables—have been used to establish stochastic stability in various works \cite{athreya2012propagating, herzog2015noise, hening2018coexistence}, our use of the Feynman-Kac operator in constructing a Lyapunov function is, to the best of our knowledge, novel. The key advantage of our Feynman-Kac-Lyapunov approach is that it directly yields a Lyapunov function with respect to the infinitesimal generator, which is a stronger condition than a Lyapunov function derived from the time-$T$ sample chain of a sufficiently large
$T$. As we will discuss later, this property provides significant technical benefits, particularly when handling perturbations or integrating Lyapunov-type conditions across different regions.

\subsection{Mechanism of energy transfer}

Before presenting the mathematical proof, we first explain the detailed mechanism of energy transfer that enables stochastic stability. As in many earlier works, establishing stochastic stability relies on constructing a Lyapunov function. The general idea is to assign higher values to states that the system \eqref{bigsystem} is less likely to reach and sustain. To prevent the system from escaping to infinity, we expect the total energy \( I_1 + I_2 + I_3 \) to decrease when it exceeds a certain threshold. Similarly, since the energy transfer would be  permanently broken at \( I_2 = 0 \), when \( I_2 \ll 1 \), we expect \( I_2^{-1} \) to decrease. These considerations naturally suggest two candidate Lyapunov functions, namely $( I_1 + I_2 + I_3 )^\beta$ and $ I_2^{-\alpha}$ for some constants $\alpha$ and $\beta$. 

However, this approach encounters two major difficulties. The first, the {\it high-energy problem}, arises when the middle mode \( I_2 \) accumulates a large amount of energy while \( I_1 \) and \( I_3 \) remain small. In this scenario, the total energy dissipation, approximately \( O( I_2(I_1 + I_3)) \), can be outweighed by the total energy injection, causing the expected total energy to increase rather than decrease. The second, the {\it low-energy problem}, occurs when \( I_2 \) is extremely small. Since mode interactions are proportional to \( I_2 \), a very low \( I_2 \) effectively breaks the chain. However, the growth rate coefficient of $I_2$ is \( -2 (I_1 \sin \theta_1 + I_3 \sin \theta_3) \), which can take both positive and negative values. This makes it impossible to conclude that \( I_2 \) must increase over time, further complicating the explicit construction of a Lyapunov function.  

A closer examination of system \eqref{bigsystem} reveals time-scale separation dynamics in both high- and low-energy regimes. In each case, a fast-scale, lower-dimensional subsystem guides the system toward a more stable state, where the natural Lyapunov function proposed earlier becomes valid. The dynamics of each subsystem provide a pre-factor for the main Lyapunov function. As shown in Section 3.1, this pre-factor is determined by the existence of a positive solution to a Cauchy-Dirichlet problem of the Feynman-Kac equation.  

We now give more details on the delicate dynamics that system \eqref{bigsystem} encounters in both the high- and the low-energy settings.

\medskip

{\it High-Energy Setting.} When the internal energy \( I_2 \) is high, the main challenge is that the rate of energy exchange between \( I_2 \) and \( I_1 \) (resp. \( I_3 \)) is proportional to \( I_1 I_2 \) (resp. \( I_3 I_2 \)). When the boundary energy \( I_1 \) (resp. \( I_3 \)) is small, the release of internal energy is hindered. As a result, energy dissipation into the heat bath can only continue once \( I_1 \) or \( I_3 \) recovers. A natural approach to modeling this process is to assign a high Lyapunov function value to states where \( I_1 \) (resp. \( I_3 \)) is very low, so that an increase in \( I_1 \) (resp. \( I_3 \)) significantly decreases the Lyapunov function’s value, compensating for any potential increase in internal energy. This leads to the natural Lyapunov function:  
\[
  \tilde{V} := I_1^{-\beta_1} + I_3^{-\beta_1} + (I_1 + I_2 + I_3)^{\beta_0}
\]
for some positive constants \( \beta_0 \) and \( \beta_1 \). However, calculations reveal that this approach does not work well. Specifically, when the angle \( \theta_1 \) (resp. \( \theta_3 \)) satisfies \( \sin \theta_1 < \gamma \) (resp. \( \sin \theta_3 < \gamma \)), \( I_1 \) (resp. \( I_3 \)) initially decreases, causing \( \tilde{V} \) to increase rather than decrease. Consequently, there exists a “bad set” where \( I_2 \gg 1 \), \( I_1, I_3 \sim I_2^{-1} \), and \( \sin \theta_1, \sin \theta_3 < \gamma \), such that \( \mathcal{L} \tilde{V} \) remains positive.  

A closer look at numerical simulations reveals the actual mechanism behind the dissipation of high internal energy \( I_2 \). When \( I_2 \) is large and \( I_1 \) (resp. \( I_3 \)) is small, the angle \( \theta_1 \) (resp. \( \theta_3 \)) evolves on a fast time scale and quickly converges toward \( 2\pi/3 \). This results in \( \sin \theta_1 > \gamma \) (resp. \( \sin \theta_3 > \gamma \)), enabling continued energy dissipation into the heat bath. To incorporate this mechanism into the construction of the Lyapunov function, we introduce a pre-factor \( f(\theta_1) \) (resp. \( f(\theta_3) \)) that assigns high values to the “bad” angles \( \theta_1 \) (resp. \( \theta_3 \)). This pre-factor function is obtained by solving a Cauchy-Dirichlet problem of the Feynman-Kac equation.  

A more subtle technical issue arises because the dynamics of \( \theta_1 \) (resp. \( \theta_3 \)) may linger near the unstable equilibrium \( 4\pi/3 \) for an unexpectedly long time, further complicating the analysis. To mitigate this, we artificially amplify the noise to accelerate the system’s escape from the unstable equilibrium. This leads to the assumption on the weight function \( g \) in Assumption {\bf (H)} in Section \ref{highenergy} below. While we expect Assumption {\bf  (H)} to hold for a broader class of weights, in this paper, we are only able to rigorously prove it for a specific choice of \( g \).

\medskip

{\it Low-Energy Setting.} When \( I_2 \ll 1 \), our goal is to show that \( I_2 \) cannot continue decreasing indefinitely and permanently break the chain. However, similar to the high-energy setting, the natural Lyapunov function  
\[
  \tilde{W} = I_2^{-1}
\]
fails because \( \mathcal{L} \tilde{W} \) has a leading term \( 2(I_1 \sin \theta_1 + I_3 \sin \theta_3) \), which is not necessarily negative. The actual mechanism preventing \( I_2 \) from vanishing is first observed in numerical simulations. When \( T_1 \ll T_3 \) and \( I_2 \ll 1 \), we have \( I_3 \gg I_1 \) with high probability. Consequently, the dynamics of \( \theta_3 \) is primarily governed by the term \( -I_3 (1 + 2 \cos \theta_3) \), which admits a stable equilibrium at \( 4\pi/3 \). Furthermore, because \( T_1 \ll T_3 \), we typically have \( I_1 \ll I_3 \) on average. This implies that the dominant term in \( h := I_1 \sin \theta_1 + I_3 \sin \theta_3 \) is \( I_3 \sin \theta_3 \), which takes a negative value near the stable equilibrium at \( 4\pi/3 \). Since \( \theta_3 \) remains near \( 4\pi/3 \) with high probability, on average, \( I_1 \sin \theta_1 + I_3 \sin \theta_3 < 0 \), ensuring that \( I_2 \) increases when \( I_2 \ll 1 \).  

Similar to the high-energy setting, we address this issue by introducing a pre-factor to the natural Lyapunov function \( \tilde{W} \). This pre-factor is obtained as the solution to a Feynman-Kac equation, following the same methodology. Notably, in this regime, equation \eqref{bigsystem} is well-approximated by the toy model \eqref{toymodel}, with \( X = I_2 \) and \( Y = (I_1, I_3, \theta_1, \theta_3) \). As described in Section 2, a solution to the Feynman-Kac equation exists if the reduced system governing \( Y \) is strongly mixing and the expectation of \( I_1 \sin \theta_1 + I_3 \sin \theta_3 \) is positive with respect to the invariant probability measure \( \pi^Y \) of the reduced system. This motivates a detailed study of the reduced system \( (I_1, I_3, \theta_1, \theta_3) \) in \( \mathbb{R}^4 \). Through explicit calculations, we ultimately establish the existence of such a pre-factor function.  

\medskip  

To further illustrate the mechanism of energy transfer, we conduct a detailed study of the reduced systems arising in both the high-energy and low-energy settings. In each case, we establish the stochastic stability of the reduced system and estimate the expectation and variance of the coefficient function under its steady state. While not all of these results are directly required for proving the existence of the Lyapunov function’s pre-factor, we hope they provide insight into the behavior of the reduced systems and clarify why their properties eventually contribute to the stochastic stability of the full system.

\section{Dissipation of high energy}
\label{highenergy}
The goal of this section is to show that the total energy has zero probability of blow up. In other words, we would like to show that system
\eqref{bigsystem} admits a Lyapunov function
$$
  V( \mathbf{x}) = (I_{1} + I_{2} + I_{3})^{\beta_{0}} +
  I_{1}^{-\beta_{1}}f(\theta_{1}) +
  I_{3}^{-\beta_{1}}f(\theta_{3})=
  V_2( \mathbf{x})+V_1( \mathbf{x})+V_3( \mathbf{x})$$
  where the positive constants $\beta_{0}$ and $\beta_{1}$ satisfy
  $\beta_{1} + 1 < \beta_{0} < \beta_{1} + 2$, and $f$ is a pre-factor that is determined by the following assumption {\bf (H)}.

  {\bf Assumption (H)}
  There exists a pre-factor function $f(\theta)$ on $\mathbb{S}^{1}$ and a weight function
  $g(I, \theta)$ on $\mathbb{R}^{+} \times \mathbb{S}^{1}$ such that
  
\begin{itemize}
  \item[(a)] There exist $O(1)$ constants $m_{0}$ and $m_{1}$ such that $0
    < m_{0} \leq f(\theta) \leq m_{1} < \infty$,
  \item[(b)] There exists an $O(1)$ constant $m_{2}$ such that
    $|f'(\theta)| \leq m_{2}$ and $|f''(\theta)| \leq m_{2}$,
    \item[(c)] $g(I, \theta) \geq 1$ with $\lim_{I \rightarrow 0} g(I,
      \theta) = 1$,
      \item[(d)] There exist constants $\beta_{1} >0$ and $\epsilon > 0$ such that $\mathcal{L}^{I} f \leq -
        \epsilon I$ for all sufficiently large $I$, where    
$$
  \mathcal{L}^{I} f = -2 \beta_{1} I (\sin \theta - \gamma) f + I(1 + 2\cos
  \theta) f' + \frac{\gamma }{2}g^{2}(I, \theta) f'' \,.
$$
\end{itemize}

Although we expect assumption {\bf (H)} to hold for a wide range of
coefficient functions $g(I, \theta)$, including the constant case, due
to technical reasons, in Section \ref{highenergy}, we will only
prove that assumption {\bf (H)} holds for a specific choice: $g =
\sqrt{I}$ when $I \geq 2$.

  \subsection{Analysis of the infinitesimal generator}
  We will first show that $V$ is the Lyapunov function for the
  infinitesimal generator $\mathcal{L}$ of the system \eqref{bigsystem}, assuming the coefficient $f$ satisfies
  the assumption {\bf (H)} described in Section \ref{mainresult}. The following lemma addresses the more challenging case when
  $I_{2}$ is large but $I_{1}$ and/or $I_{3}$ are small.

\begin{figure}[htbp]
\centerline{\includegraphics[width = \linewidth]{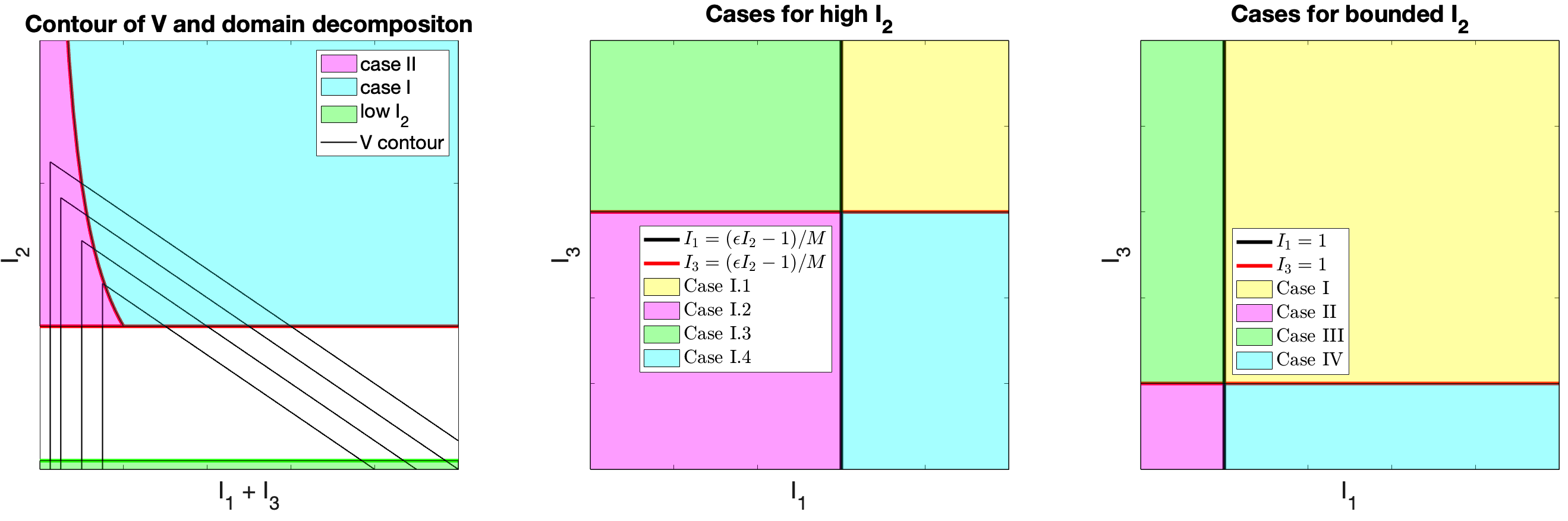}}
\caption{Left: Illustration of Case I and II in the proof of Lemma \ref{highI2} and some contour plots of $V(
  \mathbf{x})$. Middle: Illustration of different subcases of Case I in the proof of
  Lemma \ref{highI2}. Right: Illustration of different cases in the proof of Lemma \ref{lowI2}.}
\label{sketch}
\end{figure}
  
  \begin{lem}
    \label{highI2}
Assume assumption {\bf (H)} holds. Then there exist constant $c_{I_{2}} > 0$
and $C_{I_{2}} < \infty$ such that 
$$
\mathcal{L}V( \mathbf{x}) \leq -c_{I_{2}} V^{1 - 1/\beta_{0}}
$$
for all $I_{2} \geq C_{I_{2}}$.
\end{lem}

\begin{proof}
  Apply Ito's formula, we have
  
  \begin{align}
    \label{lv}
  \mathcal{L}V &= \beta_{0}(I_{1} + I_{2} + I_{3})^{\beta_{0} -
                 1}\left (-2
                 \gamma I_{1}I_{2} - 2 \gamma I_{3}I_{2} + \gamma
                 (T_{1} - I_{1}^{3}) + \gamma (T_{3} - I_{3}^{3})
                 \right )\\\nonumber
  &+ \frac{1}{2}\beta_{0}(\beta_{0} - 1)(I_{1} + I_{2} +
    I_{3})^{\beta_{0} - 2}(\frac{1}{2} \gamma T_{1}I_{1} + \frac{1}{2}
    \gamma T_{3}I_{3})\\\nonumber
  &-\beta_{1} I_{1}^{-\beta_{1} - 1}f(\theta_{1})\left (2 I_{1}I_{2} (\sin \theta_{1} -
    \gamma) + \gamma(T_{1} - I_{1}^{3})  +
    \frac{1}{4}\beta_{1}(\beta_{1} + 1)\gamma T_{1} \right )\\\nonumber
    &+ I_{1}^{-\beta_{1}}\left \{ [I_{2}(1 + 2 \cos \theta_{1}) -
    I_{1}(1+ 2 \cos \theta_{1}) - 2 I_{3} \cos \theta_{3}]
    f'(\theta_{1}) + \frac{\gamma }{2}
    g^{2}(I_{2}, \theta_{1}) f''(\theta_{2}) \right \} \\\nonumber
  &-\beta_{1} I_{3}^{-\beta_{1} - 1} f(\theta_{3})\left (2 I_{3}I_{2} (\sin \theta_{3} -
    \gamma) + \gamma(T_{3} - I_{3}^{3}) ) +
    \frac{1}{4}\beta_{1}(\beta_{1} + 1)\gamma T_{3} \right )\\\nonumber
        &+ I_{3}^{-\beta_{1}}\left \{ [I_{2}(1 + 2 \cos \theta_{3}) -
    I_{3}(1+ 2 \cos \theta_{3}) - 2 I_{1} \cos \theta_{1}]
    f'(\theta_{3}) + \frac{\gamma}{2}
    g^{2}(I_{2}, \theta_{3}) f''(\theta_{3})\right \} \\\nonumber
    :=& \mathcal{L}V_{2} + \mathcal{L}V_{1} + \mathcal{L}V_{3},
  \end{align}
where
\begin{equation}
  \label{L2}
  \mathcal{L}V_{2} = \gamma \beta_{0}(I_{1} + I_{2} + I_{3})^{\beta_{0} - 1} \left (
     -2 I_{2}( I_{1} + I_{3}) + (T_{1} - I_{1}^{3}) + (T_{3} -
     I_{3}^{3}) + \frac{T_{1}I_{1} + T_{3}I_{3}}{4(I_{1} + I_{2} +
     I_{3})}(\beta_{0} - 1) 
     \right ) \,,
\end{equation}
\begin{align}
  \label{L1}
 &  \mathcal{L}V_{1} = - \beta_{1}T_{1} \gamma (1 - \frac{\beta_{1} + 1}{4})f(\theta_{1})I_{1}^{-\beta_{1} - 1} + \gamma \beta_1 I_1^{-\beta_1 + 2} f(\theta_1)\\\nonumber
  & + I_{1}^{-\beta_{1}} \left \{ - 2 \beta_{1} I_{2}( \sin \theta_{1} - \gamma)f(\theta_{1})
    +  I_{2}( 1 + 2 \cos \theta_{1}) f'(\theta_{1}) + \frac{\gamma }{2} g^{2}(I_{2}, \theta_{1})
    f''(\theta_{1}) \} \right
    \}\\\nonumber
  & + I_{1}^{-\beta_{1}}[- I_{1} (1 + 2 \cos \theta_{1}) - 2
    I_{3} \cos \theta_{3} ] f'(\theta_{1}) \\\nonumber
  &:= A_{1} + A_{2} + A_{3}\,,
\end{align}
and
\begin{align}
  \label{L3}
 & \mathcal{L}V_{3} = - \beta_{1}T_{3} \gamma f(\theta_{3})(1 - \frac{\beta_{1} + 1}{4})I_{3}^{-\beta_{1} - 1} + \gamma \beta_1 I_3^{-\beta_1 + 2} f(\theta_3)\\\nonumber
  & + I_{3}^{-\beta_{1}} \left \{ - 2 \beta_{1} I_{2}( \sin \theta_{3} - \gamma)f(\theta_{3})
    +  I_{2}( 1 + 2 \cos \theta_{3}) f'(\theta_{3}) + \frac{\gamma }{2} g^{2}(I_{2}, \theta_{3})
    f''(\theta_{3}) \right
    \}\\\nonumber
    & + I_{3}^{-\beta_{1}} [- I_{3} (1 + 2 \cos \theta_{3}) - 2
    I_{1} \cos \theta_{1}] f'(\theta_{3}) \\\nonumber
  &:= B_{1} + B_{2}  + B_{3}\,,
\end{align}

It is easy to see that when $I_1 \ll 1$, the leading term in $\mathcal{L}V_{1}$ is $A_{1}$. If $\beta_1 < 2$, the $I_1^{-\beta_1+2}$ term is large when $I_1 \gg 1$ and has to be absorbed by $\mathcal{L}V_2$.

The key term that must be controlled by the inequality in assumption {\bf (H)} is  
$$
  A_{2} = I_{1}^{-\beta_{1}}\left \{  2 \beta_{1} I_{2}( \sin \theta_{1} - \gamma)f(\theta_{1})
    +  I_{2}( 1 + 2 \cos \theta_{1}) f'(\theta_{1}) + \frac{\gamma T_{1}}{2} g^{2}(I_{2}, \theta_{1})
   f''(\theta_{1}) \right \} \,.$$
By assumption {\bf (H)}, it follows that
$$
  A_{2}  \leq - \epsilon I_{2} I_{1}^{-\beta_{1}}
$$
for all sufficiently large $I_{2}$. In the worst-case scenario, when
$I_{1} \ll 1$ and $I_{3} > I_{1}$, the dominant term in $A_{3}$ is
$-2I_{3} \cos \theta_{3} f'( \theta_{1})I_{1}^{-\beta_{1}}$. By assumption {\bf (H)} we have
$$
  A_{3} \leq I_{1}^{-\beta_{1}}( M I_{3} + 1) 
  $$
for an $O(1)$ constant $M$. This results the bound
  
\begin{equation}
  \label{L1bound}
  \mathcal{L}V_{1} \leq A_{1} - \epsilon I_{2} I_{1}^{-\beta_{1}} +
  I_{1}^{-\beta_{1}}(M I_{3} + 1) \,.
\end{equation}
  
An analogous argument holds for $\mathcal{L}V_{3}$.

\medskip

The calculation is then divided into the following steps. Refer to Figure \ref{sketch} for an illustration of the different cases discussed in the proof.

{\it Step 1, worst case analysis for $\mathcal{L}V_{1}$ and $\mathcal{L}V_{3}$.}

Suppose $I_{1} \ll 1$ and $I_{2} \gg 1$. To control $\mathcal{L}V_{1}$, we need to
first study the worst case when $I_{3}$ is also a large number. Let $x = I_{1}$
be a changing variable. Then we have
$$
\mathcal{L}V_{1} \leq r x^{-\beta_1 + 2} - p x^{-\beta_{1} - 1} + q x^{-\beta_{1}} :=
 r x^{-\beta_1 + 2} + \phi(x) \,, 
  $$
  where
  
$$
  p = \beta_{1} T_{1} \gamma\left ( 1 - \frac{\beta_{1} + 1}{4}\right ) f(\theta_{1}) \,,
  $$
$$
  q = (M I_{3} + 1 - \epsilon I_{2}) \,,
$$
and 
$$
r = \gamma \beta_1 m_1 \,.
$$
If $ M I_{3} + 1 \leq \epsilon I_{2}$ we have
$$
  \mathcal{L}V_{1} \leq A_{1} \,.
$$
Otherwise, if $I_{3}$ is large enough that $q > 0$, then by taking the derivative, we obtain
$$
  \phi'(x) = p (\beta_{1} + 1) I_{1}^{-\beta_{1} -2}  - \beta_{1} q
  I_{1}^{-\beta_{1}-1}\,.
  $$
 Hence 
$$
  x^{*} = \frac{ (\beta_{1} + 1)p}{\beta_{1} q}
$$
is the unique critical point. Thus, after simplifications, we obtain
the bound
$$
  \mathcal{L}V_{1} \leq \frac{\beta_{1}^{\beta_{1}}}{ (1 + \beta_{1})^{1 +
      \beta_{1}}}\frac{q^{\beta_{1}+1}}{p^{\beta_{1}}} + r I_1^{-\beta_1 + 2}\,.
$$
Since $p = O(1)$ for all $\theta_{1}$, it follows that when $M I_{3}  + 1 \geq
\epsilon I_{2}$, the worst-case upper bound is
\begin{equation}
  \label{worstlv1}
  \mathcal{L}V_{1} \leq C(\beta_{1}, \gamma, T_{1}) I_{3}^{\beta_{1} + 1} + r I_1^{-\beta_1 + 2}\,.
\end{equation}

  for some $O(1)$ constant $C(\beta_{1}, \gamma, T_{1})$

A similar argument applies to $\mathcal{L}V_{3}$, leading to the bound

\begin{equation}
  \label{worstlv3}
  \mathcal{L}V_{3} \leq C(\beta_{1}, \gamma, T_{3}) I_{1}^{\beta_{1} + 1} + r I_3^{-\beta_1 + 2}\,.
\end{equation}
when $I_{1}$ is large enough such that $M I_{1} + 1 \geq \epsilon
I_{2}$. Otherwise, we can use the trivial bound
$$
  \mathcal{L}V_{3} \leq B_{1} \,.
$$

\medskip

{\it Step 2, Calculation for case I: large $(I_{1} + I_{3})I_{2}$.}

As seen in equation \eqref{L2}, the key to making $\mathcal{L}V_{2} < 0$ is that
$(I_{1} + I_{3})I_{2}$ must be greater than a certain constant. More
precisely, because
$$
  (T_{1} - I_{1}^{3}) + (T_{3} - I_{3}^{3}) + \frac{T_{1}I_{1} +
    T_{3}I_{3}}{4(I_{1} + I_{2} + I_{3})}(\beta_{0} - 1) \leq (1 +
  \frac{\beta_{0} - 1}{4})(T_{1} + T_{3}) \,,
  $$
  by letting
$$
  \mu = \frac{1}{2}(1 + \frac{\beta_{0} -1}{4})(T_{1} + T_{3}) +
  \frac{1}{2} \,,
$$

we will have  
\begin{align*}
  \mathcal{L}V_{2} &= \beta_{0}(I_{1} + I_{2} + I_{3})^{\beta_{0} -
                 1}\left (-2
                 \gamma I_{1}I_{2} - 2 \gamma I_{3}I_{2} + \gamma
                 (T_{1} - I_{1}^{3}) + \gamma (T_{3} - I_{3}^{3})
           \right )\\\nonumber
  &\leq \gamma \beta_{0}(I_{1} + I_{2} + I_{3})^{\beta_{0} - 1}\left (
    -2(I_{1} + I_{3})I_{2} + (1 + \frac{\beta_{0} - 1}{4})(T_{1} + T_{3} - I_1^3 - I_3^3)
    \right )\\\nonumber
  &= \gamma \beta_{0}(I_{1} + I_{2} + I_{3})^{\beta_{0} - 1}(-2(I_{1}
    + I_{3})I_{2} + 2 \mu - 1 - I_1^3 - I_3^3)\\
  &\leq - \gamma \beta_{0}(I_{1} + I_{2} + I_{3})^{\beta_{0} - 1}(1 + I_1^3 + I_3^3)
\end{align*}
whenever $(I_{1} + I_{3})I_{2} \geq \mu$.

Thus, we define {\bf case I} as $(I_{1} + I_{3})I_{2} \geq \mu$ for
some constant $\mu = O(1)$ described above.

\medskip

Let $0 < \epsilon < 0.5$. Case I can be further divided into four
subcases (see also Figure \ref{sketch}):
\begin{itemize}
  \item[Case I.1] $I_{1} \geq (\epsilon I_{2} - 1)/M, I_{3} \geq
    (\epsilon I_{2} - 1)/M$;
  \item[Case I.2] $I_{1} < (\epsilon I_{2} - 1)/M$, $I_{3} < (\epsilon I_{2} - 1)/M$;
    \item[Case I.3] $I_{1} < (\epsilon I_{2} - 1)/M$, $I_{3} \geq
      (\epsilon I_{2} - 1)/M$;
      \item[Case I.4] $I_{1} \geq (\epsilon I_{2} - 1)/M$, $I_{3} <
        (\epsilon I_{2} - 1)/M$.
\end{itemize}

\medskip

{\it Step 2.1: Case I.1.}
Since $I_{2}$ is a sufficiently large number, case I.1 means $\mathcal{L}V_{1}$
and $\mathcal{L}V_{3}$ are both very small numbers that can be easily
absorbed. More precisely by equation \eqref{L1bound} we have 

$$
  \mathcal{L}V_{1} + \mathcal{L}V_{3} \leq I_{1}^{-\beta_{1}}(M I_{3} + 1) +
  I_{3}^{-\beta_{1}}(M I_{1} + 1) + r I_1^2 + r I_3^2\,. 
$$
Since $I_{1}$ and $I_{3}$ are both large numbers and 
$$
  \mathcal{L}V_{2} \leq -\gamma \beta_{0}( I_{1} + I_{2} +
  I_{3})^{\beta_{0}-1}\left ( 1 + I_1^3 + I_3^3
  \right ) \,,
$$
it is easy to see that

$$
  - \mathcal{L}V_{2} \gg (\mathcal{L}V_{1} + \mathcal{L}V_{3}) \,.
$$

Therefore, noting that $V_{1}$ and $V_{3}$ are both small terms, we have
$$
  \mathcal{L} V  \leq -\frac{1}{2}\gamma \beta_{0}(I_{1} + I_{2} +
  I_{3})^{\beta_{0} -1} \leq -C_{1,1} V^{1 - 1/\beta_{0}}
$$
for some $O(1)$ constant $C_{1,1}$

\medskip

{\it Step 2.2: Case I.2}

In this case $A_3$ (resp. $B_3$) is absorbed by $A_2$ (resp. $B_2$). Thus, we have $\mathcal{L}V_{1} \leq A_{1}$ and $\mathcal{L}V_{3} \leq B_{1}$,
where $- A_{1}$ is dominantly larger than $V_{1}$ and $- B_{1}$ is
dominantly larger than $V_{3}$ for all sufficiently small $I_{1}$ and
$I_{3}$, respectively. So we have $\mathcal{L}V_1 \ll -V_1$ (resp. $\mathcal{L}V_3 \ll -V_3$) if $I_1 \ll 1$ (resp. $I_3 \ll 1$).

If otherwise $I_1$ or $I_3$ is not a small term, then $A_1$ (resp. $B_1$) is at most $r I_1^2$ (resp. $r I_3^2$). Since we also have
$$
  \mathcal{L}V_{2} \leq - \gamma \beta_{0}( I_{1} + I_{2} +
  I_{3})^{\beta_{0}-1} (1 + I_1^3 + I_3^3)\,,
$$
the $I_1^2$ and $I_3^2$ terms can be absorbed by $\mathcal{L}V_2$.

    Therefore, regardless of the scales of $I_{1},
    I_{2}$, and $I_{3}$,  in the worst case, we still obtain
    
$$
  \mathcal{L} V \leq - C_{1,2} V^{1 - 1/\beta_{0}}
$$
for a constant $C_{1,2}$.

\medskip

{\it Step 2.3: Case I.3 and Case I.4.}

We will only show the proof of Case I.3, as the proof of Case I.4 is
identical. The main idea is to use the worst-case analysis for
$\mathcal{L}V_{1}$. For all sufficiently large $I_2$, when $I_{3}$ is large
, consider the worst case of $I_{1}$. That is, \eqref{worstlv1}, we
have
$$
  \mathcal{L}V_{1} \leq C(\beta_{1}, \gamma, T_{1})I_{3}^{\beta_{1} + 1} + r I_1^2\,.
$$
In addition, it is easy to see that $\mathcal{L}V_{3} \leq 1 + r I_3^2$ because all other terms in $\mathcal{L}V_3$ are small.

Since both $I_{2}$ and $I_{3}$ are large, without loss of
generality, we assume $I_{2}$ is sufficiently large so that
$$
  2 I_{2}I_{3} - (1 + \frac{\beta_{0}-1}{4})(T_{1} + T_{3}) \geq
  (I_{1} + I_{2} + I_{3}) \,.
$$

Therefore, we have
$$
  \mathcal{L}V \leq - \gamma \beta_{0}(I_{1} + I_{2} +
  I_{3})^{\beta_{0}-1}(I_1 + I_2 + I_3 + I_1^3 + I_3^3) + r I_3^2 + C(\beta_{1}, \gamma, T_{1})I_{3}^{\beta_{1}
    + 1} + r I_1^2\,.
$$

Since $\beta_{0} > \beta_{1} + 1$, the dominant term is still
$\mathcal{L}V_{2}$. Thus, we have
$$
  \mathcal{L}V \leq -\frac{1}{2} \gamma \beta_{0}( I_{1} + I_{2} +
  I_{3})^{\beta_{0}-1} (I_1 + I_2 + I_3 + I_1^3 + I_3^3) \leq - \frac{1}{2} \gamma \beta_0 V\,.
  $$

  It remains to compare $\mathcal{L}V$ and $V^{1 - 1/\beta_0}$. Since $V_{3}$ is
  small, if $V_{1} \leq V_{2}$, then obviously we have
$$
  \mathcal{L}V \leq -C_{1,3} V^{1 - 1/\beta_{0}} \,.
  $$
  
Otherwise, if $V_{1} > V_{2}$, we have $I_{1}^{-1} \gg I_{3}$ because
$\beta_{0} > \beta_{1} + 1$. In this case, we have better bound
$$
  \mathcal{L}V_{1} \leq 1 +  p I_{1}^{-\beta_{1}-1} + q I_{1}^{-\beta_{1}} \leq
  -\frac{1}{2}p I_{1}^{-\beta_{1}-1}
  $$
  because $q = O(I_{3})$, where the terms $p$ and $q$ come from Step 1 of
  the proof. In this case we have $\mathcal{L}V_{1} \leq - V_{1}$ when $I_{1}$
  is sufficiently small. Thus, we still have
  
$$
  \mathcal{L}V \leq -C_{1,3} V^{1 - 1/\beta_{0}} \,.
$$

  {\it Step 3, Case II.}

  Case II is defined by $(I_{1} + I_{3})I_{2} < \mu$ for
  
$$
  \mu = \frac{1}{2}(1 + \frac{\beta_{0} - 1}{4})(T_{1} + T_{3}) +
  \frac{1}{2} \,.
$$

The main difference in Case II is that $\mathcal{L}V_{2}$ becomes
positive. Thus, we need $\mathcal{L}V_{1}$ and $\mathcal{L}V_{3}$ to control
$\mathcal{L}V$. 

Notice that in Case II, we obviously have $\mathcal{L}V_{1} \leq A_{1}$
and $\mathcal{L}V_{3} \leq B_{1}$, as the term $A_3$ (resp. $B_3$) is absorbed by $A_2$ (resp. $B_2$). In addition, the terms $I_1^{-\beta_1 + 2}$ and $I_3^{-\beta_1 + 2}$ are both small. The same calculation gives the bound
\begin{align*}
  \mathcal{L}V  \leq & \gamma \beta_{0}(I_{1} + I_{2} + I_{3})^{\beta_{0} - 1} \left (
     -2 I_{2}( I_{1} + I_{3}) + (T_{1} - I_{1}^{3}) + (T_{3} -
     I_{3}^{3}) + \frac{T_{1}I_{1} + T_{3}I_{3}}{4(I_{1} + I_{2} +
     I_{3})}(\beta_{0} - 1) 
     \right )\\
     &+ A_{1} + B_{1}\\
  \leq & \gamma \beta_{0}(I_{1} + I_{2} + I_{3})^{\beta_{0} - 1}\left ( (1 +
         \frac{\beta_{0} - 1}{4})(T_{1} + T_{3}) \right) + A_{1} +
         B_{1} \,.
\end{align*}

When $I_{2}$ is large, both $A_{1}$ and $B_{1}$ are significantly large
negative terms on the scale of $O(I_{2}^{\beta_{1}+1})$. Because
$\beta_{0} < \beta_{1} + 2$, we have $\beta_{0} - 1 < \beta_{1} +
1$. Hence $I_{1}^{-\beta_{1} - 1}$ and $I_{3}^{-\beta_{1} - 1}$ are
much larger than $(I_{1} + I_{2} + I_{3})^{\beta_{0} - 1}$ when
$I_{2}$ is sufficiently large. This gives
$$
  \mathcal{L}V \leq - C( I_{1}^{-\beta_{1} - 1} +
  I_{3}^{-\beta_{1} - 1} )
  $$
  for some constant $C > 0$. Noting that $(I_{1} + I_{3})I_{2} <
  \mu$, we have $O(I_{1}^{-\beta_{1} - 1}) \gg O(I_{2}^{\beta_{1} +
    1})$, $O(I_{3}^{-\beta_{1} - 1}) \gg O(I_{2}^{\beta_{1} +
    1})$, and $(I_{1} + I_{2} + I_{3})^{\beta_{0}} \approx
  I_{2}^{\beta_{0}}$. Thus, when $I_{2}$ is large, we have
  
$$
  \mathcal{L}V \leq - C_{2} V \leq - C_{2}
  V^{1 - 1/\beta_{0}} 
  $$
for an $O(1)$ constant $C_{2}$. 
  \medskip

  The proof is completed by combining the estimates from {\bf Case I.1} to {\bf Case
  I.4} and {\bf Case II}. 
\end{proof}

It remains to show that $V$ is also a Lyapunov function when $I_{2}
\leq C_{I_{2}}$ but $V$ still has a large value. This is proved in the
next lemma.

\begin{lem}
  \label{lowI2}
  There exist constants $\epsilon_{1,3} > 0$ and $0<C_{1,3} < \infty$
  such that for all $\mathbf{x}(I_{1}, I_{2}, I_{3}, \theta_{1}, \theta_{3})$
  with $I_{2}\leq C_{I_{2}}$, if $I_{1} + I_{3} > C_{1,3}$ or
  $\min\{I_{1}, I_{3}\} < \epsilon_{1,3}$, we have
$$
  \mathcal{L}V(\mathbf{x}) \leq -V(\mathbf{x}) \,.
  $$ 
\end{lem}
\begin{proof}
We still use notations defined in equation \eqref{L1} and \eqref{L3}. The main difference from the proof of the previous lemma is that now
we only have $A_{2} \leq C I_{1}^{-\beta_{1}}$ and $B_{2} \leq C I_{3}^{-\beta_{1}}$ in
equations \eqref{L1} and \eqref{L3}, respectively. Similarly, as in the proof of the previous lemma, we need to discuss four different cases (see also Figure \ref{sketch}). 

{\bf Case I:} $I_{1} \leq 1$, $I_{3} \leq 1$, and $\min \{ I_1, I_3 \} < \epsilon_{1,3}$.

This is the trivial case because $\mathcal{L}V_{2}$ is now bounded by a
constant. Since $I_{2}$ is bounded, we have $A_{2} \leq C
I_{1}^{-\beta_{1}}$ and $B_{2} \leq C I_{3}^{-\beta_{1}}$  for some
constant $C$. Additionally, the cross-terms $A_{3}$ and $B_{3}$ in equations \eqref{L1}
and \eqref{L3} are bounded by $C I_{1}^{-\beta_{1}}$ and $C
I_{3}^{-\beta_{3}}$, respectively, for some constant $C$. Therefore, the
dominant terms in $\mathcal{L}V$ are $A_{1}$ and $B_{1}$ in equations
\eqref{L1} and \eqref{L3}, both of which are significantly larger in absolute value, than
$V_{1}$ and $V_{3}$, respectively, when $I_{1} \ll 1$ and $I_{3} \ll 1$. This gives
$$
  \mathcal{L}V_{1} \leq - \beta_{1} T_{1} \gamma( 1 - \frac{\beta_{1} + 1}{4})
  m_{0} I_{1}^{-\beta_{1}-1} + O( I_{1}^{-\beta_{1}}) \leq - 1.1 V_{1}
$$
(resp. 
$$
  \mathcal{L}V_{3} \leq - \beta_{1} T_{3} \gamma( 1 - \frac{\beta_{1} + 1}{4})
  m_{0} I_{3}^{-\beta_{1}-1} + O( I_{3}^{-\beta_{1}}) \leq - 1.1 V_{3}\quad)
$$
for all sufficiently small $I_{1}$ (resp. I$_{3}$).  If $V$ is large, at least one of $I_{1}$ or
$I_{3}$ has to be extremely small. Since all other terms have significantly smaller
scales that can be absorbed into $0.1 V_{1}$ or $0.1 V_{3}$, it is easy to see that we must have
$$
  \mathcal{L}V \leq -V 
$$
if $V$ is sufficiently large.

\medskip

{\bf Case II:} $I_{1} \geq 1$, $I_{3} \geq 1$, and $I_1 + I_3 > C_{1,3}$.
This is another trivial case because $V_{1}$ and $V_{3}$ are bounded by constants, and $\mathcal{L}V_{1}$ and $\mathcal{L}V_{3}$ are bounded by $O(I_1^{-\beta_1 + 2})$ and $O(I_3^{-\beta_1 + 2})$, respectively. Note that if $V$ is
sufficiently large, then $I_{1} + I_{3}$ must be large. In addition, for $\mathcal{L}V_{2}$ we obtain
\begin{align*}
  \mathcal{L} V_{2} & = \gamma \beta_{0}(I_{1} + I_{2} +
                      I_{3})^{\beta_{0} - 1}\left( -2I_{2}(I_{1} +
                      I_{3}) + (T_{1} - I_{1}^{3}) + (T_{3} -
                      I_{3}^{3}) + \frac{T_{1}I_{1} +
                      T_{3}I_{3}}{4(I_{1} + I_{2} + I_{3})}(\beta_{0}
                      - 1)\right)\\
  &\leq \gamma \beta_{0}(I_{1} + I_{2} +
                      I_{3})^{\beta_{0} - 1}\left( (T_{1} + T_{3})(1 +
    \frac{\beta_{0}-1}{4}) - I_{1}^{3} - I_{3}^{3}\right)
\end{align*}
Notice that $I_{1}^{3} + I_{3}^{3} \geq \frac{1}{4}(I_{1} +
I_{3})^{3}$ for nonnegative $I_{1}, I_{3}$ (because $I_{1}^{3} +
I_{3}^{3} - I_{1}^{2}I_{3} - I_{1}I_{3}^{2} = (I_{1} -
I_{3})^{2}(I_{1} + I_{3})$). In addition, $I_{2}$ is bounded by
$C_{I_{2}}$. Thus, the leading term of $\mathcal{L}V_{2}$ is
$- \frac{1}{4}\gamma \beta_{0}(I_{1} + I_{2} + I_{3})^{\beta_{0}
  -1}(I_{1} + I_{3})^{3}$. There
exists a constant $C_{1,3}$ such that
$$
  \mathcal{L}V_{2} \leq -\frac{1}{8}\gamma \beta_{0}(I_{1} + I_{2} +
  I_{3})^{\beta_{0} + 2} \leq - 1.1 V_{2}
  $$
  for all $I_{1} + I_{3} \geq C_{1,3}$, because the $(I_{1} + I_{2} +
  I_{3})^{\beta_{0}+2}$ is of higher order. This gives
  
$$
  \mathcal{L}V \leq -V 
$$
because now all terms in $ \mathcal{L}V_{1}$, $ \mathcal{L}V_{3}$, $V_{1}$, and $V_{3}$ can be absorbed by $0.1 V_{2}$.
\medskip

{\bf Case III:} $I_{1} \leq 1$, $I_{3} \geq 1$, and either $\min\{ I_1, I_3 \} < \epsilon_{1,3}$ or $I_1 + I_3 > C_{1,3}$.

This is similar to Case I.3 in the previous proof. When $I_1 \leq 1$, since both $f(\theta)$ and its first and second derivatives are bounded, we have
$$
LV_1 \leq -p x^{-\beta_1 - 1} + q x^{-\beta_1} \,,
$$
where 
$$
p = \beta_1 T_1 \gamma \left ( 1 - \frac{\beta_1 + 1}{4}\right) f(\theta_1)
$$
and 
$$
q = M I_3 + M_2 \,,
$$
where the term $M_2$ absorbs the $I_1^{-\beta_1 + 2}$ term in $A_1$, the entire $A_2$, and the term $-2I_1 (1 + 2 \cos \theta_1)$. Notice that $M_2$ is now an $O(1)$ constant because $I_2$, $I_1$, $f(\theta_1)$, and all derivatives of $f$ are bounded. Then a similar worst case analysis as in the proof of the last lemma gives 
$$
  \mathcal{L}V_{1} \leq C_3 I_{3}^{\beta_{1}+1} 
$$
for a new $O(1)$ constant $C_3$.

 If we know that  $I_{3}$ is very  large, then we could write 
$$
  \mathcal{L}V_{2} \leq  \gamma \beta_{0}(I_{1} + I_{2} +
  I_{3})^{\beta_{0}-1}\left( (T_{1} + T_{3})(1 +
    \frac{\beta_{0}-1}{4}) - I_{3}^{3}\right) \,
$$
where the term $I_3^3$ is dominant. Since $I_{1}$, $I_{2}$ are both bounded, there exists a constant
$C_{I_{3}}$ such that
$$
  \mathcal{L}V_{2} \leq - \frac{1}{2} \gamma \beta_{0}( I_{1} + I_{2} +
  I_{3})^{\beta_{0} + 2} 
  $$
whenever $I_{3} \geq C_{I_{3}}$. In addition, since $I_{3} \geq 1$, $V_{3}$ is bounded by constant and $ \mathcal{L}V_{3}$ is bounded by $O(I_3^{-\beta_1 + 2})$
  
Therefore, we have the following cases for discussion.

(a) If $I_{3} \geq C_{I_{3}}$ and $V_{1} \leq V_{2}$, we have
$$
  \mathcal{L}V \leq  - \frac{1}{2} \gamma \beta_{0}( I_{1} + I_{2} +
  I_{3})^{\beta_{0} + 2} +  C
  I_{3}^{\beta_{1}+1}  + C I_3^{-\beta_1 + 2} \leq - 3 V_{2} \leq -V
$$
for all sufficiently large values of $V$, because now $ \mathcal{L}V_{2}$ is the
dominant term.

(b) If $I_{3} \geq C_{I_{3}}$ and $V_{1} > V_{2}$, then we have $I_{3}
< I_{1}^{-\beta_{1}/\beta_{0}} \ll I_{1}^{-1}$ because $I_{1}^{-\beta_{1}} > (I_{1} + I_{2} +
I_{3})^{\beta_{0}}$. Notice that $\beta_{0} > \beta_{1} + 1$. In this
case for all $I_{1}$ that are sufficiently small, the term
$A_{3}$ in equation \eqref{L1} is bounded by a small $\epsilon
I_{1}^{-\beta_{1} - 1}$ for 
$$
  \epsilon < \frac{1}{3}\beta_{1} T_{1} \gamma (1 -\frac{\beta_{1} +
    1}{4}) m_{0} \,,
  $$
  because $I_{1}I_{3}$ is a $\ll 1 $ term. In addition, $A_{2}$ is also an $O(I_{1}^{-\beta_{1}})$ term. This
gives
$$
  \mathcal{L}V_{1} \leq - \frac{1}{2}\beta_{1} T_{1} \gamma (1 -\frac{\beta_{1} +
    1}{4}) m_{0} I_{1}^{-\beta_{1} - 1} \,.
$$
The dominant term is $I_{1}^{-\beta_{1} - 1}$ because
$I_{1}^{-\beta_{1} - 1} \gg V_{1} > \frac{1}{3}V$. Hence, we have
$$
  \mathcal{L}V \leq -V
$$
for sufficiently large $V$.

(c)  If $I_{3} \leq C_{I_{3}}$, then the only situation that makes $V$
large is when $I_{1} \ll 1$. In addition, the terms $V_{3},  \mathcal{L}V_{3}, V_{2}$, and
$ \mathcal{L}V_{2}$ are now bounded by constants. Since $I_{3}$ is now bounded, in equation
\eqref{L1} we have both $A_{2}$ and $A_{3}$ as
$O(I_{1}^{-\beta_{1}})$ terms, making $A_{1}$ the dominant
term. Therefore, since $I_{1}^{-\beta_{1} -1 } \gg I_{1}^{-\beta}$,
for sufficiently small $I_{1}$ we have
$$
  \mathcal{L}V \leq \frac{1}{2}  \mathcal{L}V_{1} \leq - \frac{1}{2}\beta_{1} T_{1} \gamma (1 -\frac{\beta_{1} +
    1}{4}) m_{0} I_{1}^{-\beta_{1} - 1}  + O( I_{1}^{-\beta_{1}}) \leq
  -V \,.
$$

{\bf Case IV: $I_{1} \geq 1$, $I_{3} \leq 1$.} The proof of this case
is identical to that of {\bf Case III}.

The lemma is proved by combining the above four cases. 
\end{proof}

Combine Lemmata \ref{highI2} and \ref{lowI2}, we conclude that the
infinitesimal generator admits a Lyapunov function.

\begin{thm}
  \label{generator}
For every $\mathbf{x} \in \Omega$, there exist constants
$c > 0$ and $0<C_{V} < \infty$ such that
$$
  \mathcal{L} V \leq -c V^{1 - 1/\beta_{0}}
$$
for all $\mathbf{x}$ with $V( \mathbf{x}) > C_{V}$.
\end{thm}
\begin{proof}
Noting that for large $V$, $\mathcal{L}V \leq - V$ implies $\mathcal{L} V \leq - c V^{1 - 1/\beta_0}$. This theorem is implied immediately by Lemma \ref{highI2} and Lemma
\ref{lowI2} because all $\mathbf{x}$ with sufficiently large $V( \mathbf{x})$ values
are covered by the conditions in either Lemma \ref{highI2} or Lemma
\ref{lowI2}. 
\end{proof}

\subsection{Reduced system and proof of assumption {\bf (H)}}
\label{reducelow}

We provide this subsection to further explain the mechanism behind the
dissipation of high internal energy $I_{2}$ and our motivation for constructing the
Lyapunov function $V$. We will also explain the technical assumption
{\bf (H)} with respect to the coefficient $g(I_{2}, \theta)$ in equation
\eqref{bigsystem}, which ensures that the pre-factor $f$ of the Lyapunov function
works uniformly for large $I_{2}$. 

From the expression of $\mathcal{L}V_{2}$ we can
see that when both $I_{1}$ and $I_{3}$ are small, $\mathcal{L}V_{2}$ becomes a
positive term, indicating an increase in the total energy. This needs
to be compensated for by Lyapunov functions of $I_{1}$ and $I_{3}$,
which penalize the state of very low $I_{1}$ and $I_{3}$. However, if we denote $\tilde{V}_{1} =
I_{1}^{-\beta_{1}}$, we can see that
$$
  \mathcal{L} \tilde{V}_{1} = - \beta_{1}I_{1}^{-\beta_{1} - 1}\left ( 2I_{1}I_{2} (\sin \theta_{1}
       - \gamma) + T_{1} \gamma (1 - \frac{\beta_{1} + 1}{4}) - \gamma
       I_{1}^{3} \right)  \,.
$$
Therefore, when $\sin \theta_{1} - \gamma < 0$, even $\tilde{V}_{1}$
is an increasing term. In order to release the high $I_{2}$ stored in the
middle of the chain, the phase $\theta_{1}$ has to be changed.

Looking closer, one can find that the system
  \eqref{bigsystem} undergoes an interesting multi-scale
  dynamics. Because $I_{2}$ is dominantly large, the dynamics of $\theta_{1}$ and
  $\theta_{3}$ can be seen as a perturbation of the following
  effective dynamics (where $\theta$ is used to denote the generic phase variable)
  \begin{equation}
    \label{eqntheta0}
\mathrm{d} \theta =  I_{2} (1 + 2 \cos \theta) \mathrm{d}t
                                + \sqrt{\gamma } g(I_{2}, \theta)\mathrm{d}
                                B_{t } \,.
\end{equation}
The role of $g(I_{2}, \theta)$ is to increase the level of noise
in the neighborhood of the unstable equilibrium, which ``pushes'' the
$\theta$ away from the unstable equilibrium $4 \pi /3$ faster,
so that one can find a single pre-factor function $f$ that works
uniformly for all large $I_{2}$.  It is easy to see that
equation \eqref{eqntheta} quickly moves $\theta$ into a small neighborhood of $2 \pi
  /3$, which is the stable equilibrium of equation $\theta' =  I_{2}
  (1 + 2 \cos \theta)$. Because $\sin 2 \pi/3 = \sqrt{3}/2 > 0$,
 this fast dynamics means the phase $\theta$ will
 move toward the vicinity of $2 \pi /3$ after a short period of time when $\gamma$ is
 small, at which point we have $\sin \theta_{1} > \gamma$. Therefore, we
 will have $\mathcal{L} \tilde{V}_{1} < 0$ for all sufficiently small $I_{1}$. 
 The expected value of $I_{1}^{-\beta_{1}}$ decreases over the slow
 dynamics after the angle becomes ``good'', i.e., when $\sin \theta_{1} > \gamma$. After an additional
 period of time, the expected value of $\tilde{V}_{1}$ will be
 lower than its initial value. Moreover, a higher $I_{1}$ will eventually
 release the total energy stored in the system.

 In this subsection, we rigorously prove assumption {\bf (H)} for
 a particular choice of $g$:
$$
g(\theta, I) = \left \{
\begin{array}[tb]{lcl}
  1& \mbox{ if } & g \leq 1 \\
  \phi(I) &\mbox{ if} & 1 < g \leq 2\\
  \sqrt{I} & \mbox{ if } & g > 2 \,,
\end{array}
\right .
$$
where $\phi$ is an interpolation function that makes $g$ a $C^{2}$
continuous function. More precisely, for $I_{2} > 2$, we consider the system 
  \begin{equation}
    \label{eqntheta}
\mathrm{d} \theta =  I_{2} (1 + 2 \cos \theta) \mathrm{d}t
                                + \sqrt{\gamma  I_{2} }\mathrm{d}
                                B_{t } \,.
\end{equation}
instead.

We first provide some
rigorous justification for the invariant probability measure of the
$\theta$ dynamics in equation \eqref{eqntheta}  to demonstrate that after a short time, the phase of the reduced
system will be stabilized at an angle that makes $\sin \theta_{1} >
\gamma$. In addition, the mixing is strong enough that the
multiplicative ergodic theorem holds, which leads to the existence of
the pre-factor function $f(\theta)$ in item (d) of assumption {\bf
  (H)}. Although theoretically, assumption {\bf (H)} should hold for a
wide range of functions $g$ including the constant function $g = 1$, due to
technical reasons, we can only prove assumption {\bf (H)} for the
particular choice $g = \sqrt{I}$ as described above.

\begin{lem}
  \label{lem44}
  For each $I_{2}$, the invariant probability measure of equation \eqref{eqntheta}
  has a probability density function
  
$$
  u^{*}(\theta) = \frac{1}{Z} \left ( \frac{(e^{-2\pi \alpha} -
      1)e^{\alpha (\theta + 2 \sin \theta)}}{\int_{0}^{2\pi}
      e^{-\alpha(r + 2 \sin r) }\mathrm{d}r}\int_{0}^{\theta}
    e^{-\alpha(r + 2 \sin r)} \mathrm{d}s + e^{\alpha(\theta + 2 \sin
      \theta)}\right ) \,,
  $$
  where $\alpha = \frac{2}{\gamma }$, and $Z$ is a normalizer. 
\end{lem}

\begin{proof}
  The stationary Fokker-Planck equation corresponding to equation \eqref{eqntheta}
  is
$$
 \frac{\mathrm{d}}{\mathrm{d} \theta} (- I_{2} (1 + 2 \cos \theta) u) +
 \frac{1}{2}\gamma I_{2} \frac{\mathrm{d}^{2}}{\mathrm{d} \theta^{2}} u =
 0 \,.
$$

Let $\alpha = \frac{2}{\gamma }$. After some simplifications,
we have
$$
  \frac{\mathrm{d}}{\mathrm{d} \theta} ( -(1 + 2 \cos \theta) u) +
 \alpha^{-1} \frac{\mathrm{d}^{2}}{\mathrm{d} \theta^{2}} u =
 0 \,.
$$
Or
$$
  \frac{\mathrm{d}}{\mathrm{d} \theta} ( -(1 + 2 \cos \theta) u +
 \alpha^{-1} \frac{\mathrm{d}}{\mathrm{d} \theta} u) =
 0 \,,
$$
which implies
$$
  u' - \alpha(1 + 2 \cos \theta )u = C_{0}
 $$
  for some constant $C_{0}$.

  Together with the periodic boundary condition $u(0) = u(2 \pi)$, we have
  
$$
  u(\theta) = C_{0} e^{\alpha (\theta + 2 \sin
    \theta)}\int_{0}^{\theta} e^{-\alpha(r + 2 \sin r)} \mathrm{d}r +
  C_{1} e^{\alpha (\theta + 2 \sin \theta)}
$$
for a free constant $C_{1}$. Without loss of generality, let $C_{1} =
1$ (because we will renormalize $u$ anyway). The periodic boundary
condition gives
$$
  C_{0} = \frac{e^{-2 \pi \alpha} - 1}{\int_{0}^{2\pi} e^{-\alpha ( r +
      2 \pi r)} \mathrm{d}r} \,.
$$
The proof is completed by normalizing $u$.

\end{proof}

For the sake of simplicity, we denote the random variable in
$\mathbb{S}^{1}$ whose probability density function is $u^{*}(\theta)$
by $\Theta^{*}$. The following two lemmata estimate the expectation
and variance of $\Theta^{*}$. 

\begin{lem}
  \label{concentrationtheta}
For any constants $c > 0$ and $\epsilon >0$, there exists a positive
constant $\gamma_{0} > 0$, such that
$$
  \mathbb{P}\left[ \left| \Theta^{*} - \frac{2\pi}{3} \right| > c\right] < \epsilon 
  $$
  for all $\gamma < \gamma_{0}$.
\end{lem}
\begin{proof}
Define $s = I_{2} t$ and consider $s$ as the new time variable. The equation of
$\theta$ becomes
\begin{equation}
  \label{eqnthetas}
  \mathrm{d} \theta = (1 + 2 \cos \theta) \mathrm{d}s + \sqrt{\gamma} \mathrm{d}B_{s} \,.
  \end{equation}

 Since $\mathbb{S}^{1}$ is compact and the noise in \eqref{eqnthetas} 
 is non-degenerate, $\theta_{s}$ is
 irreducible. Thus, the invariant probability measure given by Lemma
 \ref{lem44} must be unique. The deterministic part
 
$$
  \mathrm{d}\theta = 2(1 + 2 \cos \theta_{3}) \mathrm{d}t
  $$
  has a unique stable equilibrium at $\theta = 2
  \pi /3$. Then the conclusion of the  lemma follows from standard Freidlin-Wentzell large
  deviation theory, for example Theorem 4.4.3 of
  \cite{freidlin2012random}. 

\end{proof}

Since $\mathbb{S}^{1}$ is compact, the following Corollary holds.

\begin{cor}
  \label{thetaexp}
For any $\epsilon  > 0$, there exists  $\gamma_{0}
> 0$ such
that
$$
  \left| \mathbb{E}\left[\sin (\Theta^{*}) \right] - \frac{\sqrt{3}}{2} \right| < \epsilon
  $$
  for any $\gamma < \gamma_{0}$.
\end{cor}

It remains to show that $\Theta^{*}$ significantly concentrates aroud the
stable equilibrium. Known results about concentration of the invariant
probability measure \cite{li2016systematic} can be applied here. But since the
invariant probability measure is already explicitly given, it is
easier to calculate the local WKB expansion directly. 

\begin{lem}
  \label{vartheta}
$$
  \mathrm{Var}( \Theta^{*}) = O( \gamma)
  $$
  if $\gamma \ll 1$. 
\end{lem}
\begin{proof}
Notice that
\begin{align*}
  u^{*}(\theta)  &= \frac{1}{Z} \left ( \frac{(e^{-2\pi \alpha} -
      1)}{\int_{0}^{2\pi}
      e^{-\alpha(r + 2 \sin r) }\mathrm{d}r}\int_{0}^{\theta}
    e^{-\alpha(r + 2 \sin r)} \mathrm{d}r + 1 \right )e^{\alpha(\theta + 2 \sin
  \theta)} \\
  & := \frac{1}{Z}h(\theta) e^{\alpha (\theta + 2 \sin \theta)} \,,
\end{align*}
where $\alpha = 2 /\gamma $. The negative integral of the
deterministic part $U: =  - (\theta + 2
\sin \theta)$  plays the role of a potential function, and $e^{\alpha
  (\theta + 2 \sin \theta)}  = e^{-\alpha U}$ is the Gibbs
measure. Freidlin-Wentzell large deviation theory implies that $u^*$ is negligible when $\theta$ is away from the unique attractor $2\pi/3$. Therefore, it is sufficient to estimate the pre-factor function $h(\theta)$ in the vicinity of $2\pi/3$.

At $\theta^{*} = 2 \pi /3$, we have
\begin{align*}
  h(\theta^{*}) &= \frac{(e^{-2\pi \alpha} -
      1)}{\int_{0}^{2\pi}
      e^{-\alpha(r + 2 \sin r) }\mathrm{d}r}\left ( \int_{0}^{2 \pi}
    e^{-\alpha(r + 2 \sin r)} \mathrm{d}r  - \int_{\theta^{*}}^{2\pi}
                  e^{-\alpha(r + 2 \sin r)} \mathrm{d}r  \right ) +
                  1\\
  &= e^{-2 \pi \alpha} + \frac{1 - e^{-2\pi \alpha} }{\int_{0}^{2\pi}
      e^{-\alpha(r + 2 \sin r) }\mathrm{d}r} \int_{\theta^{*}}^{2\pi}
                  e^{-\alpha(r + 2 \sin r)} \mathrm{d}r \,.
\end{align*}

When $\alpha \gg 1$, we have
$$
  \int_{0}^{2\pi} e^{-\alpha(r + 2 \sin r) }\mathrm{d}r =
  O(\alpha^{-1}) 
$$
because $r + 2 \sin r$ can be approximated by $3r$ at its minimum $r = 0$.

For the integral $\int_{\theta^{*}}^{2\pi} e^{-\alpha(r + 2 \sin r)}
\mathrm{d}r $, notice that in the interval $[2\pi/3, 2\pi]$, the function
$r + 2 \sin r$ reaches its minimum
at $r = 4 \pi/3$. When $\alpha \gg 1$, we can approximate $-\alpha (r + 2 \sin r)$ by 
$$
-\alpha (r + 2 \sin r) \approx \frac{4 \pi}{3} + \sqrt{3} - \frac{\sqrt{3}}{2}( r - \frac{4 \pi}{3})^2 \,.
$$
This means
$$
  \int_{\theta^{*}}^{2\pi} e^{-\alpha(r + 2 \sin r)}
\mathrm{d}r  = O(\alpha^{-1/2}\exp(-\alpha(4 \pi /3 - \sqrt{3}))) \,.
$$
In addition, $e^{-2\pi \alpha }$ is much smaller than the previous
term and $1 - e^{-2 \pi \alpha} \approx 1$. Hence we have the estimate
$$
  h( \theta^{*}) = O( \alpha^{1/2} \exp(-\alpha(4 \pi /3 - \sqrt{3}))) \,.
  $$

  The derivative of $h(\theta)$ is
  
$$
  h'(\theta) = \frac{(e^{-2\pi \alpha} -
      1)}{\int_{0}^{2\pi}
      e^{-\alpha(r + 2 \sin r) }\mathrm{d}r} e^{-\alpha ( \theta + \sin \theta)} = O(\alpha \exp(-\alpha (\theta + \sin \theta))) \,.
    $$
    for all $\theta$ near $\theta^*$. Notice that for all $\theta$ in the neighborhood of $\theta^*$ we have 
$$
  h(\theta^*) \gg h'(\theta)\,,
  $$
  meaning that the pre-factor $h(\theta)$ is nearly constant in the
  vicinity of $\theta^{*}$. Thus, the approximation
  
$$
  u^{*}(\theta) = \frac{1}{Z'} \exp(\alpha (\theta + 2 \sin \theta))
$$
does not change the scaling of the variance of $\Theta$, where $Z'$ is
a normalizer.

Hence the variance of $\Theta^*$ can be calculated explicitly. The Taylor expansion of $\theta + 2 \sin \theta$ at $2\pi/3$ gives the approximation
$$
  \theta + 2 \sin \theta = \frac{2\pi}{3} + \sqrt{3} - \frac{\sqrt{3}}{2}(\theta -
  \frac{2\pi}{3})^{2} + O(  (\theta -
  \frac{2\pi}{3})^{3} ) \,.
$$
Hence in the vicinity of $2\pi/3$, $u^{*}(\theta)$ is approximated by
a normal probability density function
$$
\frac{1}{Z'}\exp\{ \alpha ( \theta + 2 \sin \theta)\} \,.
$$
By the Taylor expansion, we know that $Z' = O(\exp(\alpha ( 2 \pi/3 + \sqrt{3}))\alpha^{-1/2})$. Therefore, we have
\begin{align*}
  \mathrm{Var}[ \Theta^*] &\approx \int (\theta - \frac{2 \pi }{3})^{2}
  \frac{1}{Z'} \exp\{\alpha ( \theta + 2 \sin \theta)\}
                          \mathrm{d}\theta\\
  &= \int  (\theta - \frac{2 \pi }{3})^{2}
  \frac{1}{Z''} \exp\{-\alpha [ \frac{\sqrt{3}}{2} (\theta - \frac{2\pi}{3})^{2} + O(  (\theta -
    \frac{2\pi}{3})^{3})] \} \mathrm{d}\theta \\
  &= O( \alpha^{-1}) = O(\gamma)\,,
\end{align*}
where $Z''$ is a normalizer at the scale of $O(\alpha^{-1/2})$. In other words, $\mathrm{Var}[\Theta^*]$ is approximately equal to the variance of a normal random variable with variance $\alpha^{-1}$. This completes the proof.

\end{proof}

\medskip

Next, we will show that equation \eqref{eqntheta} exhibits multiplicative ergodicity as described in Theorem \ref{MET}. Consequently, assumption {\bf (H)} holds for our specific choice of $g(I)$.

\medskip


First, we need to prove geometric ergodicity for \eqref{eqntheta}.

\begin{thm}
\label{ergodicitytheta}
Let $\theta_{t}$ be the Markov process given by equation
\eqref{eqntheta}. There exist constants $r > 0$ and positive $C < \infty$ such that, for all bounded measurable function $f$ on $\mathbb{S}^{1}$, we have
$$
  | \mathbb{E}_{\theta_{0}} f( \theta_{t} ) - \pi^{\alpha} (f) | \leq
  C e^{-rt} \,,
  $$
  where $\pi^{\theta}$ is the probability measure on $\mathbb{S}^{1}$
  with probability density function $u^{*}$.
\end{thm}
\begin{proof}
Since $\mathbb{S}^{1}$ is bounded, for any two points $\theta_{0}, \theta_{1} \in \mathbb{S}^{1}$, we define $z(t)$ as a linear interpolation such that $z(0) = \theta_{0}$ and $z(1) = \theta_{1}$. Then, there exists a smooth and bounded
scalar function $U(t)$ that solves the control problem
$$
  \frac{\mathrm{d} z}{\mathrm{d}t} =  I_{2} (1 + 2 \cos z) +
  \sqrt{\gamma I_{2}} \frac{ \mathrm{d} U}{\mathrm{d}t}  
  $$
  for all $t \in (0, 1)$. Note that the probability that the Wiener measure of any
$\epsilon$-tube
$$
  \{ B(t) \, t \in [0, 1]\,|\, \sup_{0 \leq s \leq 1} \| B(s) - U(s)
  \| \leq \epsilon \}
$$
is strictly positive. Thus, the assumption {\bf (A1')} is
satisfied with respect to $T = 1$ and $C = \mathbb{S}^{1}$. We refer to the proof of Lemma 3.4 in \cite{mattingly2002ergodicity} for further
details.

In addition, since $\mathbb{S}^{1}$ is bounded, $V = 1$ is a trivial
Lyapunov function. By Theorem \ref{HMSthm25}, we have
$$
   | \mathbb{E}_{\theta_{0}} f( \theta_{n} ) - \pi^{\alpha} (f) | \leq
  C r_{0}^{n}
$$
for some $r_{0} < 1$. Since equation \eqref{eqntheta} is a stochastic differential equation, the same argument as in the proof of Theorem 3.2 of \cite{mattingly2002ergodicity} extends the result from discrete time to continuous time, completing the proof.

\end{proof}

Let $c(\theta) = -2 (\sin \theta - \gamma)$. We have the proof of
assumption {\bf (H)} for $g(I_{2}) = \sqrt{I_{2}} $.

\begin{thm}
  \label{pfH}
When $\gamma$ is sufficiently small, there exists a $\beta_{1} > 0$ such
that equation
$$
  \mathcal{L}^{I} u = - I
  $$
  has strictly positive solution $u(x)$ on $\mathbb{S}^{1}$ for all $I
  \geq 2$.
\end{thm}
\begin{proof}
  For $I \geq 2$, we have
  
$$
  \mathcal{L}^{I}f = - 4 \beta_{1} I (\sin \theta - \gamma) f + 2 I (1
  + 2 \cos \theta)f' + \frac{\gamma}{2}I f'' = I \mathcal{L}^{1} f \,.
  $$
Hence, it suffices to prove that there exists a strictly positive
solution $u$ of $\mathcal{L}^{1}$ such that
$$
  \mathcal{L}^{1} u = -1 \,.
$$

Define $D = \mathbb{S}^{1}$, and let $\theta_{t}$ be the solution of equation
\eqref{eqntheta} for $I_{2} = 1$. Corollary \ref{thetaexp} implies
that $\mathbb{E}_{\pi}[c] < 0$, where the invariant probability
measure $\pi$ is given by $u^{*}$. Since $\mathbb{S}^{1}$ is bounded, the Lyapunov
condition in Theorem \ref{MET} is automatically satisfied. Therefore, the existence of a positive $u$
is given by Theorem \ref{thm28}.

\end{proof}

Therefore, Theorem \ref{generator} holds for our choice of $g(\theta, I)$.

\section{Recovery of low energy}
\label{lowenergy}
\subsection{Main conclusion based on the result of the reduced system}

The goal of this section is to demonstrate how $I_{2}$ can be
recovered when starting from a very low value. To see this, we
need to extensively work on the ``broken system'' at which $I_{2} = 0$. Noting that $g = 1$ when $I_2 \leq 1$, we have

\begin{align}
  \label{broken}
    \mathrm{d}I_{1}& = \gamma(T_{1} - I_{1}^{3}) \mathrm{d}t + \sqrt{\gamma
  T_{1} I_{1}/2} \mathrm{d}B^{(1)}_{t} \\\nonumber
  \mathrm{d}I_{3} &= 
  \gamma(T_{3} - I_{3}^{3}) \mathrm{d}t + \sqrt{\gamma T_{3}I_{3}/2}
  \mathrm{d}B^{(2)}_{t}\\\nonumber
  \mathrm{d} \theta_{1}& = -[ I_{1} (1 +
  2 \cos \theta_{1}) + 2 I_{3} \cos \theta_{3}] \mathrm{d}t +
  \sqrt{\gamma } \mathrm{d}B^{(3)}_{t}\\\nonumber
  \mathrm{d} \theta
  _{3} &= -[ I_{3} (1 + 2 \cos
  \theta_{3}) + 2 I_{1}\cos \theta_{1}] \mathrm{d}t + \sqrt{\gamma } \mathrm{d}B^{(4)}_{t}\nonumber \,.\end{align}
System \eqref{broken} is defined on $\mathbb{R}^{2} \times
\mathbb{T}^{2} := \Omega^{b}$. We denote its infinitesimal generator by $\mathcal{L}^{b}$. 

The goal of the subsection is to show that
$$
W( \mathbf{x}^b) := I_{2}^{-\alpha} U(I_{1}, I_{3}, \theta_{1}, \theta_{3})
$$

is a Lyapunov 
function for all sufficiently small $I_{2}$, where the pre-factor
function $U$ satisfies a Feynman-Kac type inequality that will be described later. 

Similar to the high-energy case, to make $I_{2}$ increase over time, the value of
$$
  h( \mathbf{x}^{b}) :=  I_{1} \sin \theta_{1} + I_{3} \sin \theta_{3}
$$
must be negative. For small $I_{2}$, the dynamics of $(I_{1}, I_{3},
\theta_{1}, \theta_{3})$ are approximately described by \eqref{broken}. Looking at
system \eqref{broken}, we find that if $T_{3} \gg T_{1}$, we have
$I_{3} \gg I_{1}$ with high probability. Therefore, the phase $\theta_3$ stays in the vicinity of the equilibrium of the $\theta_3$-dynamics, which is $4\pi/3$, with high probability. Notice that the observable $h(\mathbf{x}^b)$ is negative near $4\pi/3$. This fits the setting of the toy example \eqref{toymodel}. One can construct a Lyapunov function by solving a Feynman-Kac equation. As introduced in Section 3.1, we set up the following assumption and prove it later for the case of $T_1 \ll T_3$.

{\bf Assumption (L)}
Let 
$$
h( \mathbf{x}^b) = I_1 \sin \theta_1 + I_3 \sin \theta_3 \,.
$$
There exists a sufficiently large constant $M <
\infty$, a sufficiently small $\alpha > 0$, and a constant $\epsilon > 0$, such that the Feynman-Kac equation with
a Cauchy-Dirichlet boundary condition
\begin{align}
  \label{assumptionL}
  &\mathcal{L}^{b}Q + \alpha h( \mathbf{x}^{b}) Q = p(\mathbf{x}^{b}) \quad
    I_{1} + I_{3} \leq M\\\nonumber
  &Q( \mathbf{x}^{b}) = q(\mathbf{x}^{b}) \quad I_{1} + I_{3} = M \nonumber
\end{align}
admits a positive solution $Q$ for every $p(\mathbf{x}^{b}) < -\epsilon$,
$q(\mathbf{x}^{b}) > 0$.


\medskip

\begin{thm}
  \label{thm51}
  Assume {\bf (L)} holds. Then, there exists a pre-factor function
  $U(I_{1}, I_{3}, \theta_{1}, \theta_{3})$ such that 
$$
  W( \mathbf{x}) = I_{2}^{-\alpha} U
  $$
  is a Lyapunov function of the infinitesimal generator $\mathcal{L}$ of equation \eqref{bigsystem}. More precisely, there exist constants $\alpha, \delta, \zeta > 0$ such that 
  $$
        \mathcal{L}W \leq - \delta W
  $$
  for all $\mathbf{x} \in
  \Omega = \mathbb{R}^3_+ \times \mathbb{T}^2$ satisfying $0 < I_2 < \zeta$.
\end{thm}
\begin{proof}
 For the sake of simplicity, for each $\mathbf{x} \in \Omega$, we
 denote its $(I_{1}, I_{3}, \theta_{1}, \theta_{3})$-component by $\mathbf{x}^{b}$. Apply the infinitesimal generator to $W$, we obtain
\begin{align*}
  \mathcal{L} W &= - \alpha I_{2}^{-\alpha - 1}[-2h( \mathbf{x}^{b})
                  I_{2}]U\\
  +& \left [ \gamma( T_{1} - I_{1}^{3}) \frac{\partial U}{\partial
     I_{1}} + \frac{1}{4}\gamma T_{1} I_{1} \frac{\partial^{2} U}{\partial
     I_{1}^{2}} \right]I_{2}^{-\alpha } + \left [ \gamma( T_{3} - I_{3}^{3}) \frac{\partial U}{\partial
     I_{3}} + \frac{1}{4}\gamma T_{3} I_{3} \frac{\partial^{2} U}{\partial
     I_{3}^{2}} \right]I_{2}^{-\alpha }\\
  +& \left \{ [- I_{1}(1 + 2 \cos \theta_{1}) - 2 I_{3} \cos
  \theta_{3}]\frac{\partial U}{\partial \theta_{1}} +
  \frac{1}{2}\gamma \frac{\partial^{2} U}{\partial
     \theta_{1}^{2}}\right \} I_{2}^{-\alpha}\\
   +& \left \{ [- I_{3}(1 + 2 \cos \theta_{3}) - 2 I_{1} \cos
  \theta_{1}]\frac{\partial U}{\partial \theta_{3}} +
  \frac{1}{2}\gamma \frac{\partial^{2} U}{\partial
      \theta_{3}^{2}}\right \} I_{2}^{-\alpha}\\
  +&I_{2}\left \{ 2 I_{1}(\sin \theta_{1} - \gamma)
     \frac{\partial U}{\partial I_{1}} + 2 I_{3}(\sin \theta_{3} - \gamma)
     \frac{\partial U}{\partial I_{3}} + (1 + 2\cos
     \theta_{1})\frac{\partial U}{\partial \theta_{1}}  +  (1 + 2
     \cos \theta_{3}) \frac{\partial U}{\partial \theta_{3}}\right \}
     I_{2}^{-\alpha}\\
  &:= I_{2}^{-\alpha}[ ( \mathcal{L}^{b}U + 2\alpha h U) + I_{2} A(U)] \,,
\end{align*}
where $A(U)$ denotes the expression in the parentheses in the last term
of $\mathcal{L}W$.

We combine $Q$, the solution to the Feynman-Kac equation that appears in assumption 
{\bf $(L)$}, into the pre-factor $U$ by defining $U = I_{1} + I_{3}$
for $I_{1} + I_{3} \geq M$ and $U = Q$ for suitable functions $p(
\mathbf{x})$ and $q( \mathbf{x})$.

{\bf Case I: Large $I_{1} + I_{3}$.} Assume $I_{1} + I_{3} > M$. Then we have $U( \mathbf{x}) = I_{1} + I_{3}$. We first calculate $\mathcal{L} (I_{2}^{-\alpha}(I_{1} + I_{3}))$:

\begin{align*}
  \mathcal{L}( I_{2}^{-\alpha}(I_{1} + I_{3}) )&= 2I_{2}^{-\alpha} \alpha h
                                    (I_{1} + I_{3}) \\
  +&I_{2}^{-\alpha} \left [ \gamma(T_{1} - I_{1}^{3}) + 
     \gamma (T_{3} - I_{3}^{3}) \right] \\
  + & I_{2}^{-\alpha} \left [ 2 I_{1}I_{2}(\sin \theta_{1} - \gamma)
       + 2 I_{3}I_{2} (\sin \theta_{3} - \gamma) \right ]
\end{align*}

Notice that
\begin{align*}
  (I_{1} + I_{3})^{3} &= I^{3}_{1} +I_{3}^{3} + 3 I_{1}I_{3} (I_{1} +
                        I_{3}) \leq  I_{1}^{3} + I_{3}^{3} + 3
                        (I_{1}+I_{3})(I_{1}^{2} - I_{1}I_{3} +
                        I_{3}^{2})\\
  &= 4(I_{1}^3 + I_{3}^3) \,.
\end{align*}
In addition $h \leq I_{1} + I_{3}$, since $\sin \theta_{1}, \sin \theta_{3} \leq 1 $. Let $x
= I_{1} + I_{3}$, we have
$$
  \mathcal{L}( I_{2}^{-\alpha}U(\mathbf{x}^{b}) ) \leq
  I_{2}^{-\alpha} \left [ \gamma (T_{1} + T_{3}) +  2\alpha x - \gamma\frac{1}{4} x^{3}
     + 4 I_{2}(1 - \gamma) x \right ]\,.
  $$

Therefore, if $x := I_{1} + I_{3}$ is sufficiently large, we have
$\mathcal{L}( I_{2}^{-\alpha}U )  \leq - \frac{1}{5} \gamma
x^{3} I_{2}^{-\alpha}$
because $- \frac{1}{4} \gamma x^{3} $ is the
dominant term for all large $x$. Without loss of generality, assume
that the constant $M$ in assumption {\bf (L)} is sufficiently large so
that $\mathcal{L}( I_{2}^{-\alpha}U )  \leq - \frac{1}{5} \gamma
x^{3}I_{2}^{-\alpha}$ for all $x \geq M$. This gives
$$
  \mathcal{L}W \leq - W \,.
$$
easily.

\medskip

{\bf Case II: Bounded $I_{1} + I_{3}$.}

When $I_{1} + I_{3} \leq M$, we construct the boundary value
$$
  q( \mathbf{x}^{b}) = M
$$
for $\mathbf{x}^{b}$ on the plane $I_{1} + I_{3} = M$ and function
$p(\mathbf{x}^{b})$ that satisfies
$$
  p( \mathbf{x}^{b}) = \left \{ 
\begin{array}[tb]{lll}
 \mathcal{L}^{b}(I_{1} + I_{3})& \mbox{if } & M-1 \leq I_{1} + I_{3}
                                              \leq M \\
  \mbox{negative smooth function} &\mbox{if} &M-2 \leq I_{1} + I_{3}
                                    <M-1\\
  -1 &\mbox{if} & I_{1} + I_{3} < M-2  
\end{array}
\right . \,.
$$

Since
$$
  \mathcal{L}^{b}( I_{1} + I_{3}) = 2 I_{1}I_{2}(\sin \theta_{1} -
  \gamma) + 2 I_{3} I_{2} (\sin \theta_{3} - \gamma) + \gamma (T_{1} +
  T_{3} - I_{1}^{3} - I_{3}^{3} ) < 0
$$
for large $I_{1} + I_{3}$ and bounded $I_{2}$, we can easily
extend $p( \mathbf{x}^{b})$ from the boundary of $\Omega^{M} := \{
I_{1} + I_{3} \leq M \}$ to the interior of $\Omega^{M}$ such that $p(
\mathbf{x}^{b}) < \epsilon$ for all $\mathbf{x}^{b} \in \Omega^{M}$. 

By assuming {\bf (L)}, the Cauchy-Dirichlet problem
\eqref{assumptionL} admits a positive solution $Q( \mathbf{x})$. We
let $U = Q$ if $I_{1} + I_{3} \leq M$. Notice that the construction of
$p( \mathbf{x}^{b})$ makes $Q$ a smooth function.

Then we have
$$
  \mathcal{L}W = I_{2}^{-\alpha} [ (\mathcal{L}^{b} Q + 2\alpha h Q) +
                 I_{2} A(Q) ]  \leq I_{2}^{-\alpha}\left[ - \epsilon  + I_{2} A(Q) 
     \right] \,.
$$
We can choose sufficiently small $I_{2}$ such that
$$
  I_{2}A(Q) < \frac{1}{2} \epsilon 
$$
uniformly for all $I_{1} + I_{3} \leq M$. This is possible because $\{
\mathbf{x}^{b} \,|\, I_{1} + I_{3} \leq M \}$ is a compact set. This gives
$$
  \mathcal{L}W \leq - \frac{1}{2} \epsilon I_{2}^{-\alpha}
$$
for all $I_2 < \zeta$ for a constant $\zeta > 0$. Since $Q$ is a
bounded function in $\Omega^{M}$, there exists a constant $\delta$
such that
$$
  \mathcal{L}W \leq - \delta W \,.
  $$
  This completes the proof. 
\end{proof}

It follows from the proof of Theorem \ref{thm51} that $\zeta \ll 1$. Thus, it remains to estimate the bound of $\mathcal{L}W$ for $I_2 > \zeta$. To achieve this, we need to extend $W(\mathbf{x})$ such that
$$
    W( \mathbf{x}) = \psi(I_2) U
$$
for a monotonically decreasing $C^2$ function $\psi$ that satisfies 
\begin{enumerate}
    \item $\psi(I_2) = I_2^{-\alpha}$ on $[0, \zeta]$, 
    \item $\psi(I_2) = 0$ for $I_2 > 1$,
    \item there exists a constant $M_\zeta$ for which $|\psi'(I_2)| < M_\zeta$ and $|\psi''(I_2)| < M_\zeta$. 
\end{enumerate}
After the extension, we have the following lemma.

\begin{lem}
\label{Wupbd}
    For any $\mathbf{x}$ with the $I_2$ component satisfying $\zeta \leq I_2 \leq 1$, there exists a constant $C_\zeta < \infty$ such that
    $$
        \mathcal{L}W \leq C_\zeta  + C_{\zeta}(I_{1} + I_{3})^{2}\,.
    $$
\end{lem}
\begin{proof}
    Following the same calculations as in the proof of Theorem \ref{thm51} and using the same notations, we have
    $$
    \mathcal{L}W = 2\psi'(I_2)hU + \psi(I_2)[\mathcal{L}^b(U) + I_2 A(U)] \,. 
    $$
    When $I_1 + I_3 < M$, It is clear that $\mathcal{L}W$ is bounded. So we only need to estimate the case of large $I_1 + I_3$. 

    When $x := I_1 + I_3$ is large, we have $U = I_{1} + I_{3}$. The
    same calculation as in the proof of Theorem \ref{thm51} gives
    
$$
  \mathcal{L}^b U + I_2 A(U) \leq  O(x) - \frac{1}{4}\gamma x^3\leq 0
$$
    for all sufficiently large $x$. In addition, $\psi'$ is bounded, and $h(\mathbf{x}^b) < I_1 + I_3 = x$. Thus, we have
    $$
        \mathcal{L}W \leq  2 M_\gamma x^2
    $$
    for all large $x$. Therefore, combining both cases, we can find some finite constant $C_\zeta$ that makes
        $$
        \mathcal{L}W \leq C_\zeta (I_1 + I_3)^{2} + C_\zeta \,.
    $$
    This completes the proof.
\end{proof}

\medskip

Unlike the original system \eqref{bigsystem}, the reduced
system \eqref{broken}, which excludes $I_{2}$, has relatively straightforward
dynamics. In the next few subsections, we will present a detailed study of system \eqref{broken}. This study and its results not only prove the assumption {\bf (L)} for the case of $T_1 \ll T_3$, but also shed light on the mechanism of recovery
from low $I_{2}$-energy. More precisely, the following
justifications will be provided rigorously. 

\begin{itemize}
\item[(a)] In the reduced system \eqref{broken}, there exist constants
 $t_{0}$, $M_{1}$ and $M_{2}$ independent of the initial value, such that for all initial values $(I_{1}^{0}, I_{3}^{0})$
$$
  \mathbb{E}_{I_{1}^{0}}[ I_{1}(t_{0})] < M_{1}  \quad \left ( \mbox{
      resp. }  \mathbb{E}_{I_{3}^{0}}[ I_{3}(t_{0})] < M_{1}  \right )
  $$
  and
$$
  \mathbb{E}_{I_{1}^{0}}\left[ \int_{0}^{t_{0}} I_{1}(s) \mathrm{d}s \right] <
  M_{2} \quad \left ( \mbox{
      resp. }  \mathbb{E}_{I_{3}^{0}}\left[ \int_{0}^{t_{0}} I_{3}(s) \mathrm{d}s\right] < M_{2}  \right ).
  $$
  \item[(b)] System \eqref{broken} admits a unique invariant probability
    measure $\pi^{b}$. The speed of convergence towards $\pi^{b}$ is
    exponentially fast. 
    \item[(c)] For any $T_{3} \gg T_{1}$, function $h(\mathbf{x}^b)$ has negative expectation and bounded variance with respect to $\pi^{b}$.
\item[(d)] System \eqref{broken} satisfies the multiplicative ergodic
  theorem as described in Theorem \ref{MET} with respect to a Lyapunov
  function $\tilde{W} = I_{1} + I_{3}$. 
\end{itemize}

Result (a) is proved in Lemma \ref{I1I3bounded} in Section 5.2. Result (b) is given by Theorem \ref{xbergodicity}. Theorem \ref{thm31} and Theorem \ref{varh} imply (c). And result (d) is shown in the proof of Theorem \ref{pfL}. These results not only imply the assumption {\bf (L)} (when $T_1 \ll T_3$) but indeed  reveal the mechanism of the recovery of low $I_2$. A more heuristic (but non-rigorous) explanation for low-$I_{2}$ recovery
comes from the following arguments:

Let $\hat{W} = I_{2}^{-\alpha}$ be the Lyapunov function. Note that
when $I_{2} \ll 1$, $\mathbf{x}_{t}$ is well approximated by $\mathbf{x}^{b}_{t}$. Then by Gronwall's inequality, there is
$$
  \mathbf{E}_{\mathbf{x}}[ \hat{W}( \mathbf{x}_{t})] = \hat{W}(
  \mathbf{x}_{0}) \mathbb{E}\left[ \exp\int_{0}^{t} \alpha h( \mathbf{x}_{s})
  \mathrm{d}s\right] \approx \hat{W}(
  \mathbf{x}_{0}) \mathbb{E}\left[ \exp\int_{0}^{t} \alpha h(
  \mathbf{x}^{b}_{s}) 
  \mathrm{d}s\right] \,.
  $$

  Thus, the main task is to prove that

  $$
  \mathbb{E}[ \exp\int_{0}^{T} \alpha h( \mathbf{x}^b_{s})
  \mathrm{d}s] < 1
  $$
  for some sufficiently large $T$, which follows from Theorem \ref{MET} using the same eigenvalue
  perturbation argument from \cite{kato2013perturbation}. We remark that a similar averaging approach has been used in results like \cite{hening2018coexistence}.

Unfortunately, the above argument has two technical gaps that are not
easy to close to the best of our knowledge. First in order to apply the result to the case where
$I_{2} \neq 0$, the
difference between $\mathbf{x}_{t}$ and $\mathbf{x}^{b}_{t}$ must be bounded uniformly. Secondly, this argument only shows that the
expected value of the Lyapunov function $\hat{W}$ decreases after a sufficiently long
time. This causes some subtle technical trouble because Theorem
\ref{generator} only gives a power-law decay of $V$. In fact, Theorem
\ref{generator} does not imply
$$
  \mathbb{E}[ V( \mathbf{x}_{t})] \leq V( \mathbf{x}_{0}) - c V(\mathbf{x}_{0})^{1 - 1/\beta_{0}}
$$
because $x^{1 - 1/\beta_{0}}$ is a concave function. Instead, we obtain
$$
  \mathbb{E}[ \left (V( \mathbf{x}_{t})^{1/\beta_{0}} + c t \right
  )^{\beta_{0}}] \leq V( \mathbf{x}_{0}) \,.
$$
This subtle difference causes the Lyapunov function $V$ for the
infinitesimal generator to be incompatible with the Lyapunov function $\hat{W}$
over a finite time interval. In contrast, the Feynman-Kac-Lyapunov method we use provides a Lyapunov function with respect to the infinitesimal generator, effectively avoiding these two technical difficulties.

\subsection{Results for the reduced system Part I: Boundedness of $I_{1}$ and $I_{3}$}
We will first verify item (a) of the list above, i.e., $I_{1}$ (resp. $I_{3}$) in system
\eqref{broken} can return
to a bounded set within a finite time regardless of the initial value. More
precisely, we have
\begin{lem}
  \label{I1I3bounded}
  There exist a constant $M < \infty$ and a time $t_{0} > 0$ such that
$$
  \mathbb{E}_{\mathbf{x}^{b}_{0}}[ I_{1}(t_{0})] < M \quad \mbox{ and
  } \quad \mathbb{E}_{\mathbf{x}^{b}_{0}}\left [ \int_{0}^{t_{0}}
    I_{1}(s) \mathrm{d}s \right ] < M 
  $$
  uniformly for all initial values. The same argument holds for $I_{3}$.
\end{lem}
\begin{proof}
Define $\tilde{V}( \mathbf{x}^b ) \ = I_{1}$, and let $\mathcal{L}^{b}$ be the infinitesimal generator
of equation \eqref{broken}. Then by Jensen's inequality, we obtain
$$
  \mathcal{L}^{b} \tilde{V} = \gamma (T_{1} - \tilde{V}^{3}) \,.
$$

Hence by Dynkin's formula and Fubini's theorem (because
$\mathcal{L}^{b} \tilde{V} < \gamma T_{1}$, the expectation of
$V(\mathbf{x}^{b}_{t})$ is bounded for all $t$), 
\begin{align*}
  \mathbb{E}_{\mathbf{x}^{b}_{0}}[\tilde{V} (t)] &= \tilde{V} (0) +
  \mathbb{E}_{\mathbf{x}^{b}_{0}}\left[\int_{0}^{t} (\gamma T_1 - \tilde{V} ^{3}(s))
                                          \mathrm{d}s\right]\\
  &\leq \tilde{V} (0) + \int_{0}^{t} (\gamma T_{1} -
    (\mathbb{E}_{\mathbf{x}^{b}_{0}}[\tilde{V} (s)])^{3}) \mathrm{d}s \,.
\end{align*}
Let $\phi(t) := \mathbb{E}_{\mathbf{x}^{b}_{0}}[\tilde{V}(t)]$. When $\phi(t) >
\sqrt[3]{2 \gamma T_{1}}$, we have
$$
  \phi(t) \leq \phi(0) - \int_{0}^{t} \frac{1}{2} \phi^{3}(s) \mathrm{d}s \,.
  $$
  The generalized Gronwall's inequality in \cite{stachurska1971nonlinear} gives
  
$$
  \phi(t) \leq \phi(0)\left ( 1 + \phi(0)^{2} t \right)^{-1/2} =
  \frac{1}{\sqrt{\phi(0)^{-2} + t }} 
$$
for all $t > 0$ such that $\phi(t) > \sqrt[3]{2 \gamma T_1}$. Otherwise, we obtain the trivial bound $\phi(t) \leq \phi(0) + \gamma T_1 t$. Therefore, setting $t_0$ = 1, we obtain $\phi(t_0) \leq M$ for some constant $M$ independent of $\phi(0) = \tilde{V}( \mathbf{x}^b_0)$. 

The second bound follows from the fact that
$$
  \mathbb{E}_{\mathbf{x}^{b}_{0}}\left[ \int_{0}^{t} I_{1}(s) \mathrm{d}s\right]
  = \int_{0}^{t} \phi(s) \mathrm{d}s \,,
  $$
  which is bounded by 
  
$$
  \int_{0}^{t} \frac{1}{\sqrt{\phi(0)^{-2} + s }} \mathrm{d}s =
  2( \phi(0)^{-2} + t)^{1/2} - 2\phi(0)^{-1} \leq 2( \phi(0)^{-2} + t)^{1/2} 
$$
whenever $\phi(t) >  \sqrt[3]{2 \gamma T_{1}}$. Again, let $t_0 = 1$. It follows that the integral is uniformly bounded by some constant regardless of $\phi(0)$. Therefore, $\mathbb{E}_{\mathbf{x}^{b}_{0}}\left[ \int_{0}^{t_0} I_{1}(s) \mathrm{d}s\right]$ is also bounded by some constant $M$, independent of $\phi(0) = \tilde{V}( \mathbf{x}^b_0)$. 
\end{proof}

\subsection{Results for the reduced system Part II: Stochastic stability}
In this subsection, we will show (b) in the list at the end of Subsection \ref{lowenergy}, namely that system \eqref{broken} admits a
unique invariant probability measure $\pi^{b}$ and that the speed of
convergence to $\pi^{b}$ is exponential. By Theorem \ref{HMSthm25}, one needs
to check the Lyapunov condition and the minorization condition.

Let $\tilde{W}(I_{1}, I_{3}, \theta_{1}, \theta_{3}) = 1 + I_{1} + I_{3}$. Since
$\mathbb{T}^{2}$ is compact, the set $\{ \tilde{W} \leq c \}$ is also compact for any
finite $c$. In addition, the following lemma follows easily. 

\begin{lem}
$$
  \mathcal{L}^{b} \tilde{W} \leq - \frac{1}{2}\gamma \tilde{W} +
  (4\sqrt{2} + 1 + T_{1} + T_{3})\gamma \,.
$$
\end{lem}
\begin{proof}
  Applying It\^{o}'s formula, we have
 
$$
  \mathcal{L}^{b}\tilde{W} =  \gamma (T_{1} - I_{1}^{3}) + \gamma (T_{3} -
  I_{3}^{3}) = - \gamma (\tilde{W} -1) (I_{1}^{2} + I_{3}^{2} - I_{1}I_{3})  +
  \gamma (T_{1} + T_{3})\,.
  $$
If $I_{1}^{2} + I_{3}^{2} - I_{1}I_{3} \geq 4$ we have
$\mathcal{L}^{b} \tilde{W} \leq - \frac{1}{2}\gamma \tilde{W} + \gamma (T_{1} +
T_{3})$. Because $I_{1}^{2} + I_{3}^{2} - I_{1}I_{3} \geq 4$ means
$I_{1} + I_{3} \geq 2$, or $\tilde{W} - 1 \geq \tilde{W}/2$. 

Otherwise, we note that
$$
  I_{1}^{2} + I_{3}^{2} - I_{1}I_{3} = \frac{1}{2}I_{1}^{2} +
  \frac{1}{2}I_{3}^{2} + \frac{1}{2}(I_{1} - I_{3})^{2} \,,
  $$
  we must have $I_{1} \leq 2\sqrt{2}$ and $I_{3} \leq 2\sqrt{2}$. Hence,
  $\tilde{W}  \leq 4 \sqrt{2} + 1 $. Combining above two cases, we have
  
$$
  \mathcal{L}^{b} \tilde{W} \leq - \frac{1}{2}\gamma \tilde{W} + (4\sqrt{2} + 1 + T_{1} + T_{3})\gamma \,.
$$
  This completes the proof. 
\end{proof}

It remains to prove the minorization condition, i.e., Assumption {\bf
  (A1')} in the probability preliminary. 

\begin{lem}
System \eqref{broken} satisfies Assumption {\bf (A1')} for any compact
set $C$.
\end{lem}
\begin{proof}
The proof is similar to that of Lemma 3.4 of \cite{mattingly2002ergodicity}.
  
  Since the noise term in system \eqref{broken} is non-degenerate, it
  is easy to see that system \eqref{broken} satisfies the H\"ormander's
  condition. Thus, the transition kernel possess a jointly continuous
  density function by \cite{stroock2008partial}.

  In addition, let $\mathbf{x}_{0}$ and $\mathbf{x}_{1}$ be two points
  in the compact set. Let $\mathbf{z}(t)$ be a linear interpolation
  such that $\mathbf{z}(0) = \mathbf{x}_{0}$ and $\mathrm{z}(1) =
  \mathbf{x}_{1}$. Then we can find a smooth $\mathcal U(t) \in \mathbb{R}^{4}$
  solving the control problem 
$$
 \frac{ \mathrm{d} \mathbf{z}}{ \mathrm{d}t} = F( \mathbf{z}) + \Sigma
 \frac{ \mathrm{d} \mathcal U}{\mathrm{d}t}
$$
for all $t \in (0, 1)$, where $F$ and $\Sigma$ are the vector field
and the coefficient matrix of the noise term in system \eqref{broken}
respectively. Note that the Wiener measure of any
$\epsilon$-tube
$$
  \left\{ B(t) \, t \in [0, 1]\,|\, \sup_{0 \leq s \leq 1} \| B(s) - U(s)
  \| \leq \epsilon \right\}
$$
is strictly positive. Hence, Assumption {\bf (A1')} is
satisfied. We refer to the proof of Lemma 3.4 in \cite{mattingly2002ergodicity} for further
details. 
  
\end{proof}

Recall that Assumption {\bf (A1')} implies Assumption {\bf (A1)}. Therefore, by Theorem \ref{HMSthm25}, we have the geometric ergodicity of $\mathbf{x}^b_t$:

\begin{thm}
\label{xbergodicity}
    There exist constants $r > 0$ and $C_{\tilde{W}}< \infty$ such that for all measurable function $f$ with $|f| \leq \tilde{W}$, we have
    $$
    | \mathbb{E}_{\mathbf{x}^b_0} f( \mathbf{x}^b_t) - \pi^b ( f) | \leq C_{\tilde{W}} e^{-rt} \tilde{W}( \mathbf{x}^b_0 )
    $$
\end{thm}
\begin{proof}
Since conditions {\bf (A1)} and {\bf (A2)} are both satisfied, by Theorem \ref{HMSthm25} we have
$$
| \mathbb{E}_{\mathbf{x}^b_0} f( \mathbf{x}^b_{nT} ) - \pi^b (f) | \leq C_W r_0^n \tilde{W}( \mathbf{x}^b_0 )
$$
for some $r_0 < 1$. Since, in addition, $\mathbf{x}^b_t$ is a stochastic
differential equation, the same argument as in the proof of Theorem 3.2 in
\cite{mattingly2002ergodicity} extends the result from discrete times $nT$  to
continuous time and completes the proof. 
\end{proof}

\subsection{Results for the reduced system Part III: expectation and
  variance of $h$}
 The goal of this subsection is to prove (c) in the list at the end of Subsection \ref{lowenergy}. We start with  the following theorem.
\begin{thm}
  \label{thm31}
System \eqref{broken} admits a unique invariant probability measure
$\pi^{b}$. In addition, for any $\gamma > 0$ and $T_{1} > 0$, there
exists a sufficiently large $T_{3}^{*}$ such that 
$$
  \int_{\Omega^{b}} (I_{1} \sin \theta_{1} + I_{3} \sin \theta_{3})
  \pi^{b}(\mathrm{d} \mathbf{x} ) < -\frac{1}{4} \sqrt[3]{\frac{3T_{3}}{4}}\frac{\Gamma(5/3)}{\Gamma(4/3)}
  $$
for all $T_{3} > T_{3}^{*}$.
\end{thm}

Without $I_{2}$, it is easy to see that $I_{1}$ and $I_{3}$ are now
independent variables. We can explicitly solve the marginal
distribution of $I_{1}$ and $I_{3}$.

\begin{lem}
  \label{invI1I3}
The marginal probability density function of system \eqref{broken} for
$I_{1}$ (resp. $I_{3}$) is
$$
  u_{1} = \frac{4 \sqrt[3]{4}}{\sqrt[3]{3} \Gamma(\frac{4}{3}) T_{1}^{4/3}}
  x^{3}e^{-4 x^{3}/3 T_{1}}
  $$
  (resp.
$$
 u_{3} =  \frac{4 \sqrt[3]{4}}{\sqrt[3]{3} \Gamma(\frac{4}{3}) T_{3}^{4/3}}
  x^{3}e^{-4 x^{3}/3 T_{3}}
  $$
  )
  
\end{lem}
\begin{proof}
Since $I_{1}$ is independent of other variables, it is sufficient to
consider the 1D stochastic differential equation
$$
  \mathrm{d} I_{1} = \gamma (T_{1} - I_{1}^{3}) \mathrm{d}t +
  \sqrt{\gamma T_{1} I_{1}/2} \mathrm{d} B_{t} \,.
  $$

  The invariant probability density function of this equation,
  denoted by $u(I_{1})$, must satisfy the stationary Fokker-Planck equation
  \begin{equation}
    \label{FPEI1}
 - \frac{\partial}{\partial I_{1}}( \gamma (T_{1} - I_{1}^{3}) u) +
 \frac{1}{4}\gamma T_{1} \frac{\partial ^{2}}{\partial
   I_{1}^{2}}(I_{1} u )  = 0\,. 
\end{equation}

Simplifying this equation, one obtains
$$
  \frac{\partial}{\partial I_{1}}[ - (T_{1} - I_{1}^{3}) u + \frac{1}{4} T_{1}
  \frac{\partial}{\partial I_{1}}(I_{1} u)] = 0 \,.
$$

Now let
$$
  u = I_{1}^{3} e^{-4 I_{1}^{3}/3 T_{1}} \,.
$$
We have
$$
  (T_{1} - I_{1}^{3}) u = T_{1}I_{1}^{3}e^{-4 I_{1}^{3}/3 T_{1}} - I_{1}^{6}e^{-4 I_{1}^{3}/3 T_{1}}
  $$
  and
  
$$
  \frac{1}{4} T_{1}\frac{\partial}{\partial I_{1}}(I_{1} u) =
  \frac{1}{4}T_{1}( 4 I_{1}^{3}e^{-4 I_{1}^{3}/3 T_{1}} -
  \frac{4}{T_{1}}I_{1}^{6}e^{-4 I_{1}^{3}/3 T_{1}}) = T_{1}I_{1}^{3}e^{-4 I_{1}^{3}/3 T_{1}} - I_{1}^{6}e^{-4 I_{1}^{3}/3 T_{1}}
$$
Hence $u$ solves the Fokker-Planck equation \eqref{FPEI1}. It is easy
to see that $u$ is proportional to a generalized Gamma density
function. Normalizing $u$, we have
$$
  u(x) = \frac{4 \sqrt[3]{4}}{\sqrt[3]{3} \Gamma(\frac{4}{3}) T_{1}^{4/3}}
  x^{3}e^{-4 x^{3}/3 T_{1}} \,.
  $$

  The calculation for the marginal distribution of $I_{3}$ is identical. 

\end{proof}

For each pair $(I_{1}, I_{3})$, the system
\begin{align}
  \label{theta13}
    \mathrm{d} \theta_{1}& = -[ I_{1} (1 +
  2 \cos \theta_{1}) + 2 I_{3} \cos \theta_{3}] \mathrm{d}t +
  \sqrt{\gamma} \mathrm{d}B^{(3)}_{t}\\\nonumber
  \mathrm{d} \theta
  _{3} &= -[ I_{3} (1 + 2 \cos
  \theta_{3}) + 2 I_{1}\cos \theta_{1}] \mathrm{d}t + \sqrt{\gamma} \mathrm{d}B^{(4)}_{t}\nonumber \,
\end{align}
is a 2D stochastic differential equation on $\mathbb{T}^{2}$. Since
the noise term is non-degenerate, system
\eqref{theta13} clearly admits an invariant probability measure, denoted by
$\pi^{I_{1}, I_{3}}$. Let the probability density function of
$\pi^{I_{1}, I_{3}}$ be $u_{I_{1}, I_{3}}( \theta_{1},
\theta_{3})$. Thus, the unique invariant probability density
function of system \eqref{broken} takes the form
$$
 u(I_{1}, I_{3}, \theta_{1}, \theta_{3}) =  u_{1}(I_{1})u_{3}(I_{3}) u_{I_{1}, I_{3}}( \theta_{1},
\theta_{3}) \,,
$$
where $u_{1}$ and $u_{3}$ are defined in Lemma \ref{invI1I3}. The
probability density of $\pi^{I_{1}I_{3}}$ cannot be explicitly expressed. However, when $I_{1} = 0$, the marginal probability density
function with respect to  $\theta_{3}$ can be explicitly calculated in the next lemma.

\begin{lem}
  \label{theta3inv}
  The invariant probability measure of equation
  
\begin{equation}
  \label{theta3}
  \mathrm{d} \theta
  _{3} = -I_{3} (1 + 2 \cos
  \theta_{3}) \mathrm{d}t + \sqrt{\gamma} \mathrm{d}B_{t}
\end{equation}
has a probability density function
$$
  u^{*}(\theta) = \frac{1}{Z} \left ( \frac{(e^{2\pi \alpha} -
      1)e^{-\alpha (\theta + 2 \sin \theta)}}{\int_{0}^{2\pi}
      e^{\alpha(r + 2 \sin r) }\mathrm{d}r}\int_{0}^{\theta}
    e^{\alpha(r + 2 \sin r)} \mathrm{d}s + e^{-\alpha(\theta + 2 \sin
      \theta)}\right ) \,,
  $$
  where $\alpha = \frac{2 I_{3}}{\gamma}$, and $Z$ is a normalizer. 
\end{lem}
\begin{proof}
  The stationary Fokker-Planck equation corresponding to equation \eqref{theta3} is
$$
 \frac{\mathrm{d}}{\mathrm{d} \theta} ( I_{3} (1 + 2 \cos \theta) u) +
 \frac{1}{2}\gamma T_{3} \frac{\mathrm{d}^{2}}{\mathrm{d} \theta^{2}} u =
 0 \,.
$$

Let $\alpha = \frac{4 I_{3}}{\gamma} $. After some simplifications, we obtain
$$
  \frac{\mathrm{d}}{\mathrm{d} \theta} ( (1 + 2 \cos \theta) u) +
 \alpha^{-1} \frac{\mathrm{d}^{2}}{\mathrm{d} \theta^{2}} u =
 0 \,.
$$
Or
$$
  \frac{\mathrm{d}}{\mathrm{d} \theta} ( (1 + 2 \cos \theta) u +
 \alpha^{-1} \frac{\mathrm{d}}{\mathrm{d} \theta} u) =
 0 \,,
$$
which implies
$$
  u' + \alpha(1 + 2 \cos \theta )u = C_{0}
 $$
  for some constant $C_{0}$.

  Together with the periodic boundary condition $u(0) = u(2 \pi)$, we have
  
$$
  u(\theta) = C_{0} e^{-\alpha (\theta + 2 \sin
    \theta)}\int_{0}^{\theta} e^{\alpha(r + 2 \sin r)} \mathrm{d}r +
  C_{1} e^{-\alpha (\theta + 2 \sin \theta)}
$$
for a free constant $C_{1}$. Without loss of generality, let $C_{1} =
1$. (Because $u$ will be renormalized later.) The periodic boundary
condition gives
$$
  C_{0} = \frac{e^{2 \pi \alpha} - 1}{\int_{0}^{2\pi} e^{\alpha ( r +
      2 \sin r) }\mathrm{d}r} \,.
$$
 The proof is completed by normalizing $u$.
\end{proof}

For all $I_{3} \gg 1$, $u^{*}(\theta)$ concentrates at the unique stable
equilibrium of $\theta' = - (1 + 2 \cos \theta)$, which is $\frac{4 \pi}{3}$. Therefore, we have the bound
$$
  \int_{0}^{2\pi} \sin \theta u^{*}(\theta) \mathrm{d} \theta < - \frac{3}{4}
$$
for all sufficiently large values of $I_{3}$.

Note that the concentration of $u^{*}(\theta)$ is similar to that in Lemmata \ref{concentrationtheta} and \ref{vartheta}. Most of
the probability density of $u^{*}$ is concentrated around $\frac{4
\pi}{3}$. Again, let $\Theta^{*}$ denote a random variable on
$\mathbb{S}^{1}$ with probability density function $u^{*}$. Using the same calculation as in Lemma \ref{vartheta}, we obtain the estimate
\begin{equation}
  \label{vartheta1}
  \mathrm{Var}[ \Theta^{*}] = O(I_{3}^{-1}) \,.
\end{equation}

Let $\mu^{*}$ denote the invariant probability measure corresponding to the probability density function $u(\theta)$. When $I_{1} \ll I_{3}$, the dynamics of equation \eqref{theta13} form only a small perturbation of those of \eqref{theta3}. Therefore, we apply the level set method introduced in \cite{huang2015integral} to estimate the concentration of the stationary distribution.

The basic idea of the level set method is that if the infinitesimal generator $\mathcal{L}$ of a stochastic differential equation admits a Lyapunov function $U$ such that $\mathcal{L} U \leq - \gamma < 0$ in a given set, then the probability density of the level set of $U$ must decay at least exponentially with respect to the value at the level set. Let $\rho$ be a constant. Denote the sub-level set of $U$ by $\Omega_{\rho} := \{ U(x) \leq \rho \}$. Then, the probability measure of $\Omega_\rho^c$ has an exponentially small upper bound. Similarly, if $\mathcal{L}$ admits an anti-Lyapunov function such that $\mathcal{L} U \geq \gamma > 0$, then the
probability of $\Omega_{\rho}$ has an exponentially small upper bound. We refer to Theorems A and B in \cite{huang2015integral} for full details.

\begin{lem}
  \label{smallI1}
  If $I_{1}^{2}/I_{3} < \gamma/11$ and $I_{3}/\gamma > 500$, then
  
$$
  \mathbb{E}_{\pi^{I_{1}, I_{3}}}[I_{3} \sin \theta_{3} + I_{1} \sin
  \theta_{1}] \leq - 0.48 I_{3}
$$
\end{lem}
\begin{proof}
Let $\eta$ be a small positive number such that $\eta \ll I_{1}$. Define a Lyapunov function $U_{3}$ on $\mathbb{T}^{2}$ such that  
$$
U_{3}(\theta_{1}, \theta_{3}) = \left \{ 
\begin{array}[tb]{lcl}
 (\theta_{3} - \frac{4\pi}{3})^{2}&\mbox{ if } & \theta_{3} \in [\frac{4
                                                 \pi}{3} + \eta, 2
                                                 \pi] \\
  (\theta_{3} + \frac{2 \pi }{3})^{2} &\mbox{ if } &\theta_{3} \in [0, \frac{2 \pi}{3}]\\
  4 (\theta_{3} - \frac{4 \pi }{3})^{2} &\mbox{ if } &\theta_{3} \in
                                          [\frac{2\pi}{3}, \frac{4\pi}{3} - \eta]\\
\phi( \theta_{3})& \mbox{ if }& \theta_{3} \in [\frac{4\pi}{3} - \eta,
  \frac{4\pi}{3} + \eta]  \,,
\end{array}
\right .
$$
where $\phi(\theta_{3})$ is an interpolation function ensuring that 
$U_{3}$ has continuous second order derivatives and satisfies
$\phi'' \leq 8$, $\phi'(\theta) \geq 0$ for $\theta \geq 4\pi/3$, and $\phi'(\theta)
\leq 0$ for $\theta \leq 4\pi/3$. 

Note that the assumption implies $I_{1}^{2}/I_{3}^{2} < 1/5500$, which implies
$I_{3} > 5500^{-1/2} I_{1}$. Therefore,  if $\theta_{3} \in [2
\pi /3 + 0.08, 4\pi/3 - 0.08]$, we have
$$
  I_{3}(1 + 2 \cos \theta_{3}) + 2 I_{1} \cos \theta_{1} < -0.1 I_{3}\,.
$$
If $\theta_{3} \in [4 \pi /3 + 0.16, 2\pi] \cup [0, 2 \pi/3 - 0.16]$,
we have
$$
  I_{3}(1 + 2 \cos \theta_{3}) + 2 I_{1} \cos \theta_{1} > 0.25 I_{3}\,.
  $$
Define $\mathcal{L}^{1,3}$ as the infinitesimal generator of equation
\eqref{theta13}.  
$$
  \mathcal{L}^{1,3} U_{3} = -2 [ I_{3} (1 + 2 \cos \theta_{3}) + 2 I_{1}
  \cos \theta_{1}] U_{3}' + \frac{1}{2} \gamma U_{3}'' \,.
$$
Hence
$$
  \mathcal{L}^{1.3} U_{3} \leq - 0.25 \times 2 \times 0.16
  I_{3} + \gamma < -0.08 I_{3} + \gamma
  $$
  for $\theta_{3} \in [4 \pi /3 + 0.16, 2\pi] \cup [0, 2 \pi/3 -
  0.16]$ and
$$
  \mathcal{L}^{1.3} U_{3} \leq -0.1 \times 8 \times 0.08
  I_{3} + \gamma < -0.064 I_{3} + \gamma
  $$  
  for $\theta_{3} \in [2 \pi /3 + 0.08, 4\pi/3 - 0.08]$. Since $I_{3}
  > 500 \gamma$, it follows that
  $\mathcal{L}^{1,3}U_{3} \leq - 0.063 I_{3}$ for $\theta_{3} \in [2
  \pi /3 + 0.08, 4\pi/3 - 0.08] \cup [4 \pi /3 + 0.16, 2\pi] \cup [0, 2 \pi/3 -
  0.16]$.

 Let $\mathcal{U} = [0, 2 \pi /3 - 0.08] \cup [ 2 \pi /3 + 0.16, 2
 \pi]$. Let $\mu$ be the invariant measure on $[0, 2\pi]^{2}$
 satisfying $\mu( \mathcal{U}) = 1$. Let $\rho_{m} = 4 \times 0.08^{2} =
 0.0256$ and $\rho_{M} = 4\times (2 \pi /3 - 0.08 - 4 \pi /3)^{2} = 16.2312$
 be the essential lower and upper bound of $U_{3}$ respectively. Then
 for each $\rho \in [\rho_{m}, \rho_{M}]$, Theorem A(b) of
 \cite{huang2015integral} implies that
 
\begin{equation}
  \label{levelset}
  \mu( \mathcal{U} \setminus \Omega_{\rho}) \leq \exp \{ - 0.063 I_{3}
  \int_{\rho_{m}}^{\rho}\frac{1}{H(t)} \mathrm{d}t\} \,,
\end{equation}
where $\Omega_{\rho} := \{ (\theta_{1}, \theta_{3})\,|\,
U_{3}(\theta_{1}, \theta_{3}) \leq \rho \}$ is the $\rho$-sublevel
set, and $H(t)$ satisfies
$$
  \gamma (\frac{\partial U_{3}}{\partial \theta_{3}}(\theta_{1}, \theta_{3}))^{2} \leq H(t)
$$
 for all $(\theta_{1}, \theta_{3}) \in U_{3}^{-1}(t), t \in [0,
 \rho_{M}]$. A straightforward calculation shows that $H(t) = 16 \gamma t$ satisfies
 this requirement. 

 Therefore, when $\rho = (4\pi/3 - 0.2)^{2} = 15.9104$, we have
 
\begin{align*}
 & \mu([2 \pi/3 + 0.08, 2 \pi /3 + 0.1] \cup [2 \pi /3 - 0.2, 2\pi/3 -
   0.16] ) := \mu(A_{1}) \\
  \leq & e^{-0.063 I_{3}\int_{0.0256}^{15.9104} (16
    \gamma t)^{-1} \mathrm{d}t} < e^{-0.025 I_{3}/\gamma} \,.
\end{align*}
When $\rho = 4\times (4 \pi /3 - 7 \pi /6)^{2} = \pi^{2}/9 \approx 1.0966$, we have
\begin{align*}
 & \mu([2\pi/3 + 0.08, 7\pi/6]\cup [5\pi/3, 2\pi] \cup [0, 2\pi/3 -
   0.16]):= \mu(A_{2})\\
  \leq & e^{-0.063 I_{3}\int_{0.0259}^{1.0966} (16
    \gamma t)^{-1} \mathrm{d}t} < e^{-0.0148 I_{3}/\gamma} \,.
 \end{align*}
  Since $I_{3}/\gamma > 500$, we have $\mu(A_{1}) < 3.8 \times
  10^{-6}$ and $\mu_{A_{2}} < 6.2 \times 10^{-4}$. We refer to Figure \ref{fig:circ} for an illustration of sets $A_1$, $A_2$ and the dynamics of equation \eqref{theta3} on $\mathbb{S}^1$.
  
  \medskip

\begin{figure}
    \centering
    \includegraphics[width=0.5\linewidth]{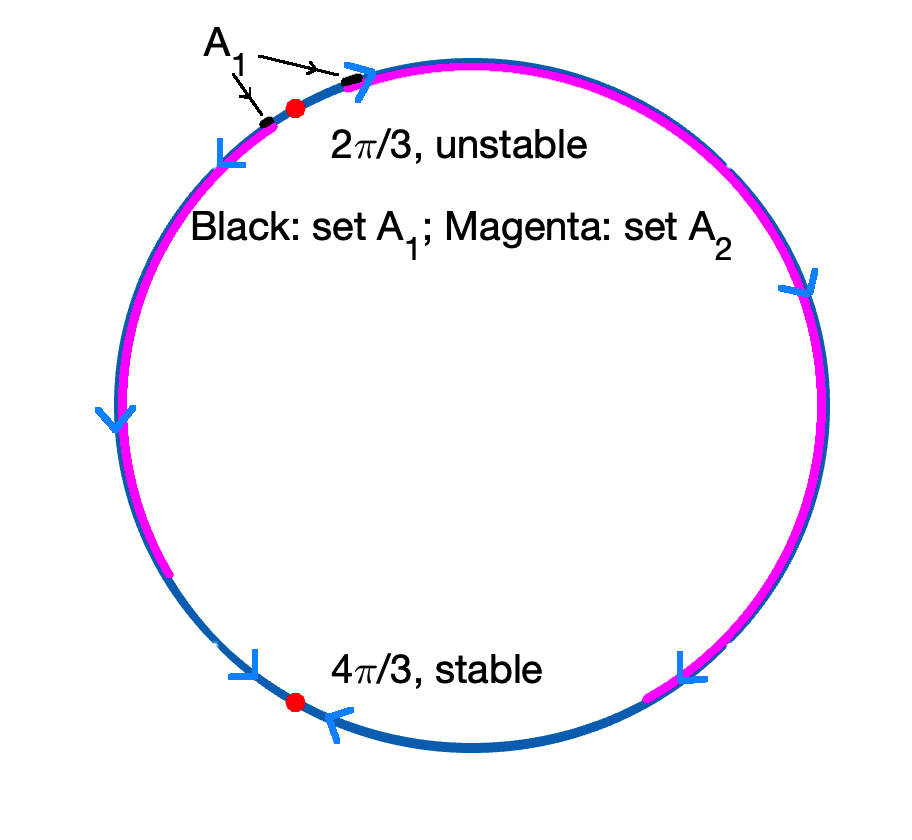}
    \caption{Illustration of the dynamics of equation \eqref{theta3} and sets $A_1$, $A_2$ used in the proof. }
    \label{fig:circ}
\end{figure}

It remains to estimate $\mu([2\pi/3 - 0.16, 2\pi/3 + 0.08])$. Since
$I_{1} \ll I_{3}$, our aim is to show that equation \eqref{theta13}
admits an anti-Lyapunov function in a small neighborhood of $2\pi/3$
so that the stationary distribution cannot concentrate at $2\pi/3$.

Let $\eta \ll I_{1}$ be a small number. Define an anti-Lyapunov function $W_{3}$ on $\mathbb{T}^{2}$ such that  
$$
W_{3}(\theta_{1}, \theta_{3}) = \left \{ 
\begin{array}[tb]{lcl}
 4(\theta_{3} - \frac{2\pi}{3})^{2}&\mbox{ if } & \theta_{3} \in [\frac{2
                                                 \pi}{3} + \eta,
                                                 \frac{2\pi}{3} + 0.1] \\
  (\theta_{3} - \frac{2 \pi }{3})^{2} &\mbox{ if } &\theta_{3} \in
                                                     [\frac{2
                                                     \pi}{3}-0.2,
                                                     \frac{2\pi}{3} - \eta]\\
\phi( \theta_{3})& \mbox{ if }& \theta_{3} \in [\frac{2\pi}{3} - \eta,
  \frac{2\pi}{3} + \eta]  \,,
\end{array}
\right .
$$
where $\phi(\theta_{3})$ is an interpolation function ensuring that
$W_{3}$ has continuous second order derivatives and satisfies
$\phi'' \leq 8$, $\phi'(\theta) \geq 0$ for $\theta \geq 2\pi/3$, and $\phi'(\theta)
\leq 0$ for $\theta \leq 2\pi/3$. Notice that the derivative of $-(1 + 2
\cos \theta_{3})$ is at least $1.49$ in the interval $[2\pi/3 - 0.2, 2\pi/3 + 0.1]$, we have
\begin{align*}
  \mathcal{L}^{1,3} W_{3} &= - [I_{3}(1 + 2 \cos \theta_{3}) + 2 I_{1}
  \cos \theta_{1}] \frac{\partial W_{3}}{\partial \theta_{3}} +
  \frac{1}{2}\gamma \frac{\partial^{2} W_{3}}{\partial
                            \theta_{3}^{2}}\\
  \geq&-4 I_{1}|\theta_{3} - 2\pi/3| + \gamma + 2.98 I_{3} |\theta_{3}
        - 2 \pi/3|^{2}\\
   \geq & \gamma - 1.35 I_{1}^{2}/I_{3}     
\end{align*}
for $\theta_{3} \in [2\pi/3 -0.2, 2\pi/3 - \eta]$, where the first
inequality follows from a Taylor expansion of $1 + 2 \cos \theta_{3}$
and the second inequality follows from the minimum of the quadratic
function $-8I_{1}x + 5.98I_{3} x^{2}$. Similar calculation in the
interval $[2\pi/3 + \eta, 2\pi/3 + 0.1]$ gives
$$
  \mathcal{L}^{1,3} W_{3}  \geq 4 (\gamma - 1.35 I_{1}^{2}/I_{3}) \,.
$$
Since $\eta \ll I_{1}$, in $[2\pi/3 - \eta, 2\pi/3 + \eta]$ we can
easily have $\mathcal{L}^{1,3} W_{3} \geq \gamma/2$. Therefore,
because $I_{1}^{2}/I_{3} < \gamma/11$, $W_{3}$ is an anti-Lyapunov
function with $\mathcal{L}^{1,3} W_{3} \geq \gamma/2$.

Let $\Omega^{w}_{\rho} := \{ (\theta_{1}, \theta_{3}) \,|\, W_{3} \leq
\rho \}$ be the sublevel set of $W_{3}$. The upper bound
$$
  \gamma (\frac{\partial W_{3}}{\partial \theta_{3}}(\theta_{1}, \theta_{3}))^{2} \leq H(t)
$$
is still $16 \gamma t$. Therefore, by Theorem B(a) of \cite{huang2015integral}, we
have
$$
  \mu( \Omega^{w}_{0.0256} )\leq \frac{1}{e^{\frac{1}{32}
      \int_{0.0256}^{0.04} t^{-1} \mathrm{d}t} - 1}
  \mu(\Omega^{w}_{0.04}\setminus \Omega^{w}_{0.0256} ) \leq 71.21
  \mu(\Omega^{w}_{0.04}\setminus \Omega^{w}_{0.0256} ) \,.
$$
Notice that $\Omega^{w}_{0.0256} = [2\pi/3 - 0.16, 2\pi/3 + 0.08]$ and
$\Omega^{w}_{0.04}\setminus \Omega^{w}_{0.0256}  = [2 \pi/3 + 0.08, 2 \pi /3 + 0.1] \cup [2 \pi /3 - 0.2, 2\pi/3 -
0.16] = A_{1}$. Therefore, we have
$$
  \mu( [2\pi/3 - 0.16, 2\pi/3 + 0.08]) \leq 71.21 \times \mu(A_{1})
  \leq 2.8\times 10^{-4} \,.
$$
In other words, $\mu([0, 2 \pi]^{2}) \leq 1 + 2.8 \times
10^{-4}$. Hence, the difference between $\mu$ and the normalized
probability measure, which is $\pi^{I_{1}, I_{3}}$, is negligible. A straightforward calculation shows that
$$
  \pi^{I_{1}, I_{3}}( [7\pi/6, 5\pi/3]) = 1 -
  \pi^{I_{1}, I_{3}}(\Omega^{w}_{0.0256}) - \pi^{I_{1}, I_{3}}(A_{2}) \geq 0.9999 \,.
  $$
Therefore,
 
\begin{align*}
  \mathbb{E}_{\pi^{I_{1}, I_{3}}}[ I_{1} \sin \theta_{1} + I_{3}\sin
  \theta_{3}] &\leq I_{1} + I_{3}(1 - \pi^{I_{1}, I_{3}}( [7\pi/6,
                5\pi/3])) - \frac{1}{2}I_{3} \pi^{I_{1}, I_{3}}([7\pi/6,
                5\pi/3])) \\
  &\leq I_{1} - 0.499 I_{3} < -0.48 I_{3} 
\end{align*} 
because the condition in the Lemma implies $I_{1} < 1/\sqrt{5500}
I_{3}$. This completes the proof.
\end{proof}

We note that the estimate in Lemma \ref{smallI1} is not optimal. To apply Theorem A and Theorem B in \cite{huang2015integral}, we use uniform Lyapunov and anti-Lyapunov constants. This condition can be significantly relaxed by using a more refined version of the level set method developed in \cite{li2016systematic}.

Theorem \ref{thm31} then follows from Lemma \ref{smallI1}.

\begin{proof}[Proof of Theorem \ref{thm31}]
  Let
$$
  \langle I_{3} \rangle := \int_{\mathbb{R}^{2}}  I_{3} u_{1}(I_{1})u_{3}(I_{3})
  \mathrm{d}I_{1} \mathrm{d}I_3 = \sqrt[3]{\frac{3T_{3}}{4}}\frac{\Gamma(5/3)}{\Gamma(4/3)} \,.
  $$
  It is easy to see that $\langle I_{3} \rangle$ increases with $T_{3}$.
  Since the marginal distribution of $I_{1}$ is explicitly given in
  Lemma \ref{invI1I3}, there exists a constant $I_{1}^{*}$ depending only on $T_{1}$ such that
$$
  \int_{\{ I_{1} > I_{1}^{*} \}}  I_{3} u_{1}(I_{1})u_{3}(I_{3})
  \mathrm{d}I_{1} \mathrm{d} I_{3}= \langle I_{3} \rangle \int_{I_{1}^{*}}^{\infty}
  u_{1}( I_{1}) \mathrm{d} I_{1} < 0.1 \langle I_{3} \rangle\,.
  $$
 Let $C_{I_{1}^{*}} := \max\{ 500 \gamma, 11 I_{1}^{2}/\gamma\}$. By making $T_{3}$ large enough, we obtain
$$
  \int_{0}^{C_{I_{1}^{*}}} I_{3} u_{3}(I_{3}) \mathrm{d} I_{3} < 0.1 \langle I_{3} \rangle \,.
$$
Therefore, let
$$
A = \{ (I_{1}, I_{3}) \,|\, I_{1} \leq I_{1}^{*}
\mbox{ and } I_{3} \geq C_{I_{1}^{*}}  \}
$$
and $\bar{A}$ be the complement of $A$. We have
\begin{align*}
  &\int_{\bar{A}}  I_{3} u_{1}(I_{1})u_{3}(I_{3})
  \mathrm{d}I_{1} \mathrm{d} I_{3} \\
  \leq& \int_{\{ I_{1} \geq I_{1}^{*} \}}  I_{3} u_{1}(I_{1})u_{3}(I_{3})
  \mathrm{d}I_{1} \mathrm{d} I_{3} + \int_{\{ I_{3} \leq  C_{I_{1}^{*}} \}}  I_{3} u_{1}(I_{1})u_{3}(I_{3})
    \mathrm{d}I_{1} \mathrm{d} I_{3}\\
  \leq & 0.1 \langle I_{3} \rangle + \int_{0}^{\infty} u_{1}(I_{1}) \mathrm{d} I_{1}
         \int_{0}^{C_{I_{1}^{*}}} I_{3} u_{3}(I_{3}) \mathrm{d} I_{3}
  \\
  \leq & \frac{1}{5} \langle I_{3} \rangle \,. 
\end{align*}
In other words, we have
$$
  \int_{A} I_{3} u_{1}(I_{1})u_{3}(I_{3}) \mathrm{d}I_{1} \mathrm{d}
  I_{3} \geq \frac{4}{5}\langle I_{3} \rangle\,.
$$

  Since the dynamics of $(I_{1}, I_{3})$ are independent of $(\theta_{1}, \theta_{3})$, we have
\begin{align*}
 & \int_{\Omega^{b}} (I_{1} \sin \theta_{1} + I_{3} \sin \theta_{3}) \pi^{b}( \mathrm{d}
  I_{1} \mathrm{d} I_{3} \mathrm{d}\theta_{1} \mathrm{d}\theta_{3} )
  \\
  = &  \int_{\mathbb{R}^{2}_{+}} \left (
      \int_{\mathbb{T}^{2}}I_{1}\sin \theta_{1} + I_{3}
         \sin \theta_{3} \pi^{I_{1}, I_{3}}( \mathrm{d} \theta_{1}
         \mathrm{d} \theta_{3})\right ) u_{1}(I_{1})u_{3}(I_{3})
      \mathrm{d} I_{1} \mathrm{d} I_{3}\\
 =& \int_{A} \mathbb{E}_{\pi^{I_{1}, I_{3}}}[ I_{1} \sin \theta_{1} +
    I_{3} \sin \theta_{3}] u_{1}(I_{1})u_{3}(I_{3})
      \mathrm{d} I_{1} \mathrm{d} I_{3} + \int_{\bar{A}} \mathbb{E}_{\pi^{I_{1}, I_{3}}}[ I_{1} \sin \theta_{1} +
    I_{3} \sin \theta_{3}] u_{1}(I_{1})u_{3}(I_{3})
      \mathrm{d} I_{1} \mathrm{d} I_{3} \,.
\end{align*}
Note that every $(I_{1}, I_{3}) \in A$ satisfies $I_{1}^{2}/I_{3} <
\gamma/11$ and $I_{3}/\gamma > 500$. By Lemma \ref{smallI1}, we have
$$
  \int_{A} \mathbb{E}_{\pi^{I_{1}, I_{3}}}[ I_{1} \sin \theta_{1} +
    I_{3} \sin \theta_{3}] u_{1}(I_{1})u_{3}(I_{3})
      \mathrm{d} I_{1} \mathrm{d} I_{3} \leq -\int_{A} 0.48 I_{3} u_{1}(I_{1})u_{3}(I_{3})
      \mathrm{d} I_{1} \mathrm{d} I_{3} \,.
      $$
      In addition, we have
      
$$
  \int_{\bar{A}} \mathbb{E}_{\pi^{I_{1}, I_{3}}}[ I_{1} \sin \theta_{1} +
    I_{3} \sin \theta_{3}] u_{1}(I_{1})u_{3}(I_{3})
      \mathrm{d} I_{1} \mathrm{d} I_{3}  \leq \int_{\bar{A}} I_{3} u_{1}(I_{1})u_{3}(I_{3})
      \mathrm{d} I_{1} \mathrm{d} I_{3}  + \int_{\Omega^{b}} I_{1}u_{1}(I_{1})u_{3}(I_{3})
      \mathrm{d} I_{1} \mathrm{d} I_{3} \,.
$$
The calculation above gives
$$
  \int_{\bar{A}} I_{3} u_{1}(I_{1})u_{3}(I_{3})
      \mathrm{d} I_{1} \mathrm{d} I_{3}  \leq \frac{1}{5} \langle
      I_{3} \rangle \,.
$$
In addition
$$
  \int_{\Omega^{b}} I_{1} u_{1}(I_{1})u_{3}(I_{3})
      \mathrm{d} I_{1} \mathrm{d} I_{3} = \int_{0}^{\infty} I_{1}
      u_{1}( \mathrm{d} I_{1})  =
      \sqrt[3]{\frac{3T_{1}}{4}}\frac{\Gamma(5/3)}{\Gamma(4/3)}  := \langle I_{1} \rangle  \,.
      $$
Therefore, one can make $T_{3}$ large enough such that $\langle I_{1}
\rangle \leq 0.03 \langle I_{3} \rangle$. This gives
$$
  \int_{\Omega^{b}} (I_{1} \sin \theta_{1} + I_{3} \sin \theta_{3}) \pi^{b}( \mathrm{d}
  I_{1} \mathrm{d} I_{3} \mathrm{d}\theta_{1} \mathrm{d}\theta_{3} )
  \leq 0.2 \langle I_{3} \rangle + 0.03 \langle I_{3} \rangle - 0.48
  \langle I_{3} \rangle = -\frac{1}{4}\langle I_{3} \rangle \,.
$$
This completes the proof. 
\end{proof}

We now calculate the variance of $h^{b}$ with respect to $\pi^{b}$ by proving the following theorem.

\begin{thm}
  \label{varh}
$$
  \mathrm{Var}_{\pi^b}[h] = \mathrm{Var}_{\pi^{b}}[ I_1 \sin \theta_1 + I_3 \sin \theta_3] < \frac{1}{10}T_{3}^{2/3}
  $$
  provided that $\gamma$ and $T_{1}$ are sufficiently small. 
\end{thm}
\begin{proof}
Let $\pi^{I}$ denote the probability measure with probability density
function $u_{1}(I_{1})u_{3}(I_{3})$. Note that the conditional
distribution of $(\theta_{1}, \theta_{3})$ follows from equation
\eqref{theta13} for each given pair $(I_{1}, I_{3})$. By the law of total variance, we have
\begin{align*}
  \mathrm{Var}_{\pi^{b}}[ I_{3} \sin \theta_{3}]&
                                                   =\mathbb{E}_{\pi^{2}}[
                                                   \mathrm{Var}_{\pi^{I_{1},
                                                  I_{3}}}[I_{3}
                                                   \sin \theta_{3}
                                                   \,|\, I_{1}, I_{3}]]
  + \mathrm{Var}_{\pi^{I}}[ \mathbb{E}_{\pi^{I_{1}, I_{3}}}[ I_{3}\sin \theta_{3} \,|\,
                                                  I_{1}, I_{3}]] \\
  &:= V_{1} + V_{2} \,.
\end{align*}

It is easy to see that
$$
  V_{2} \leq \mathrm{Var}_{\pi^{I}}[I_{3}] = \left ( \frac{3
      T_{3}}{4}\right )^{2/3} \left ( \frac{\Gamma(2)}{\Gamma(4/3)} -
    \left ( \frac{\Gamma(5/3)}{\Gamma(4/3)}\right )^{2}\right )
  \approx 0.08078 T_{3}^{2/3}
  $$
  by the variance formula of the generalized Gamma distribution.

When $I_{3} \gg \gamma$, the same calculation in the proof of Lemma
\ref{vartheta} gives 
$$
  \mathrm{Var}_{\pi^{I_1, I_3}}[\sin \theta_{3}] = O(\gamma I_{3}^{-1} ) \,.
  $$
In addition, the variance of an angle is at most $\pi^{2}$. When
$I_{3} \geq \gamma^{1/2}$, we use the estimate $O(\gamma I_{3}^{-1} )
$. Otherwise, we bound the variance by the worst-case scenario. This gives
$$
  \mathrm{Var}_{\pi^{I_{1}, I_{3}}}[I_{3} \sin \theta_{3}] = I_{3}^{2}O(\gamma
  I_{3}^{-1})  = I_{3}O(\gamma)
$$
if $I_{3} \geq \gamma^{1/2}$ and
$$
  \mathrm{Var}_{\pi^{I_{1}, I_{3}}}[I_{3} \sin \theta_{3}] =  I_{3}^{2}\pi^{2}
$$
otherwise. 

Note that the integral of $I_{3}^{2}u_{3}(I_{3})$ up to $\gamma^{1/2}$ is of order
$O(\gamma^{3})$. We have
\begin{align*}
\mathbb{E}_{\pi^{I}}[\mathrm{Var}_{\pi^{I_{1},I_{3}}}[I_{3}\sin
  \theta_{3}\,|\, I_{1}, I_{3}]] &\leq
                                   \mathbb{E}_{\pi^{I}}[I_{3}]O(\gamma)
                                   + \mathbb{E}_{\pi^{I}}[ \pi^{2}I_{3}^{2}
                                   \mathbf{1}_{\{I_{3} \leq
                                   \gamma^{1/2}\}}]\\
  &= O(\gamma)T_{3}^{1/3} + O(\gamma^{3})
  \end{align*}
because the expectation of $I_{3}$ with respect to $u_{3}$ is
$O(T_{3}^{1/3})$. Therefore, we have
$$
  \mathrm{Var}_{\pi^{b}}[I_{3} \sin \theta_{3}] \leq 0.08078
  T_{3}^{2/3} + O(\gamma^{3}) + O(\gamma)T_{3}^{1/3} \,.
  $$

In addition we have  
\begin{align*}
  \mathrm{Var}_{\pi^{b}}[I_{1} \sin \theta_{1}] &=
 \mathbb{E}_{\pi^{I}} [\mathrm{Var}_{\pi^{I_{1}, I_{3}}}[I_{1}\sin
 \theta_{1} \,|\, I_{1}, I_{3}] ] + \mathrm{Var}_{\pi^{I}}[
 \mathbb{E}_{\pi^{I_{1}, I_{3}}}[I_{1} \sin \theta_{1} \,|\, I_{1},
                                                  I_{3}]] \\
  &:= V_{1}' + V_{2}' \,.
 \end{align*}
 A similar estimate holds for $V_{2}'$:
 
$$
  V'_{2} \leq \mathrm{Var}_{\pi^{2}}[I_{1}] \approx 0.08078 T_{1}^{2/3} \,.
$$
Consider the worst-case bound for the angular variance term $V_{1}'$, we have
$$
  V_{1}' \leq \mathbb{E}_{\pi^{2}}[I_{1}^{2}\pi^{2}] = O(T_{1}^{2/3}) \,.
  $$
  Hence
  
$$
  \mathrm{Var}_{\pi^{b}}[I_{1} \sin \theta_{1}]  = O(T_{1}^{2/3}) \,.
$$

  Finally, note that the covariance $\mathrm{COV}( I_{1} \sin
  \theta_{1}, I_{3}\sin \theta_{3})$ is bounded by
$$
  \sqrt{ \mathrm{Var}_{\pi^{b}}[I_{1} \sin \theta_{1}]
    \mathrm{Var}_{\pi^{b}}[I_{3} \sin \theta_{3}] } =
  \sqrt{O(T_{1}^{2/3}) \left (O(T_{3}^{2/3} )+ O(\gamma^{3}) +
      O(\gamma)T_{3}^{1/3}  \right )} \,.
$$

Recall that
$$
\mathrm{Var}_{\pi^b}[I_1\sin \theta_1 + I_3 \sin \theta_3] = \mathrm{Var}_{\pi^b}[I_1\sin \theta_1] +  \mathrm{Var}_{\pi^b}[I_3\sin \theta_3]  + 2 \mathrm{COV}(I_1 \sin \theta_a, I_3 \sin\theta_3) \,.
$$
Combining the estimate of all three terms in the equation above, when $\gamma \ll T_3^{2/9}$ and $T_{1} \ll T_3$ , the leading term
of $\mathrm{Var}_{\pi^{b}}[ I_1\sin \theta_1 + I_3 \sin \theta_3]$ is the variance of $I_{3}
\theta_{3}$. Thus, by making $\gamma$ and $T_{1}$ sufficiently small, we
have
$$
  \mathrm{Var}_{\pi^{b}}[ h] \leq \frac{1}{10}T_{3}^{2/3} \,.
$$
This completes the proof.   
\end{proof}

\subsection{Results for the reduced system Part IV: Proof of
  Assumption (L)}

Finally, we will prove assumption {\bf (L)} for suitable choices of parameters. 

\begin{thm}
  \label{pfL}
For any $\gamma >0, T_1 > 0$, if $T_3$ is sufficiently large, there exists a strictly positive pre-factor function $Q$ such
that
$$
 \left \{ 
\begin{array}[tb]{ll}
 \mathcal{L}^{b}Q + h( \mathbf{x}^{b}) Q \leq p(\mathbf{x}^{b})  &
                                                                   \mathbf{x}^{b}
                                                                   \in
                                                                   \Omega^{M}\\
  Q = q( \mathbf{x}_{b}) & I_{1} + I_{3} = M
\end{array}
\right .
$$
for all $C^{2}$ functions $p( \mathbf{x}^{b}) <0$ and $q( \mathbf{x}^{b})>0$, where
$\Omega^{M} = \{ I_{1} + I_{3} \leq M \}$.
\end{thm}
\begin{proof}
The geometric ergodicity of $\mathbf{x}^{b}_{t}$ was proved in Theorem \ref{xbergodicity}
with respect to the Lyapunov function $\tilde{W} = I_{1} + I_{3}$. Since
the marginal distributions for $I_{1}$ and $I_{3}$ are explicitly
given in Lemma \ref{invI1I3}, it is easy to check $\pi^{b}( \tilde{W}^{2})$ is also
finite. Hence the first assumption of Theorem \ref{MET} is satisfied. 

Let $\bar{h}_M( \mathbf{x}^{b}) = c (h_{M}( \mathbf{x}^{b}) -
\pi^{b}(h_{M}))$ 
 for a constant $c$ that makes $| \bar{h}_M| < 1$. Note that $T_3$ is sufficiently large, from Theorem \ref{thm31}, we have $\pi^{b}(h_{M}) < 0$ for all sufficiently large $M$.  In addition,
the asymptotic variance of $\bar{h}_M$ is finite due to the
geometric ergodicity of $\mathbf{x}^{b}_{t}$ and the fact that $\pi^{b}( \tilde{W}^{2})
< \infty$ (Theorem 17.5.3 of
\cite{meyn2012markov}). Therefore, the second assumption of Theorem
\ref{MET} is also satisfied. And the multiplicative ergodic theorem holds for system \eqref{broken} with respect to a Lyapunov function $\tilde{W}$.

As a consequence, by Theorem \ref{thm28}, a strictly positive solution to the Feynman-Kac equation exists.
\end{proof}

\section{Proof of the main theorem}
\label{pfmainthm}

\subsection{Accessibility of the small set}
It remains to verify the minorization condition {\bf (A1')} for system
\eqref{bigsystem}. The main difficulty is that there is no noise at the $I_{2}$
term. Hence, we need to use the remaining four variables to move $I_{2}$ to
the desired place. The proof is divided into two steps.

\begin{lem}
  \label{access}
  For any compact set $C\subset \Omega$ satisfying
  $$
 \inf\{ I_2 \,|\, (I_{1}, I_{2}, I_{3}, \theta_{1}, \theta_{3}) \in C
  \} > 0 
  $$
   and for any small $\delta > 0$, we have
$$
  P^{1}(\mathbf{x}_{0}, \mathcal{B}_{\delta}(\mathbf{x}_{1}) ) > 0
  $$
  uniformly for all $\mathbf{x}_{0}, \mathbf{x}_{1} \in C$, where
  $P^{t}( \mathbf{x}, \cdot)$ is the transition
  kernel of equation \eqref{bigsystem} for $t \geq 0$ and $\mathbf{x}
  \in \Omega$, and $\mathcal{B}_{\delta}( \mathbf{x}_{1})$
  is the open ball of radius $\delta$ centered at $\mathbf{x}_{1}$.
\end{lem}
\begin{proof}
  Denote by $F$ and $\Sigma$  the drift vector fields and the diffusion coefficient matrix in equation \eqref{bigsystem}, respectively. Note that
  $\Sigma$ is a $5\times 4$ matrix. The goal is to find a $C^{1}$
  control function $\mathcal U(t)$ taking value in $\mathbb{R}^{4}$ solving the control problem
  
  \begin{equation}
    \label{accesscontrol}
  \frac{\mathrm{d} \mathbf{q}}{ \mathrm{d} t} = F( \mathbf{q}) +
  \Sigma \frac{ \mathrm{d}\mathcal U}{ \mathrm{d}t} \quad , \quad \mathbf{q}(0)
  = \mathbf{x}_{0},\quad \mathbf{q}(1) = \mathbf{x}_{1} \,. 
\end{equation}
Then, the lemma follows from the fact that the probability that the Wiener
measure of any $\epsilon$-tube
$$
  \left\{ W(t), t \in [0, 1] \,|\, \sup_{0 \leq s\leq 1} \| W(s) - U(s)\|
  \leq \epsilon \right\}
  $$
 is strictly positive, where $\| \cdot \|$ is the supremum norm in
 $C([0, 1])$.

  \medskip

  We denote
  
$$
  \mathbf{x}_{0} = (I_{1}^{0}, I_{2}^{0}, I_{3}^{0}, \theta_{1}^{0}, \theta_{2}^{0})
$$
and
$$
  \mathbf{x}_{1} = (I_{1}^{1}, I_{2}^{1}, I_{3}^{1}, \theta_{1}^{1},
  \theta_{2}^{1})=(I_{1}(1), I_{2}(1), I_{3}(1), \theta_{1}(1),
  \theta_{2}(1)) \,.
$$
Further we denote the $(I_{1}, I_{3}, \theta_{1},
\theta_{3})$-component of $\mathbf{x}_{i}$ by $\mathbf{y}_{i}$ for $i
= 0, 1$. Without loss of generality, assume $I_{2}^{0} \leq
I_{2}^{1}$. Observe that $I_{2}^{0} > 0$ because $\mathbf{x}_0 \in C$. In order
to move the $I_{2}$ term from $I_{2}^{0}$ to $I_{2}^{1}$, we should have
$$
  I_{2}^{1} = I_{2}^{0} \exp \left ( \int_{0}^{1} -2(I_{1} \sin
    \theta_{1} + I_{3} \sin \theta_{3}) \mathrm{d}t\right ) :=
  I_{2}^{0} \exp \left ( \int_{0}^{1} -2h( \mathbf{y}(t)) \mathrm{d}t\right )\,,
$$
where $\mathbf{y}(t)$ is the $(I_{1}, I_{3}, \theta_{1},
\theta_{3})$-component of $\mathbf{q}(t)$ 

Let $\xi = \log \frac{I_{2}^{1}}{I_{2}^{0}} \geq 0$. One can find two
points $\mathbf{z}^{0}$ and $\mathbf{z}^{1}$ in $\mathbb{R}^{2}_{+}
\times \mathbb{T}^{2}$ such that
$$
  -2h( \mathbf{z}^{0}) = e^{\xi} - 0.1 \quad , \quad -2h( \mathbf{z}^{1})
  = e^{\xi} + 0.1 \,.
$$
Since $C$ is compact and $I_{2}^{0}$ is uniformly bounded away from
zero, the choice of $\mathbf{z}^{0}$ and $\mathbf{z}^{1}$ can also be
uniformly bounded for each initial value $\mathbf{x}_0$ and the target $\mathbf{x}_1$ in $C$.

Now, let $0 < t_{\xi} < 1/2$ be a small parameter to be decided
later. For $\alpha = 0$ or $1$, we define two $C^{1}$ functions by
$$
  \mathbf{z}^{\alpha}(t) = \left \{ 
\begin{array}[tb]{ll}
 \mathbf{y}_{0} + \frac{t}{t_{\xi}}( \mathbf{z}^{\alpha} - \mathbf{y}_{0})
  & 0 < t < 0.99 t_{\xi}  \\
  \zeta_1(t) & 0.99 t_{\xi} \leq t < t_{\xi} \\
  \mathbf{z}^{\alpha} & t_{\xi} \leq t < 1 - t_{\xi}\\
  \zeta_2(t) & 1 - t_{\xi} \leq t < 1 - 0.99
                                 t_{\xi}\\
  \mathbf{z}^{\alpha} + \frac{t - (1 - t_{\xi})}{t_{\xi}}(
  \mathbf{y}_{1} - \mathbf{z}^{\alpha}) & 1 - 0.99 t_{\xi} \leq t \leq 1
\end{array}
\right . \,,
$$
where $\zeta_1$ and $\zeta_2$ are two spline interpolation functions ensuring $C^1$ continuity. Since all relevant quantities are uniformly bounded, one can
then make $t_{\xi}$ small enough such that the integral of $-2h(
\mathbf{z}^{\alpha}(t))$ from $0$ to $t_{\xi}$ and from $1 - t_{\xi}$
to $1$ are small enough, hence giving
$$
  \int_{0}^{1} -2h( \mathbf{z}^{0}(t)) \mathrm{d}t  < \xi
$$
and
$$
  \int_{0}^{1} -2h( \mathbf{z}^{1}(t)) \mathrm{d}t  > \xi \,.
$$
Next, we interpolate 
$$
  \mathbf{z}^{\lambda}(t) = \mathbf{z}^{0}(t) + \lambda (
  \mathbf{z}^{1}(t) - \mathbf{z}^{0}(t))
  $$
for $\lambda \in (0, 1)$. It is easy to see that
$\mathbf{z}^{\lambda}(t)$ is a family of continuous functions with initial value $\mathbf{x}_{0}$ and terminal value $\mathbf{x}_{1}$. In addition
$$
  R( \lambda) := \int_{0}^{1} -2h( \mathbf{z}^{\lambda}(t)) \mathrm{d}t
$$
is continuous in $\lambda$ with $R(0) < \xi < R(1)$. Therefore, by the
mean value theorem, there must be a $\lambda^{*} \in (0, 1)$ such that
$R(\lambda^{*}) = \xi$. Since the noise term is non-degenerate in
$I_{1}, I_{3}, \theta_{1}, \theta_{3}$, there exists a $C^{1}$
function $\mathcal U(t)$  such that the $I_{1}, I_{3}, \theta_{1},
\theta_{3}$-components of $\mathbf{q}(t)$ identical to that of
$z^{\lambda^{*}}(t)$. The function $\mathcal U(t)$ is uniformly bounded for any
pair of boundary points $\mathbf{x}_{0}$ and $\mathbf{x}_{1}$, because
$\mathbf{z}^{\lambda^*}(t)$ is uniformly bounded.

Finally, since
$$
  I_{2}^{1} = I_{2}^{0} \exp \left ( \int_{0}^{1} -2h(
    \mathbf{z}^{\lambda^{*}}(t)) \mathrm{d}t\right) \,,
$$
the function $\mathcal U(t)$ solves the control problem in equation
\eqref{accesscontrol}. This completes the proof.

\end{proof}

It remains to show that the transition kernel of system
\eqref{bigsystem} is jointly continuous. As discussed in Section 2, this is equivalent to checking
the H\"ormander's condition.

\begin{lem}
The transition kernel of system \eqref{bigsystem} is jointly continuous in both $t$ and $\mathbf{x}$.
\end{lem}
\begin{proof}
First, we need to convert system \eqref{bigsystem} into a Stratonovich stochastic
differential equation of the form
$$
  \mathrm{d} \mathbf{x}_{t} = {\bm v}_{0}( \mathbf{x}_{t}) \mathrm{d}t +
  \sum_{i = 1}^{4} {\bm v}_{i}( \mathbf{x}_{t}) \circ \mathrm{d}B^{(i)}_{t}
  \,.
$$
This yields
$$
  {\bm v}_{0} = 
\begin{bmatrix}
  [2 I_{1} I_{2} (\sin \theta_{1} -
  \gamma) + \gamma(T_{1} - I_{1}^{3}) ] + \frac{1}{8}\gamma T_{1}\\
   - 2 I_{2} (I_{1} \sin \theta_{1} + I_{3} \sin
  \theta_{3}) \\
  [2 I_{2}I_{3}(\sin \theta_{3} - \gamma) +
  \gamma(T_{3} - I_{3}^{3})] + \frac{1}{8}\gamma T_{3}\\
   I_{2}(1 + 2 \cos \theta_{1}) - I_{1} (1 +
  2 \cos \theta_{1}) - 2 I_{3} \cos \theta_{3} + \frac{1}{2}\gamma
  T_{1} g_{\theta_{1}} g\\
 I_{2} (1 + 2 \cos \theta_{3}) - I_{3} (1 + 2 \cos
  \theta_{3}) - 2 I_{1}\cos \theta_{1} + \frac{1}{2} \gamma T_{3}
  g_{\theta_{3}} g \,,
\end{bmatrix}
$$
$$
  {\bm v}_{1} = 
\begin{bmatrix}
\sqrt{\gamma T_{1} I_{1}/2}\\0\\0\\0\\0
\end{bmatrix}
, \quad
  {\bm v}_{2} = 
\begin{bmatrix}
0\\0\\ \sqrt{\gamma T_{3} I_{3}/2}\\0\\0
\end{bmatrix}
,\quad
  {\bm v}_{3} = 
\begin{bmatrix}
0\\0\\0\\ \sqrt{\gamma T_{1}} g(I_{2}, \theta_{1})\\0
\end{bmatrix}
, \quad
  {\bm v}_{4} = 
\begin{bmatrix}
0\\0\\0\\0\\ \sqrt{\gamma T_{3}} g(I_{2}, \theta_{3}) \,.
\end{bmatrix}
$$

Since ${\bm v}_{1}, {\bm v}_{2}, {\bm v}_{3}, {\bm v}_{4}$ already have nonzero components in entries $1, 3, 4$, and $5$, the goal is to use  the Lie brackets to create new
vector fields that have linearly independent second entries. After that, the lemma follows immediately because the H\"ormander's condition is satisfied (Theorem \ref{hormander}).

To do so, it is sufficient to calculate 
$$
  {\bm v}_{5} = [{\bm v}_{0} , {\bm v}_{1}] = J_{{\bm v}_{0}} {\bm v}_{1} - J_{{\bm v}_{1}} {\bm v}_{0}
$$
and
$$
  {\bm v}_{6} = [{\bm v}_{5}, {\bm v}_{3}] =  J_{{\bm v}_{5}} {\bm v}_{3} - J_{{\bm v}_{3}} {\bm v}_{5} \,,
  $$
where $J_{{\bm v}_i}$ is the Jacobian matrix of the vector field ${\bm v}_i$.

  A straightforward computation yields
$$
  {\bm v}_{5} = 
\begin{bmatrix}
\omega( \mathbf{x}) \\ - 2 I_2\sqrt{\gamma T_{1} I_1/2} \sin \theta_{1}\\0\\-\sqrt{\gamma T_{1} I_{1}/2}(1 + 2 \cos
\theta_{1})\\-2 \sqrt{\gamma T_{1} I_{1}/2} \cos \theta_{3}
\end{bmatrix}
$$
with
$$
  \omega( \mathbf{x}) = \sqrt{\gamma T_{1} I_{1}/2}[2 I_{2} (\sin
  \theta_{1} - \gamma) - 3 \gamma I_{1}^{2}] - \frac{1}{2}\sqrt{\gamma
  T_{1}/2} I_{1}^{-1/2} [2 I_{1}I_{2}(\sin \theta_{1} - \gamma) +
\gamma (T_{1} - I_{1}^{3}) + \frac{1}{8} \gamma T_{1}] \,.
$$
Comparing with ${\bm v}_{i}$, for  $i=1,2,3,4$, we have that ${\bm v}_{5}$ now has nonzero second
entry. However, the second component of ${\bm v}_{5}$ becomes zero
when $\sin \theta_{1} = 0$. Vector fields $[{\bm v}_{0}, {\bm v}_{2}]$ and
$[{\bm v}_{0}, {\bm v}_{3}]$ have similar problems. Therefore, we need to
calculate ${\bm v}_{6}$ using ${\bm v}_{3}$ and ${\bm v}_{5}$. The second entry of ${\bm v}_{6}$ is
$$
  - \sqrt{2}\gamma T_{1} I_{2} I_{1}^{1/2} g(I_{2}, \theta_{1}) \cos
  \theta_{1} \,,
$$
  which is linearly independent of the $2$nd entry of ${\bm v}_{5}$. Hence,
  the vectors $\{ {\bm v}_{1}, \cdots, {\bm v}_{6} \}$ span $\mathbb{R}^{5}$ for all
  $\mathbf{x}$ with $I_{1}, I_{2}, I_{3} \neq 0$. This completes the
  proof. 
\end{proof}

\subsection{Proof of the main theorem}
\begin{proof}[Proof of Theorem 1]
  Now let 
  
$$
  \mathcal{V} = V + W \,,
$$
where
$$
V = (I_1 + I_2 + I_3)^{\beta_0} + I_1^{-\beta_1} f(\theta_1) + I_3^{-\beta_1}f(\theta_3)
$$
and
$$
W = \psi(I_2) U \,.
$$

Then Theorem 4.3 implies
$$
  \mathcal{L}V \leq -c V^{1-1/\beta_{0}}
  $$
  for all $\mathbf{x}$ with $V( \mathbf{x}) > C_{V}$. In addition, since the set $\{ \mathbf{x} \in \Omega \,|\, V( \mathbf{x} ) \leq C_V\}$ is compact, it is easy to see that $\mathcal{L} V$ is bounded in $\{ \mathbf{x} \in \Omega \,|\, V( \mathbf{x} ) \leq C_V\}$. For simplicity, we can relax the estimate and obtain
  $$
  \mathcal{L}V \leq -c V^{1-1/\beta_{0}} + M_V
  $$
for all $\mathbf{x} \in \Omega$, where $M_V < \infty$ is a constant.
  
 For the Lyapunov function $W$, Theorem 5.1 implies
$$
  \mathcal{L} W \leq - \delta W
$$  
for all $\mathbf{x}$ with $I_2$-component less than $\zeta\leq 1$. If $I_2 > 1$, then $W$ and $\mathcal{L} W$ are both zero. Hence we have 
$$
\mathcal{L}(V+W) \leq -c V^{1 - 1/\beta_0} - \delta W + \mbox{constant}
$$
outside the region $\zeta \leq I_2 \leq 1$.

It remains to bound $\mathcal{L}(V+W)$ in the strip $\zeta \leq I_2 \leq 1$. Lemma \ref{Wupbd} gives
$$
\mathcal{L}W \leq C_\zeta [1 + (I_1 + I_3)^{2} ] \,.
$$
Therefore, if $I_1 + I_3$ is bounded by a constant $C$, then we have $\mathcal{L}W \leq C_\zeta(1 +C)^{2}$. Otherwise, since $C_\zeta [1 + (I_1 + I_3)^{2} ]$ is unbounded as $I_1 + I_3 \rightarrow \infty$, $\mathcal{L}W$ in this strip needs to be offset by $\mathcal{L}V$. When $\zeta \leq I_2 \leq 1$ and $I_1 + I_3$ is sufficiently large, the dominant term in $\mathcal{L}V$ is
\begin{align*}
    &\mathcal{L}V \leq\\
    & \gamma \beta_{0}(I_{1} + I_{2} + I_{3})^{\beta_{0} - 1} \left (
     -4 I_{2}( I_{1} + I_{3}) + (T_{1} - I_{1}^{3}) + (T_{3} -
     I_{3}^{3}) + \frac{T_{1}I_{1} + T_{3}I_{3}}{4(I_{1} + I_{2} +
     I_{3})}(\beta_{0} - 1) 
     \right ) \\
     & \leq - \frac{1}{4}\gamma \beta_0 (I_{1} + I_{2} + I_{3})^{\beta_{0} - 1} (I_1 + I_3)^3
\end{align*}
for all values of $I_1 + I_3$ that are large enough such that
$$
-4 \zeta (I_1 + I_3) + (T_1 + T_3) + \frac{T_{1}I_{1} + T_{3}I_{3}}{4(I_{1} + I_{2} +
     I_{3})}(\beta_{0} - 1) \leq 0 \,.
$$
(Recall that $(I_1 + I_3)^3 \leq 4 (I_1^3 + I_3^3)$.) Therefore, for all large $I_1 + I_3$, we have
$$
\mathcal{L}(V + W) \leq -c_{VW}(I_{1} + I_{2} + I_{3})^{\beta_{0} - 1} (I_1 + I_3)^3 
$$
for a small constant $c_{VW} > 0$. Recall that $W = (I_1 + I_3)I_2^{-\alpha}$ for all sufficiently large $I_1 + I_3$. Thus, the bound $-c_{VW}(I_{1} + I_{2} + I_{3})^{\beta_{0} - 1} (I_1 + I_3)^3$ dominates both $V$ and $W$. In other words. we have
$$
\mathcal{L}(V+W) \leq - (V + W)
$$
in the strip for all sufficiently large values of $(I_1 + I_3)$.

Therefore, by further relaxing the constant $M_V$ to a finite constant $M_{\mathcal{V}}$, we have
$$
\mathcal{L}(V+W) \leq -c V^{1 - 1/\beta_0} - \delta W  + M_{\mathcal{V}} 
$$
for all $\mathbf{x} \in \Omega$. Thus, there exists a constant $c_1 > 0$ that makes
$$
  \mathcal{L} \mathcal{V} \leq - c (V^{1 - 1/\beta_{0}} + W^{1 -
    1/\beta_{0}} ) + M_{\mathcal{V}}  \leq -c_{1} \mathcal{V}^{1 - 1/\beta_{0}} + M_{\mathcal{V}} 
$$
for all $\mathbf{x} \in \Omega$, where the last
inequality follows follows from norm equivalence on finite-dimensional spaces. Hence assumption {\bf (A2P)} holds. 

In addition, all sublevel sets $\{\mathcal{V} \leq C\}$ satisfy the
condition of Lemma \eqref{access}. Hence, assumption {\bf (A1')} is
satisfied. This means that assumption {\bf (A1)} also holds for all sublevel sets of $\mathcal{V}$. The proof is thus complete by Theorem \ref{hairernote41}.

\end{proof}

\section{Discussion and future work}
In this paper, we connect a 3-mode oscillator chain, embedded in the nonlinear Schrödinger equation, to damping and heat baths at different temperatures. We prove that if the temperatures at the two ends are significantly different, then the 3-mode chain admits a unique nonequilibrium steady state (NESS). In addition, the rate of convergence to the NESS is at least polynomial. Compared to
anharmonic chains that model heat conduction problems, the 3-mode chain
is significantly more challenging because the energy can only be
transferred in a desired way when the phases are properly aligned. Therefore, one must
show that when facing the ``overheating'' or ``overcooling'' problems,
the phase angles are properly aligned with high probability. Then we
use a novel Feynman-Kac-Lyapunov function method to construct a
Lyapunov function. The proof is completed by combining two Lyapunov
functions, each representing the ``overheating'' and ``overcooling''
scenarios.

As recalled in the introduction,  recently it has been proved  that the wave kinetic equation can be rigorously derived from a
dispersive equation. It is also known that the wave kinetic equation admits formal stationary solutions corresponding to nonequilibrium energy cascades, called the Kolmogorov–Zakharov (KZ) spectra.
Therefore, a mathematical justification of energy cascade system
embedded in dispersive equations is greatly needed. To the best of our
knowledge, this result is the first of its kind. It opens the door to
further investigations. For example, after establishing the stochastic stability of the NESS, one can study the energy distribution and flux of the NESS, and compare it with the Kolmogorov–Zakharov (KZ) spectra. Our next step is to write a follow-up paper demonstrating these
properties numerically.

To complete this challenging proof, we simplified the manner in which the 3-mode chain is connected to heat baths. As a result, the Gibbs measure is not invariant for this system, even when the two heat baths have the same temperature. Since the main challenge is always to
understand what stabilizes the system in the out-of-equilibrium
setting, this does not affect the main theme of the paper. 

For the sake of completeness, we plan to present a follow-up paper demonstrating that a more natural coupling to heat baths does not change the results obtained here. This analysis will rely on certain assumptions that can be verified numerically, such as the existence of an invariant probability measure for the modified (“broken”) system. Additionally, while we anticipate that the convergence to the nonequilibrium steady state (NESS) occurs at an exponential rate, current technical limitations only allow us to rigorously prove a polynomial rate of convergence. These issues will be addressed in detail in the subsequent work.

\appendix
\section{Proof of Theorem \eqref{thm28}}
\begin{proof}

The proof of Theorem \eqref{thm28} requires the eigenvalue
perturbation theory in \cite{kato2013perturbation}. the following definitions and notation are needed.

Let $T$ be a closed linear operator in space $X$. Let $R(\zeta) = (\zeta I - T)^{-1}$ be the resolvent of $T$. The {\it eigenprojection} of $T$ with respect to a semisimple isolated eigenvalue $\lambda$ is denoted by 
\begin{equation}
    \label{eigenproj}
    P= -\frac{1}{2\pi i} \oint R(\zeta) \mathrm{d}\zeta \,.
  \end{equation}

  Consider the perturbation $T(\kappa) = T + \kappa T_{1}$. Let
  $R_{\alpha}$ and $P_{\alpha}$ denote the resolvent of $T(\kappa)$
  and the projection with respect to $\lambda$, respectively. A semisimple isolated eigenvalue $\lambda$ is said to be {\it
    stable} under the perturbation given above if (i) the region of convergence
  in which $R_{\alpha}(\zeta) \rightarrow R(\zeta)$ as $\alpha
  \rightarrow 0$ in the
  generalized sense (see Section 4.2 of \cite{kato2013perturbation})
  contains a neighborhood of $\lambda$ with the exception of
  $\lambda$, and (ii) $\mathrm{Dim} \,P_{\alpha} \leq \mathrm{Dim} \,P$
  for all sufficiently small $\alpha$. 

  \medskip

  We have the following theorem regarding the perturbation of an isolated eigenvalue of $T$. 

\begin{thm}[Theorem 8.2.6 of \cite{kato2013perturbation}]
  \label{kato}
  Let $\lambda$ be an isolated, semisimple eigenvalue of a linear
  operator $T$ with the eigenprojection $P$, $\mbox{dim} \,P = m <
  \infty$. Denote $D(T_1)$ the domain of $T_1$. Let $\lambda$ be stable and let $PX \subset D( T_{1})$. Then the
  eigenvalues of $T(\kappa)$ admit asymptotic expansions
  
$$
  \mu_{j}(\kappa) = \lambda + \kappa \mu_{j}^{(1)} + o( \kappa) ,
  \quad j = 1, \cdots, m \,,
$$
where $\mu_{j}^{(1)}$ are the repeated eigenvalues of the operator $P
T_{1} P$ considered in the m-dimensional space $P X$.
\end{thm}

Going back to the proof of Theorem \eqref{thm28}, let $F = c_{0} \mathbf{1}_{D} - \pi(c_{0} \mathbf{1}_{D})$,  let $T =
\mathcal{L}$, the generator of $X_{t}$, let $T_{1}$ be the simple
operator with $T_{1} u = F u$. Since $X_{t}$ is geometrically ergodic with a unique invariant probability
measure $\pi$, we know that $0$ is an eigenvalue of
$\mathcal{L}$.  Since $c_{0} \mathbf{1}_{D}$ obviously satisfies
$c_{0}^{2} \mathbf{1}_{D} \leq V$, by Theorem 17.0.1 of
\cite{meyn2012markov} $c_{0} \mathbf{1}_{D}$ must have finite asymptotic
variance. Hence $F$ satisfies Assumption (b) of Theorem \ref{MET}. Theorem \ref{MET} shows that the maximal eigenvalue of
$\mathcal{L} + \alpha F$, denoted by $\Lambda( \alpha)$, is isolated
for all sufficiently small $|\alpha|$. Therefore, $0$ must be an
isolated maximal eigenvalue of $\mathcal{L}$.

Since $F$ is bounded, $T_{1}$ is a bounded linear
operator. Therefore, by Theorem 4.3.11 of \cite{kato2013perturbation},
the resolvent $R_{\alpha}$ is continuous under small
perturbations. By Theorem 4.3.16 of \cite{kato2013perturbation}, the
spectrum is upper semicontinuous with respect to small perturbation
(meaning that $\mathrm{Dim} \, P_{\alpha} = \mathrm{Dim} \, P$ for all small
$\alpha$). Hence $0$ is a stable eigenvalue by the definition of stable
eigenvalue given above.

Without loss of generality, let $P$ be the eigenprojection of $T$ with
respect to eigenvalue $0$. It remains to study the property of
$P$. Since $\pi$ is unique (up to constant multiplication), the
multiplicity of the eigenvalue $0$ must be $1$. Thus, $\mathrm{dim} \, P =
1$. Therefore, the eigenfunction of $T$ with respect to $0$ is a
constant function $1$, denoted by $v$. The eigenfunction of $T^*$ with
respect to $0$ is the probability density function of $\pi$, denoted
by $u$. In addition, denote the resolvent of $T^*$ by $R^*$, it follows that $R^*(\bar{\zeta}) = (R(\zeta))^*$. Therefore, since
$0$ is a real number, the eigenprojection of $T^*$ with respect to $0$
must be $P^*$. 

Note that $P$ is the projection into the set of constant functions, for any function $w$ we have $P w = C_w w$. Integrating $Pw$ with $u$, we have
$$
\langle u, Pw \rangle = C_w \langle u, v \rangle 
$$
because $\langle u, v \rangle = 1$. Therefore, we have
$$
C_w = \frac{\langle Pw, u \rangle}{\langle u, v \rangle} = \frac{\langle w, P^* u \rangle}{\langle u, v\rangle} = \langle w, u\rangle
$$
because $P^*$ projects $u$ into itself. This gives $P u = \pi (u)$. Therefore, we have $P T_1 = 0$ because $T_1 u = F u$ and $\pi(F) = 0$. Hence, by Theorem
\ref{kato}, we have
$$
  \Lambda(\alpha) = o(\alpha) \,.
  $$

  By Theorem \ref{MET}, we have
  
$$
  \mathbb{E}_{x}[ \exp \left ( \alpha \int_{0}^{t} F(X_{s})
    \mathrm{d}s\right )] \leq e^{t \Lambda (\alpha)}[
  \hat{f}_{\alpha}( x) + B_{0}|\alpha| W(x) e^{-b_{0}t}] \leq C e^{t
    \Lambda(\alpha)} \,.
  $$
  
  For $c_{0}(x) \mathbf{1}_{D}$, we have

  $$
  \mathbb{E}_{x}[ \exp \left ( \alpha \int_{0}^{t} c_{0}(X_{s}) \mathbf{1}_{D}
    \mathrm{d}s\right )] \leq e^{t \Lambda (\alpha)}[
  \hat{f}_{\alpha}( x) + B_{0}|\alpha| W(x) e^{-b_{0}t}] \leq C e^{t
    (\alpha \pi(c_{0} \mathbf{1}_{D}) + \Lambda(\alpha))} \,.
  $$
  
Since
$$
\lim_{\alpha \rightarrow0} \frac{\Lambda(\alpha)}{\alpha} = 0 \,,
$$
notice that $D$ is a bounded set, for all sufficiently small $\alpha >
0$, we have
$$
  \mathbb{E}_{x}[ \exp \left ( \alpha \int_{0}^{t} c_{0}(X_{s}) \mathbf{1}_{D}
    \mathrm{d}s\right )]\leq C_{M} e^{-b_{1} \alpha t}
$$
uniformly for some $b_{1} > 0$ and $C_{M} < \infty$. Therefore, let $c = \alpha c_{0}
\mathbf{1}_{D}$.

Now we are ready to move to $X_{t}$ on a bounded domain $D$. Let $\tau^{D}$ be the first exit time of $X_{t}$ from
$D$. For any $x \in D$ we have
$$
\mathbb{E}_{x}[ \mathbf{1}_{t < \tau^{D}} \exp \left ( \int_{0}^{t}
  c(X_{s}) \mathrm{d}s\right )] \leq \mathbb{E}_{x}[ \exp \left ( \alpha \int_{0}^{t} c(X_{s})
    \mathrm{d}s\right )]\leq C_{M} e^{-b_{1} \alpha t} \,.
$$
Let $t \rightarrow \infty$ and take the logarithm, by Theorem \ref{liu}, we have $\lambda_{1}(\mathcal{L} + \alpha c_0, D) \geq b_{1} \alpha > 0$, where $\lambda_{1}(
\mathcal{L} + \alpha c_{0}, D)$ is the principal eigenvalue defined in \cite{berestycki1994principal}. Thus, by theorem \ref{nirenberg},
for any function $\hat{f} \in L^{p}(D)$, equation
\begin{align*}
  & \mathcal{L} u(x) + \alpha c_{0}(x) u(x)= \hat{f}(x) , \quad x \in D\\
  &u(x) = 0, \quad x \in \partial D
\end{align*}
admits a unique solution in $W^{2, d}_{loc}(D)$.

Let $\omega$ be a $C^{2}$ function in $D$ satisfying
$\omega = \psi$ on $\partial D$. Then equation \eqref{cauchySDE} can
be converted to
\begin{align*}
  & \mathcal{L} \tilde{u} + \alpha c_{0}  \tilde{u} = f(x) -  (\mathcal{L} + \alpha
    c_{0}) \omega, \quad  x \in D\\\nonumber
  & \tilde{u}(x_{0}) = 0, \quad  x_{0} \in \partial D \nonumber\,.
\end{align*}
for $\tilde{u} = u - \omega$. Since $f(x) -  (\mathcal{L} + \alpha
    c_{0}) \omega$ is both bounded and continuous, by Theorem
    \ref{nirenberg}, equation
\eqref{cauchySDE} also admits a unique bounded solution in $W^{2,
  p}_{loc}(D)$ for any $p > 1$. 
  
  Notice that the operator $\mathcal{L}$ actually has $C^3$ coefficient and the domain $D$ has $C^4$ boundary. By higher regularity  existence results of solution to elliptic equations (see for example Chapter 6, Theorem 5 of \cite{evans2022partial}), the solution $u$ belongs to $W^{4, p}(D)$ for any $p > 1$. Therefore, by general Sobolev inequality (see for example Chapter 5, Theorem 6 of \cite{evans2022partial}), the solution $u$ belongs $C^2(D)$. Therefore, by the stochastic representation in Theorem \ref{freidlinPDE}, the Feynman-Kac equation
\eqref{cauchySDE} admits a strictly positive solution $u(x)$.

\end{proof}

\bibliography{myref_arxiv}{}
\bibliographystyle{amsplain}
\end{document}